\documentclass[hidelinks,12pt]{amsart}

\usepackage{amsmath}
\usepackage{amsthm}
\usepackage{stmaryrd}
\usepackage{amsfonts,mathrsfs,amssymb}%,txfonts}
\usepackage{times}
\usepackage{fullpage}
\usepackage{verbatim}
\usepackage{tikz-cd}
\usepackage{enumerate}

\title{%Variation of Bergman Spaces II.  \\
Berndtsson-Lempert-Sz\H{o}ke Fields \\associated to \\Proper Holomorphic Families of Vector Bundles}
\author{Dror Varolin} 
\thanks{Mathematics Subject Classification (2020 Database): 32L05}
\newcommand{\noi}{\noindent}

\newcommand{\cO}{{\mathcal O}}
\newcommand{\cP}{{\mathcal P}}
\newcommand{\cQ}{{\mathcal Q}}

\newcommand{\cS}{{\mathcal S}}

\newcommand{\sC}{{\mathscr C}}

\newcommand{\sH}{{\mathscr H}}

\newcommand{\sL}{{\mathscr L}}

\newcommand{\fH}{{\mathfrak H}}
\newcommand{\fI}{{\mathfrak I}}

\newcommand{\fL}{{\mathfrak L}}

\newcommand{\fa}{{\mathfrak a}}
\newcommand{\fb}{{\mathfrak b}}

\newcommand{\fd}{{\mathfrak d}}

\newcommand{\ff}{{\mathfrak f}}
\newcommand{\fg}{{\mathfrak g}}
\newcommand{\fh}{{\mathfrak h}}
\newcommand{\fri}{{\mathfrak i}}

\newcommand{\fl}{{\mathfrak l}}

\newcommand{\fs}{{\mathfrak s}}

\newcommand{\vp}{\varphi}

\newcommand{\bbC}{{\mathbb C}}
\newcommand{\bbD}{{\mathbb D}}

\newcommand{\bbG}{{\mathbb G}}

\newcommand{\bbN}{{\mathbb N}}
\newcommand{\bbP}{{\mathbb P}}

\newcommand{\bbZ}{{\mathbb Z}}

\newcommand{\red}{\hfill $\diamond$}

\newcommand{\di}{\partial}
\newcommand{\dbar}{\bar \partial}

\newcommand{\re}{{\rm Re\ }}
\newcommand{\im}{{\rm Im\ }}

\newcommand{\ii}{\sqrt{\text{\rm -1}}}
\newcommand{\two}{\mathbf{I\!I}}

\newcommand{\emb}{\hookrightarrow}
\newcommand{\tensor}{\otimes}

\def\XXint#1#2#3{{\setbox0=\hbox{$#1{#2#3}{\int}$} 
\vcenter{\hbox{$#2#3$}}\kern-.5\wd0}}

\usepackage{hyperref}

 \usepackage[utf8]{inputenc}
        \usepackage[T1]{fontenc}
        \usepackage{lmodern}

        \usepackage[shortlabels]{enumitem}
                    \setlist[enumerate, 1]{a.}
                    \setlist[enumerate, 2]{i.}

\begin{document}
%\setlength{\textwidth}{24cm}
%\setlength{\textheight}{30cm}
%\maketitle

\theoremstyle{plain}
\newtheorem{m-thm}{\sc Theorem}
\newtheorem*{s-thm}{\sc Theorem}
\newtheorem{lem}{\sc Lemma}[section]
\newtheorem{thm}[lem]{\sc Theorem}
\newtheorem{prop}[lem]{\sc Proposition}
\newtheorem{cor}[lem]{\sc Corollary}

\theoremstyle{definition}
\newtheorem{conj}[lem]{\sc Conjecture}
\newtheorem{prob}[lem]{\sc Open Problem}
\newtheorem{defn}[lem]{\sc Definition}
\newtheorem{def-prop}[lem]{\sc Definition/Proposition}
\newtheorem*{s-defn}{\sc Definition}
\newtheorem{qn}{\sc Question}
\newtheorem*{s-qn}{\sc Question}
\newtheorem{ex}[lem]{\sc Example}
\newtheorem{rmk}[lem]{\sc Remark}
\newtheorem*{s-rmk}{\sc Remark}
\newtheorem*{s-rmks}{\sc Remarks}
\newtheorem{rmks}[lem]{\sc Remarks}
\newtheorem*{ack}{\sc Acknowledgment}
\newtheorem{exe}{\sc Exercise}

%\begin{abstract}
%\end{abstract}

\maketitle

\setcounter{tocdepth}1

\vskip .3in

%\begin{center}{\sc Abstract}\end{center}

%\noi {\small }

\begin{center}
{\it Dedicated to Bo Berndtsson and Laszlo Lempert}
\end{center}

\begin{abstract}
Drawing on work of Berndtsson and of Lempert and Sz\H{o}ke, we define a complex analytic structure for families of (possibly finite-dimensional) Hilbert spaces that might not fit together to form a holomorphic vector bundle but nevertheless have a reasonable definition of curvature that agrees with the curvature of the Chern connection when the family of Hilbert spaces is locally trivial.  As one consequence, we obtain a new proof of a theorem of Berndtsson on the curvature of direct images of semi-positively twisted relative canonical bundles (i.e., adjoint bundles), and of its higher-rank generalization by Liu-Yang.
\end{abstract}

%\tableofcontents

In the present article a \emph{holomorphic family} of complex manifolds is a surjective holomorphic submersion $p:X \to B$.  The holomorphic family is said to be \emph{proper} if $p:X \to B$ is proper.   The holomorphic family is said to be \emph{K\"ahler} if it admits a \emph{relative K\"ahler form}, i.e., a closed real $(1,1)$-form $\omega$ on the total space $X$, whose pullback to every fiber is K\"ahler.

\section{Introduction}

Let $p: X \to B$ be a holomorphic family, and denote by $m$ and $n$ the dimension of $B$ and the fiber dimension of $p$ respectively.  Let $E \to X$ be a holomorphic vector bundle with smooth Hermitian metric $h$.  We study the family of Hilbert spaces $\sH \to B$ whose fiber over $t \in B$ is 
\[
\sH _t := \left \{ f \in H^0(X_t, \cO _{X_t} (K_{X_t}\tensor E|_{X_t}))\right \},
\]
where $X_t := p^{-1} (t)$, and whose Hermitian metric on the fiber $\sH_t$ is 
\begin{equation}\label{bls-inner-product}
(f_1,f_2)_t := \int _{X_t}\ii^{n^2} \left < f_1 \wedge \bar f_2 , h \right >.
\end{equation}
The study of such families of Hilbert spaces was initiated  in \cite{bo-annals} by Berndtsson, who considered the case in which $E$ is a line bundle.  Berndtsson gives such families of Hilbert spaces a complex geometric structure by defining notions of smooth sections of $\sH \to B$ and a $\dbar$ operator for such sections.  He notes that in general such structures do not yield vector bundles, even when the family $p:X \to B$ is proper.  (Some natural examples are discussed in Paragraph \ref{non-loc-triv-examples}.)  However, even when not locally trivial, $\sH \to B$ resembles a holomorphic Hermitian vector bundle in certain ways.  

In the present article, we mostly study proper holomorphic families of complex manifolds.  We propose a definition of the curvature endomorphism of $\sH \to B$ (Definition \ref{H-curvature-definition}).  Our definition agrees with the usual definition if $\sH \to B$ is a Hermitian holomorphic vector bundle, i.e., it is locally trivial.  Our first result is the following theorem.  

\begin{m-thm}\label{bo-2nd-thm-cor}
Let $(E,h) \to X \stackrel{p}{\to} B$ be a proper K\"ahler family of Hermitian  holomorphic vector bundles and let $k \ge 1$ be an integer.  If the metric $h$ is $k$-non-negatively curved (resp. $k$-positively curved) then $\sH \to B$ is $k$-non-negatively curved (resp. $k$-positively curved).
\end{m-thm}

\noi The notion of $k$-positivity was introduced by Demailly in \cite[Définition 2.2]{dem-82}.  Metrics are $1$-positively curved if and only if they are positively curved in the sense of Griffiths, while metrics are $k$-positively curved for every $k$ if and only if they are positively curved in the sense of Nakano.  If $E$ is a line bundle then all of these notions are equivalent.  Notions of positivity are recalled in Paragraph \ref{positivity-par}.

\medskip

By taking $k \ge \min (n, r)$ where $r = {\rm Rank} (E)$ and $n$ is the fiber dimension of $p$, one recovers a result first proved by Mourougane and Takayama \cite[Theorem 1.1, $q=0$ case]{mt}.

Except in the case $k \ge \min (n, r)$, we do not know if the hypotheses of Theorem \ref{bo-2nd-thm-cor} imply that $\sH \to B$ is a holomorphic vector bundle.  At present the main tool for local triviality is the $L^2$ Extension package of Ohsawa-Takegoshi type.  (See Paragraph \ref{OT-par}.)  To apply the $L^2$ Extension Theorem, the vector bundle $E \to X$ must admit a metric with non-negative curvature in the sense of Nakano.  The hypotheses in the next two corollaries yield such metrics.

%The first corollary of Theorem \ref{bo-2nd-thm-cor} is the following celebrated theorem of Berndtsson.

\begin{cor}[\text{\cite[Theorem 1.2]{bo-annals}}]\label{bo-B2-thm}
Let $(E, e^{-\vp}) \to X \stackrel{p}{\to} B$ be a proper K\"ahler family of holomorphic Hermitian line bundles.  If the curvature of $(E, e^{-\vp}) \to X$ is non-negative (resp. positive) then $\sH \to B$ is a holomorphic vector bundle and the curvature of the Chern connection of the $L^2$ metric \eqref{bls-inner-product} is non-negative (resp. positive) in the sense of Nakano.
\end{cor}

The next corollary is a direct consequence of Theorem \ref{bo-2nd-thm-cor}, and local triviality holds because Nakano positivity of the metric $h$ on fibers implies that for a sufficiently plurisubharmonic function $\psi$ on an open subset $U \subset B$ the metric $e^{-p^* \psi} h$ for $E|_U$ is Nakano positive.

\begin{cor}\label{xu-wang-cor}
Let $(E,h) \to X \stackrel{p}{\to} B$ be a proper K\"ahler family of holomorphic Hermitian vector bundles and let $k \ge 1$ be an integer.  Assume that for all $t \in B$ the metrics $h_t = h|_{X_t}$ are Nakano-positive, and that the metric $h$ is $k$-non-negatively (resp. $k$-positively) curved.  Then the holomorphic vector bundle $\sH \to B$ is $k$-non-negatively (resp. $k$-positively) curved. 
\end{cor}

\begin{rmk}\label{better-pos}
One can weaken the hypotheses of Corollary \ref{xu-wang-cor}: instead of assuming that the metric $h$ is fiberwise-Nakano positive, one can assume that 
\[
\Theta (h) \ge - C p^* \omega \tensor {\rm Id}_E
\]
for some positive $(1,1)$-form $\omega$ on $B$ and some continuous $C :X \to [0,\infty)$.  If the fibers are projective (or more generally essentially Stein) then one can even just assume that $h$ is fiberwise-nonnegative, though in that case the proof requires a bit more work.  These hypotheses also imply that $\sH \to B$ is locally trivial, and then the result follows from Theorem \ref{bo-2nd-thm-cor}.  However, we have chosen to state Corollary \ref{xu-wang-cor} with the stronger hypotheses because in Section \ref{sff-section} we give a different, more direct proof.  Our proof uses H\"ormander's $L^2$ estimate and twisted Lefschetz theory (both of which are recalled in Section \ref{Hodge-thy-par}), as well as Theorem \ref{B2-thm}, which we formulate next.
\red
\end{rmk}

Theorem \ref{bo-2nd-thm-cor} can be strengthened.  Indeed, there is always a non-negative contribution to the curvature of $\sH \to B$ from the deformation-theoretic properties of the proper holomorphic family $p :X \to B$, and this contribution can be positive when the family is not locally trivial.  That this is so was already observed by Berndtsson in \cite{bo-annals} and then quantified in \cite{bo-sns} and later by Liu and Yang in \cite{ly}, who were the first to consider the case of higher rank $E$.    In all cases the extra positivity is a consequence of formulas for the curvature of $\sH \to B$.  (Liu and Yang also obtained a number of other equivalent curvature formulas and discussed several interesting applications of these formulas.)  

In \cite{bo-sns} and \cite{ly} the curvature formula is obtained under the assumption of local triviality of $\sH \to B$, which makes the computation more manageable.  In the non-locally trivial case the derivation of the curvature formula is necessarily more complicated.  To make its meaning maximally clear, we obtain the curvature formula in a pair of theorems, whose statements make use of the following notation.
\begin{enumerate}
\item[i.]  A \emph{horizontal distribution} for $p:X \to B$ is a (not necessarily integrable) subbundle $\theta \subset T_X$ such that $dp |_{\theta} : \theta \to T_B$ is a vector bundle isomorphism.  Such distributions $\theta$ are in one-to-one correspondence with \emph{horizontal lifts}, i.e., linear maps that assign to a $(1,0)$-vector $\tau \in  T^{1,0}_{B,t}$ a $(1,0)$-vector field $\xi ^{\theta}_{\tau} \in \Gamma (X_t, T^{1,0} _X|_{X_t})$ such that $dp \xi ^{\theta} _{\tau} = \tau$.  (See Paragraph \ref{lifts-paragraph}.)
\item[ii.] Let $t \in B$ be fixed.  For $\tau \in T^{1,0} _{B,t}$ we define 
\[
\kappa ^{\theta} _{\tau} : \sH _t \to H^0(X_t, \sC^{\infty} (\Lambda ^{n-1,1}_{X_t} \tensor E|_{X_t}))\quad \text{and} \quad \two ^{\theta} _{\tau} : \sH _t \to H^0(X_t, \sC^{\infty} (\Lambda ^{n,0}_{X_t} \tensor E|_{X_t})) 
\]
by 
\[
\kappa ^{\theta} _{\tau} f := \iota _{X_t} ^* (\dbar \xi ^{\theta} _{\tau} \lrcorner u )\quad \text{and} \quad \two ^{\theta} _{\tau} f := P_t ^{\perp} \iota _{X_t} ^* (L^{1,0} _{\xi ^{\theta} _{\tau}} u),
\]
where $u \in H^0 (X, \sC^{\infty} (\Lambda ^{n,0}_{X} \tensor E))$ is any section such that $\iota _{X_t} ^* u = f$ and $\iota _{X_t} ^* L^{1,0} _{\xi ^{\theta} _{\tau}} \dbar u \equiv 0$. Here $P_t ^{\perp} \!: H^0(X_t, L^2_{\ell oc} (K_{X_t} \tensor E|_{X_t})) \to \sH _t ^{\perp}$ is the orthogonal projection and 
\[
L^{1,0} _{\xi ^{\theta} _{\sigma}} \alpha := \nabla ^{1,0} ( \xi^{\theta} _{\sigma} \lrcorner \alpha )+ \xi^{\theta} _{\sigma} \lrcorner \nabla ^{1,0} \alpha,
\] 
is the \emph{twisted Lie derivative} of an $E$-valued differential form $\alpha$ associated to the Chern connection $\nabla = \nabla ^{1,0} + \dbar$ for $(E,h) \to X$.  (See paragraph \ref{Lie-deriv-paragraph}.)  As we will see later, these definitions are independent of the choice of $u$.
\end{enumerate}

\begin{m-thm}\label{B2-thm}
Let $(E,h) \to X \stackrel{p}{\to} B$ be a proper K\"ahler family holomorphic Hermitian vector bundles. Then the curvature $\Theta ^{\sH}$ of $\sH \to B$ is given by 
\begin{equation}\label{H-curv-formula-in-mthm}
( \Theta ^{\sH} _{\sigma \bar \tau} f_1, f _2) =  \left ( \Theta (h) _{\xi^{\theta} _{\sigma} \bar \xi^{\theta} _{\tau}} f _1, f_2\right )  - \int _{X_t}  \!\!\!\! \ii^{n^2} \!\! \left < \kappa ^{\theta} _{\sigma} f_1 \wedge \overline{ \kappa ^{\theta} _{\tau}f_2}, h \right > - \int _{X_t} \!\!\!\!  \ii^{n^2} \!\! \left < \two^{\theta} _{\sigma} f_1 \wedge \overline{\two^{\theta} _{\tau}f_2}, h \right >, 
\end{equation}
where $\theta$ is any horizontal distribution, $t \in B$, $\sigma, \tau \in T^{1,0} _{B,t}$, $f_1, f_2 \in \sH _t$, $\Theta (h)$ is the curvature of $(E,h) \to X$ and $( \cdot , \cdot )$ is the inner product \eqref{bls-inner-product}.
\end{m-thm}

We emphasize that the curvature $\Theta ^{\sH}$ is independent of the choice of $\theta$.

\begin{rmk}\label{hol-ext-convention-rmk}
Each of the three terms on the right hand side of Formula \eqref{H-curv-formula-in-mthm}, while dependent on the choice of horizontal distribution $\theta$, is independent of the way in which the $(1,0)$-vectors $\sigma$ and $\tau$ are extended to $(1,0)$-vector fields.  We therefore adopt the inessential but simplifying convention that these vectors are extended locally holomorphically, i.e., $\sigma, \tau \in \cO (T^{1,0}_B)_t$.
\red
\end{rmk}

The precise definition of the curvature $\Theta ^{\sH}$ of $\sH \to B$ requires some technology that we collect or develop in Sections \ref{bls-section}, \ref{lie-deriv-section} and \ref{BLS-section-fields-section}.  But with Theorem \ref{B2-thm} now stated, we can explain the basic idea behind the definition of $\Theta ^{\sH}$.  Consider the family of Hilbert spaces $\sL \to B$ whose fiber over $t \in B$ is 
\[
\sL _t := \left \{ f \in H^0(X_t, L^2_{\ell oc}(K_{X_t}\tensor E|_{X_t}))\right \},
\]
equipped with the same inner product \eqref{bls-inner-product} used to define $\sH$.  The curvature of $\sL$ can be defined because $\sL$ is reasonably well-behaved: it is a smooth Hilbert field in the sense of Lempert-Sz\H{o}ke.  

Evidently the fibers of $\sH \to B$ are subspaces of the fibers of $\sL \to B$.  Each horizontal distribution $\theta$ defines an operator $\dbar ^{\theta}$ on smooth sections of $\sL$ (Definition \ref{dbar-theta-defn}).  Remarkably, for every $\theta$ the operator $\dbar^{\theta}$, when restricted to sections of $\sH \to B$, agrees with the $\dbar$-operator of $\sH \to B$ introduced by Berndtsson (Definition \ref{holo-dbar-defn}).  Therefore $(\sH, \dbar )$ is a `holomorphic sub-field' of $(\sL , \dbar ^{\theta})$.  

If $\sH \to B$ is locally trivial then a Gauss-Griffiths-type formula links the curvature of $\sH$, the curvature of $(\sL, \dbar ^{\theta})$ restricted to $\sH$, and the second fundamental form of $\sH \subset (\sL, \dbar ^{\theta})$ (Proposition \ref{gg-formula}).

One can also define the second fundamental form in $\sH$ even when $\sH$ is not locally trivial.  This definition yields the curvature as an endomorphism of $\sH$ viewed as a $\sC^{\infty}_B$-module.  Unlike the finite-rank case, such an endomorphism is not automatically represented by a map of Hilbert fields.  To have such a map one must be able to extend a dense subset of every fiber $\sH$ to sections of $\sL$ with particular special properties.  The details are described abstractly %, i.e., for general Hilbert fields, 
in Paragraph \ref{iBLS-paragraph}, and then more concretely in Theorem \ref{iBLS-structures-for-H} for $\sH \subset \sL$.

From the Gauss Formula perspective the sum of the first two terms on the right hand side of the identity \eqref{H-curv-formula-in-mthm} is precisely the restriction to $\sH$ of the curvature of $(\sL , \dbar ^{\theta})$ (Theorem \ref{L2-curvature-prop}), while the last term on the right hand side of \eqref{H-curv-formula-in-mthm} is the negative of the `square of the second fundamental form'.   

%\begin{rmk}\label{1st-non-integ-rmk}
%While Berndtsson's $\dbar$-operator satisfies $\dbar \dbar = 0$, the operator $\dbar ^{\theta}$ satisfies $\dbar ^{\theta}\dbar ^{\theta} = 0$ if and only if the horizontal distribution $\theta$ is Frobenius-integrable (Proposition \ref{integrable-dbar}).  Thus $(\sL ,\dbar ^{\theta})$ is a sort of `almost holomorphic' Hilbert field.  Like holomorphic Hilbert fields, almost holomorphic Hilbert fields have at most one Chern connection, but the curvature of this Chern connection is not a $(1,1)$-form if $\dbar \dbar \neq 0$.  
%\red
%\end{rmk}

\begin{rmk}\label{ehresmann-rmk}
By Ehresmann's Theorem the properness of the holomorphic family $p:X \to B$ implies that the field of Hilbert spaces $\sL \to B$ has a natural structure of a smooth Hilbert bundle, i.e., it is locally trivial.

Here we are more concerned with the complex structure, which is required for any notion of Chern connection.  From this perspective, we have not addressed here the interesting question of whether for \emph{some} horizontal distribution $\theta$ the field of Hilbert spaces $\sL \to B$ has the structure of an almost holomorphic vector bundle whose $\dbar$-operator is $\dbar ^{\theta}$.  We suspect that this is rarely the case.

Instead we circumvent this question by introducing the notion of Berndtsson-Lempert-Sz\H{o}ke (BLS) fields.  This notion combines Berndtsson's complex structures with the notion of smooth structure introduced by Lempert and Sz\H{o}ke.  The notion is general beyond the case of proper holomorphic families.  One can give meaning to curvature in the setting of families of complex manifolds-with-boundary, as has recently been done by Upadrashta \cite{pranav}, whose work significantly generalizes some of the elegant results of X. Wang \cite{wang}.
\red
\end{rmk} 

As stated in Remark \ref{better-pos}, Theorem \ref{B2-thm} yields a proof of Corollary \ref{xu-wang-cor} that does not pass through Theorem \ref{bo-2nd-thm-cor}.  To complete such a proof one must estimate the second and third terms on the right hand side of the curvature formula \eqref{H-curv-formula-in-mthm}.  The third term can be estimated using basic inequalities and a direct computation that is independent of the horizontal distribution $\theta$.  To handle the second term, one wants to have the equality 
\begin{equation}\label{hrbr-lift}
- \int _{X_t} (\text{-}1)^{\tfrac{n^2}{2}} \left < \kappa ^{\theta} _{\sigma} f_1 \wedge \overline{ \kappa ^{\theta} _{\tau}f_2}, h \right > = \left ( \kappa ^{\theta} _{\sigma} f_1 , \kappa ^{\theta} _{\tau} f_2 \right ),
\end{equation}
where the right hand side is the inner product induced on $E$-valued $(n-1,1)$-forms on $X_t$ by the metric $h$ and any K\"ahler metric on the fiber $X_t$.  The equality \eqref{hrbr-lift} does not hold for all horizontal lifts $\theta$, but such a $\theta$ exists on any K\"ahler family.

%Theorem \ref{prim-lift-thm} is thus a consequence of Theorem \ref{good-horizontal-distribution}, which in turn relies heavily on the Hodge Theorem.  

\begin{s-rmk}
A previous draft of the article had a more complicated approach to finding $\theta$.  We are grateful to the anonymous referee for pointing out that horizontal lift $\theta$ defined by $\xi ^{\theta} _{\tau} \lrcorner \omega = 0$, where $\omega$ is the relative K\"ahler form, has the required property.
\red
\end{s-rmk}

\medskip

The identity \eqref{hrbr-lift}, while handy for a proof of Corollary \ref{xu-wang-cor} (see Paragraph \ref{direct-cor-pf-par}), is not required in our proof of Theorem \ref{bo-2nd-thm-cor}.  Instead one uses the Hodge Theorem in a different way to compute an exact formula for the square of the second fundamental form, i.e., the third term on the right hand side of \eqref{H-curv-formula-in-mthm}.    The statement is as follows.

\begin{m-thm}\label{berndtsson-reps-thm}
Let $(E,h) \to X \stackrel{p}{\to} B$ be a K\"ahler family of holomorphic Hermitian vector bundles, let $\theta$ be a smooth horizontal distribution, let $t \in B$ and let $\omega _t$ be a K\"ahler form on $X_t$.  Then for every $f_1,f_2 \in \sH _t$ there exist smooth section $u _1, u_2 \in H^0(X, \sC^{\infty}(\Lambda ^{n,0}_X \tensor E))$, depending on $\theta$ and $\omega _t$ and satisfying $\iota _{X_t} ^* u_i  = f_i$ and $\iota _{X_t} ^*L^{1,0} _{\xi ^{\theta}_{\tau}}\dbar u_i = 0$, $i=1,2$, such that
\begin{eqnarray}\label{sff-rep-formula}
\nonumber  && - \int _{X_t} \!\!\!\!  \ii^{n^2} \!\! \left < \two^{\theta} _{\sigma} f_1 \wedge \overline{\two^{\theta} _{\tau}f_2}, h \right >  =  \int _{X_t}  \!\!\!\! \ii^{n^2} \!\! \left < \kappa ^{\theta} _{\sigma} f_1 \wedge \overline{ \kappa ^{\theta} _{\tau}f_2}, h \right >  + \left ( \xi ^{\theta} _{\sigma} \lrcorner \dbar u_1, \xi ^{\theta} _{\tau} \lrcorner \dbar u_2 \right) \\
&& \qquad  + \left (\Lambda _{\omega _t}\iota _{X_t}^*\left ( \xi ^{\theta}_{\sigma} \lrcorner \ii \Theta (h) \right ) f_1, \xi ^{\theta}_{\tau}\lrcorner u_2\right ) + \left ( \xi ^{\theta}_{\sigma}\lrcorner u_1, \Lambda _{\omega _t} \iota _{X_t}^*\left (\xi ^{\theta}_{\tau} \lrcorner \ii \Theta (h)   \right ) f_2 \right ) \\
\nonumber && \qquad \qquad - \left (\left [\ii \Theta (h_t), \Lambda_{\omega_t}\right ] (\xi ^{\theta} _{\sigma}\lrcorner  u_1), (\xi ^{\theta} _{\tau}\lrcorner u_2)\right )
\end{eqnarray}
for all  $\sigma, \tau \in T^{1,0} _{B,t}$.  The pairings $( \cdot , \cdot )$ in the last four terms on the right hand side of \eqref{sff-rep-formula} are the $L^2$ inner products induced on $E|_{X_t}$-valued forms by the Hermitian metric $h_t$ and the K\"ahler metric $\omega _t$.
\end{m-thm}

The aforementioned use of the Hodge Theorem in the proof of Theorem \ref{berndtsson-reps-thm} is equivalent to a construction of Berndtsson, formalized below as Theorem \ref{bo-rep-method}.  With Theorem \ref{bo-rep-method} in hand, the rest of the proof of Theorem \ref{berndtsson-reps-thm} is a somewhat involved integration-by-parts argument.  

Finally, Theorem \ref{bo-2nd-thm-cor} is follows from Theorems \ref{B2-thm} and \ref{berndtsson-reps-thm} and the formula 
\begin{eqnarray*}
&& \left < p _* \left ( \ii ^{n^2} \left < (\Theta (h) u _1 \wedge \overline{u _2}, h\right > \right ) , \sigma \wedge \bar \tau \right > \\
&&\quad =  \left ( \Theta (h) _{\xi^{\theta} _{\sigma} \bar \xi^{\theta} _{\tau}} f _1, f_2\right ) + \left (\Lambda _{\omega _t}\iota _{X_t}^*\left ( \xi ^{\theta}_{\sigma} \lrcorner \ii \Theta (h) \right ) f_1, \xi ^{\theta}_{\tau}\lrcorner u_2\right ) + \left ( \xi ^{\theta}_{\sigma}\lrcorner u_1, \Lambda _{\omega _t} \iota _{X_t}^*\left (\xi ^{\theta}_{\tau} \lrcorner \ii \Theta (h)   \right ) f_2 \right ) \\
\nonumber && \qquad \qquad - \left (\left [\ii \Theta (h_t), \Lambda_{\omega_t}\right ] (\xi ^{\theta} _{\sigma}\lrcorner  u_1), (\xi ^{\theta} _{\tau}\lrcorner u_2)\right ),
\end{eqnarray*}
which we establish in Section \ref{sff-section} (Proposition \ref{expand-the-current-prop}).  Indeed, together with Theorems \ref{B2-thm} and \ref{berndtsson-reps-thm}, the latter identity yields 
\begin{equation}\label{bly-formula}
( \Theta ^{\sH} _{\sigma \bar \tau} f_1, f _2) = \left < p _* \left ( \ii ^{n^2} \left < (\Theta (h)\wedge u _1 )\wedge \overline{u _2}, h\right > \right ), \sigma \wedge \bar \tau \right > + \left ( \xi ^{\theta} _{\sigma} \lrcorner \dbar u_1, \xi ^{\theta} _{\tau} \lrcorner \dbar u_2 \right).
\end{equation}
If $\sH$ is locally trivial Formula \eqref{bly-formula} is due to Berndtsson \cite[Section 2]{bo-sns} if ${\rm Rank}(E) = 1$  and to Liu-Yang \cite[Theorem 1.6]{ly} in the higher rank case.  The term 
\[
\left ( \xi ^{\theta} _{\sigma} \lrcorner \dbar u_1, \xi ^{\theta} _{\tau} \lrcorner \dbar u_2 \right)
\]
on the right hand side of \eqref{bly-formula} gives the extra positivity for the improvement of Theorem \ref{bo-2nd-thm-cor} mentioned after the statement of Corollary \ref{xu-wang-cor}.  This term and the  term $\ii^{n^2} \!\! \left < \kappa ^{\theta} _{\sigma} f_1 \wedge \overline{ \kappa ^{\theta} _{\tau}f_2}, h \right >$ are related to deformation theory of compact complex manifolds, and more specifically to the Kodaira-Spencer map, which captures the first obstruction to local triviality of holomorphic families.  (See Section \ref{KS-section} for more detail.)

We end this introduction with an explanation of one of our initial motivations.  If the family $X \to B$ is trivial then Corollary \ref{bo-B2-thm} has a considerably simpler proof.  In fact, this same simplified proof yields a significantly more general result, stated as \cite[Theorem 1.1]{bo-annals}.  By our reading of \cite{bo-annals}, there is a problem implicitly stated in the second-to-last paragraph of the introduction to \cite{bo-annals}:  whether one can prove \cite[Theorem 1.2]{bo-annals} (i.e. Corollary \ref{bo-B2-thm} above) in a manner similar to the proof of \cite[Theorem 1.1]{bo-annals}.  Berndtsson states there that he was not ``... able to find a natural complex structure on the space of all (not necessarily holomorphic) $(n, 0)$ forms, extending the complex structure on [$\sH$].''  In fact, there are infinitely many consistent ways to extend the complex structure of $\sH$ to $\sL$, each one corresponding to a choice of horizontal lift $\theta$.  The catch is that the resulting object is not a holomorphic vector bundle, and a central point of the present article is to deal with the lack of local triviality.

Aside from a solution of Berndtsson's implicit problem, we obtain a somewhat different geometric perspective.  In BLS geometry the Kodaira-Spencer class affects both the curvature of $\sL$ and the second fundamental form of $\sH$ in $\sL$.  In the curvature of $\sL$ one sees the Hodge-Riemann `square' 
\[
- \int _{X_t}\!\!\!\! \ii^{n^2} \!\! \left < \kappa ^{\theta} _{\sigma} f_1 \wedge \overline{ \kappa ^{\theta} _{\tau}f_2}, h \right >
\]
of a certain representative $\kappa ^{\theta} _{\sigma}f$ of the action on $E$-valued $(n,0)$-forms of  of the Kodaira-Spencer map of $p:X \to B$; a representative that depends only on the choice of complex structure $\dbar ^{\theta}$ for $\sL$, while in the second fundamental form of $\sH$ in $\sL$ one sees the ``difference of squares'' 
\[
- \int _{X_t}\!\!\!\!\ii^{n^2} \!\! \left < \kappa ^{\theta} _{\sigma} f_1 \wedge \overline{ \kappa ^{\theta} _{\tau}f_2}, h \right > - \left ( \xi ^{\theta} _{\sigma} \lrcorner \dbar u_1, \xi ^{\theta} _{\tau} \lrcorner \dbar u_2 \right)
\]
of two representatives of the Kodaira-Spencer map.  The first of these is the term appearing in the curvature of $\sL$, and as already mentioned it depends only on the choice of complex structure for $\sL$.  On the other hand, the other representative 
\[
\iota _{X_t}^*\xi ^{\theta} _{\sigma} \lrcorner \dbar u = \kappa ^{\theta} _{\sigma}f - \dbar \left ( \iota _{X_t}^* ( \xi ^{\theta} _{\sigma} \lrcorner u)\right )
\]
of the action of the Kodaira-Spencer map depends on how elements $f \in \sH_t$ are extended to sections $u$ of $\sL$.  Only the term $\left ( \xi ^{\theta} _{\sigma} \lrcorner \dbar u_1, \xi ^{\theta} _{\tau} \lrcorner \dbar u_2 \right)$ survives in the formula \eqref{bly-formula} for $\Theta ^{\sH}$.
 
In the same penultimate paragraph of the introduction of \cite{bo-annals} it is mentioned that the Kodaira-Spencer class of $p:X \to B$ ``plays somewhat the role of another second fundamental form, but this time of a quotient bundle, ...''.  It is certainly intriguing to try to realize $\sH$ is quotient bundle instead of a subbundle.  We have thus far been unable to find such a construction.

\subsection*{Outline of the article}

To help the reader navigate our rather long article, we summarize the content of the paper as follows.

\begin{enumerate}[{\textsection 1.}]

\setcounter{enumi}{1}

%\subsubsection*{Section \ref{bls-section}}
\item We present the abstract theory of BLS fields.  

%\subsubsection*{Section \ref{lie-deriv-section}} 
\item We discuss the relative canonical bundle, and then we carry out some explicit calculations that will be useful in various proofs.  We also discuss horizontal distributions.  

%\subsubsection*{Section \ref{BLS-section-fields-section}} 
\item We begin by recalling Berndtsson's construction of what we have called Berndtsson's Hilbert field.  We then introduce the $L^2$ BLS fields, and the iBLS structure of the Berndtsson Hilbert field, and we state the result on lift-independence of the relative curvatures of the Berndtsson Hilbert field with respect to the $L^2$ BLS fields (Theorem \ref{indep}).  After that, we define positivity for Hilbert fields; the definition is the obvious adaptation of the notions of positivity for holomorphic vector bundles to the setting of Hilbert fields.  Then we establish the Schur-complement-type result (Proposition \ref{positivity-of-dets}) needed in the proof of Corollary \ref{xu-wang-cor}.  We conclude with a discussion of semi-local descriptions of representatives of of sections of twisted relative canonical bundles.

%\subsubsection*{Section \ref{loc-triv-section}} 
\item We discuss the question of local triviality of the Berndtsson Hilbert field.

%\subsubsection*{Section \ref{Hodge-thy-par}} 
\item We collect well-known results from Hodge Theory that are needed in the proofs of Theorem \ref{bo-2nd-thm-cor} and Corollary \ref{xu-wang-cor}. 

%\subsubsection*{Section \ref{KS-section}}  
\item We discuss the Kodaira-Spencer map and its action on elements of the Berndtsson Hilbert field.

%\subsubsection*{Section \ref{B2-thm-pf-section}}  
\item We compute (the $(1,1)$-component of) the curvatures of the Ambient Hilbert fields, from which Theorem \ref{B2-thm} follows.  We then use the computation to finally prove the theorem on lift-independence of the relative curvatures (Theorem \ref{indep}).

%\subsubsection*{Section \ref{sff-section}}  
\item We begin by proving Corollary \ref{xu-wang-cor}.  We then prove Theorem \ref{berndtsson-reps-thm} using a slight generalization of a result of Berndtsson concerning the existence of good representatives of twisted relative canonical bundles (Theorem \ref{bo-rep-method}).  We conclude with an integration-by-parts computation (Proposition \ref{expand-the-current-prop}) which, when combined with Theorems \ref{B2-thm} and \ref{berndtsson-reps-thm}, yields Berndtsson-Liu-Yang's curvature formula \eqref{bly-formula}, from which Theorem \ref{bo-2nd-thm-cor} follows.
\end{enumerate}

\begin{s-rmk}
It is important to mention that many people have studied the direct image sheaf of the relative canonical bundle twisted by a holomorphic line bundle or vector bundle.  Among these, most closely related to the present article are the results of Mourougane-Takayama \cite{mt-0,mt}, Berndtsson-P\u{a}un \cite{bp}, Berndtsson-P\u{a}un-Wang \cite{bpw}, and Liu-Yang \cite{ly}.  

On the other hand, it might also be worth pointing out that the present article concerns a rather different situation; when not local trivial, the family $\sH \to B$ is not so closely related to the direct image of the twisted relative canonical bundle.
\red
\end{s-rmk}

%\begin{comment}
\subsection*{Acknowledgements}  
This article could not have been written without the discussions and conversations I've had with Roberto Albesiano and Pranav Upadrashta, and I thank them both.  Seb Boucksom has given me invaluable advice about the content and presentation of the article, and I happily thank him for it, as should the reader.  I'm grateful to Mihnea Popa and Mihai P\u aun for stimulating feedback.  I am grateful to the anonymous referees, whose comments and suggestions have been very useful in improving the article.  %As usual, all remaining errors are my own.

Finally, Bo Berndtsson and Laszlo Lempert have provided important comments and suggestions, but most of all their work has inspired and directly affected my research, not only in this article, but throughout my career.  This article is dedicated to them.
%\end{comment}

\section{Abstract Theory of Berndtsson-Lempert-Sz\H{o}ke Fields}\label{bls-section}

In this section we discuss the general geometric structure with which $\sH$ and $\sL$ are eventually endowed.  We start with a slight reorganization and minor refinement of ideas introduced by Lempert and Sz\H{o}ke in \cite{ls}, and then we introduce definitions that are already suggested or implied in \cite{ls}.  For example, in \cite{ls} the introduction of a $\dbar$-operator to define the holomorphic structure of a smooth (quasi)-Hilbert field is not abstracted, but introduced in the `direct image' constructions of \cite[Part II]{ls} for the bundle of holomorphic sections.  And the discussion of subfields is not as explicit, but many of the elements are hinted at in both \cite{ls} and \cite{bo-annals}.  In fact, \cite{bo-annals} considers subfields in the case of trivial families of complex manifolds.  

\subsection{Quasi-Hilbert fields}

\begin{defn}
A quasi-Hilbert field is a map $\pi :\fH \to B$ whose fibers are topological vector spaces each of which is isomorphic to a Hilbert space.
\red
\end{defn}

That is to say, for each $t \in B$ there is a Hilbert space $H_t$ such that the fiber $\fH_t := \pi ^{-1}(t)$ is isomorphic, as a topological vector space, to $H_t$, but $\fH _t$ does not have a canonical inner product.  The name is inspired by the fact that a linear isomorphism of two Hilbert spaces is a quasi-isometry.  We shall say that the Hilbert space $H_t$ is quasi-isometric to $\fH _t$, even though the latter does not have a canonical inner product.

\begin{rmk}
Even if the Hilbert space fibers of a quasi-Hilbert field are finite-dimensional, one does not want all of them to be isomorphic to a single vector space.  Indeed, in our setting of a  holomorphic family $E \to X \stackrel{p}{\to} B$ in which $p$ is proper, the field of Hilbert spaces $\sH \to B$ has constant fiber dimension if and only if it is locally trivial.
\red
\end{rmk}

\begin{defn}
The dual of a quasi-Hilbert field $\fH \to B$ is the quasi-Hilbert field $\fH ^* \to B$ such that for each $t \in B$ the fiber $\fH _t ^*$ is the topological dual of $\fH _t$, i.e., the set of bounded linear functionals on $\fH _t$.
\red
\end{defn}

By the Riesz Representation Theorem each fiber $\fH _t ^*$ of the dual quasi-Hilbert field $\fH ^*$ is isometric to the dual fiber $\fH _t$.  Therefore the Hilbert fields $\fH$ and $\fH ^*$ are isomorphic  quasi-Hilbert fields.  %The differences arise with the introduction of smooth structures.

\subsection{Smooth structures on quasi-Hilbert fields}

We denote by $\Gamma (U , \fH)$ the set of sections (i.e. right inverses) of $\pi : \fH \to B$ over $U \subset B$.  

\begin{defn}\label{smooth-structure-defn}
Let $B$ be a smooth manifold and let $\fH \to B$ be a quasi-Hilbert field.
\begin{enumerate}
\item[i.]  A smooth structure for a quasi-Hilbert field $\fH \to B$ is a sheaf $\sC^{\infty} (\fH) \to B$ of $\sC^{\infty}$-modules on $B$, whose elements are germs of sections of $\pi :\fH \to B$, such that for each $t \in B$ the stalk $\sC ^{\infty} (\fH) _t$ is mapped by the evaluation-at-$t$ map to a dense subspace of the Hilbert space $\fH _t$.  The associated presheaf assigns to an open set $U \subset B$ a set of sections of $\pi : \pi ^{-1} (U) \to U$ denoted $\Gamma(U, \sC^{\infty} (\fH))$, whose members are called \emph{smooth sections over $U$}.
%\item[ii.] A quasi-Hilbert field with a smooth structure is called a \emph{smooth} quasi-Hilbert field.
\item[ii.] The smooth structure for $\fH ^* \to B$ dual to the smooth structure of $\fH \to B$ is the sheaf $\sC^{\infty} (\fH ^*) \to B$ associated to the presheaf whose groups $\Gamma (U, \sC^{\infty}(\fH ^*))$ are defined as follows: $\phi \in \Gamma (U, \sC^{\infty}(\fH ^*))$ if and only if for every $f \in \Gamma (U, \sC^{\infty}(\fH))$ the function $\left < \phi , f\right > \in \sC^{\infty}(U)$.
\red
\end{enumerate}
\end{defn}

\subsection{Maps}

We shall often encounter maps between (smooth) quasi-Hilbert fields.

\begin{defn}\label{maps-defn}
Let $\fH^1, \fH^2 \to B$ be quasi-Hilbert fields.
\begin{enumerate}
\item[a.]  A map $F : \fH^1 \to \fH^2$ is said to be a \emph{map of quasi-Hilbert fields} if it is fiber-preserving, i.e., 
\[
\begin{tikzcd} \fH^1 \ar[r ,"F"] \arrow[d]  &\fH^2 \arrow[d] \\
B \arrow[r,"{\rm Id}"]  &B,
\end{tikzcd}
\]
and linear on fibers.  We write $F \in {\rm Lin} (\fH^1, \fH^2)$.

(As is common in the study of linear transformations of Hilbert spaces, the fiber restrictions need not be continuous or even everywhere defined.)  If the fiber restrictions are defined on dense subsets of the fibers we will say that $F$ is \emph{densely-defined}.
\item[b.] A map of quasi-Hilbert fields $F :\fH^1 \to \fH^2$ is called a \emph{morphism} if the restrictions to fibers are continuous.  In this case we write $F \in {\rm Hom} (\fH^1, \fH^2)$.
\item[c.] Suppose moreover that $\sC^{\infty} (\fH^1)$ and $\sC^{\infty} (\fH^2)$ are smooth structures for $\fH ^1$ and $ \fH ^2$ respectively.   A map $F \in {\rm Lin}(\fH^1,\fH^2)$ is said to be \emph{smooth} if it maps $\sC^{\infty}(\fH^1)$ to $\sC^{\infty}(\fH^2)$, i.e., if for each $t \in B$ one has $F\ff\in \sC^{\infty}(\fH^2)_t$ for every $\ff \in \sC^{\infty}(\fH^1)_t$. 
\end{enumerate}
\end{defn}

Thus a smooth map is densely-defined.  However, a smooth map need not be a morphism, and a morphism need not be smooth.

One often also encounters maps of smooth sections.  

\begin{defn}
Let $\fH^1, \fH^2 \to B$ be quasi-Hilbert fields with smooth structures $\sC^{\infty} (\fH^1)$ and $\sC^{\infty} (\fH^2)$.  
\begin{enumerate}
\item A map of sheaves $\mathbf{L} : \sC^{\infty} (\fH^1) \to \sC^{\infty} (\fH^2)$  is said to be \emph{$\sC^{\infty}$-linear} if 
\[
\mathbf{L} (f \ff) = f \mathbf{L}\ff
\]
for each $t \in B$, all $\ff \in \sC^{\infty} (\fH ^1) _t$ and all $f \in \sC^{\infty} _{B,t}$.
\item A $\sC^{\infty}$-linear $\mathbf{L}$ is said to be \emph{tensorial at $t \in B$} if 
\[
\ff \in \sC^{\infty} (\fH ^1) _t \ \text{ and } \  \ff (t) = 0 \quad \Rightarrow \quad (\mathbf{L}\ff )(t) = 0.
\]
We say $\mathbf{L}$ is \emph{tensorial} if $\mathbf{L}$ is tensorial at every point of $B$.
\end{enumerate}
\end{defn}

Note that a tensorial $\sC^{\infty}$-linear map determines a smooth map of quasi-Hilbert fields in the sense of Definition \ref{maps-defn}.a.  (The converse is obvious.)  Indeed, if $\mathbf{L}$ is a tensorial $\sC^{\infty}$-linear map then one can define the map of quasi-Hilbert fields $L$ by 
\[
L \ff(t) := (\mathbf{L} \ff)(t), \quad \ff \in \sC^{\infty} (\fH)_t,
\]
and the tensorial condition means $L$ is well-defined.

\begin{rmk}
If $\sC^{\infty} (\fH ^1) _t$ is generated, as a $\sC^{\infty} _{B,t}$-module, by germs that do not vanish at $p$ then every $\sC^{\infty}$-linear map is tensorial at $t$.  This is the situation for smooth finite-rank vector bundles with their usual smooth structure consisting of the sheaf of germs of smooth sections.  %However, in general there seems to be no reason that every $\sC^{\infty}$-linear map is tensorial.
\red
\end{rmk}

\subsection{Connections and Curvature}

\begin{defn}\label{connection-defn}
Let $\pi :\fH \to B$ be a quasi-Hilbert field with a smooth structure.
\begin{enumerate}
\item[i.] A \emph{connection} for $\fH \to B$ is a map 
\[
\nabla : \Gamma (B,\sC^{\infty}(T_B))\times \Gamma (B, \sC^{\infty}(\fH)) \in (\xi, \ff) \mapsto \nabla _{\xi} \ff \in \Gamma (B, \sC^{\infty}(\fH))
\]
such that  
\[
\nabla _{f\xi +\eta} = f \nabla _{\xi} + \nabla _{\eta} \quad \text{ and }\nabla _{\xi} (f \ff) = (\xi f)\ff + f \nabla _{\xi}\ff
\]
for all $f \in \sC^{\infty}(B)$, $\xi \in \Gamma (B,\sC^{\infty}(T_B))$ and $\ff \in \Gamma (B,\sC^{\infty} (\fH))$.  
\item[ii.]  The map $\Theta (\nabla): \Gamma (B, \sC^{\infty}(\Lambda ^2 (T_B)))\times  \Gamma (B,\sC^{\infty} (\fH)) \to \Gamma (B,\sC^{\infty} (\fH))$ defined by 
\[
\left < \Theta (\nabla), (\xi \wedge \eta , \ff) \right > := \nabla _{\xi}\nabla _{\eta}\ff - \nabla_{\eta}\nabla_{\xi} \ff- \nabla_{[\xi,\eta]}\ff
\]
is called the \emph{curvature} of $\nabla$.
\red
\end{enumerate}
\end{defn}

\begin{defn}
Let $\fH \to B$ be a quasi-Hilbert field with smooth structure, and let $\fh$ be a smooth metric for $\fH$.  A connection $\nabla$ for $\fH$ is said to be \emph{compatible with the metric $\fh$} if 
\[
\left < d \fh (\ff_1, \ff_2), \xi \right > =  \fh (\nabla _{\xi} \ff_1, \ff_2) +  \fh (\ff_1, \nabla _{\xi}\ff_2)
\]
for all $t\in B$, all $\ff _1, \ff _2 \in \sC^{\infty} (\fH)_t$ and all $\xi \in T_{B,t}$.
\red
\end{defn}

One defines the covariant derivative of a smooth metric $\fh$ by 
\begin{equation}\label{nabla-h}
\nabla \fh (\ff_1,\ff_2) := d (\fh (\ff_1,\ff_2)) - \fh (\nabla \ff_1, \ff_2) - \fh (\ff_1, \nabla \ff_2).
\end{equation}
In terms of this derivative, a connection $\nabla$ is compatible with a smooth metric $\fh$ if and only if $\nabla \fh = 0$.

The difference of two connections and the curvature of a connection are $\sC^{\infty}$-linear.  If the connections are metric compatible then one can say more.

\begin{prop}\label{metric=>tensorial}
Let $\fH \to B$ be a quasi-Hilbert field with smooth structure, let $\fh$ be a smooth metric for $\fH$ and let $\nabla, \nabla _1, \nabla _2$ be $\fh$-compatible connections.  Then 
\[
\nabla _2 - \nabla _1 \quad \text{and} \quad \Theta (\nabla)
\]
are tensorial.
\end{prop}

\begin{proof}
Denoting either $\nabla _2 - \nabla _1$ or $\Theta (\nabla)$ by $\mathbf{L}$, one has $\fh (\mathbf{L} \ff, \fg) = - \fh (\ff, \mathbf{L} \fg)$.  Therefore if $\ff \in \sC^{\infty} (\fH)_t$ and $\ff (t) = 0$ then $\fh ( (\mathbf{L}\ff)(t), \fg(t)) = 0$ for all $\fg\in \sC^{\infty} (\fH)_t$.  Since $\sC^{\infty} (\fH)_t(t)$ is dense in $\fH _t$,  $(\mathbf{L} \ff) (t) = 0$, i.e., $\mathbf{L}$ is tensorial.
\end{proof}

\subsection{Formulation with differential forms}
Though it's largely a matter of taste when the base manifold $B$ is finite-dimensional (which is a standing assumption in the present article), one might sometimes prefer to think of the connection as a total differential, rather than talk about directional derivatives.  To do so it is convenient to introduce differential forms with values in $\fH$.  

\begin{defn}\label{forms-defn}
Let $\fH \to B$ be a smooth quasi-Hilbert field.  An $\fH$-valued differential $k$-form is a section $\fb$ of $\Lambda ^k _B \tensor \fH \to B$.  If $\fH \to B$ is a smooth quasi-Hilbert field then we say that  $\fb$ is \emph{smooth}  if for each $k$-tuple of vector fields $\tau _1,...,\tau _k\in \Gamma (B,\sC^{\infty}(T_B))$ the section $\left < \alpha, \tau _1 \wedge \cdots \wedge \tau _k \right >$ of $\fH \to B$ obtained by contracting in the factor $\Lambda ^k _B$ lies in $\Gamma (B,\sC^{\infty}(\fH))$.  We denote by $\Gamma (B, \sC^{\infty}(\Lambda ^k _B \tensor \fH))$ the collection of smooth $\fH$-valued $k$-forms.
\end{defn}

\begin{rmk}
It may very well be that the fibers of $\fH \to B$ are infinite-dimensional.  In this case, one often sees the analytic completion $\hat \tensor$.  Here we assume our $\fH$-valued $k$-forms can be written locally as the form $\alpha ^1 \tensor \ff _1 + \cdots + \alpha ^{m_k} \tensor \ff _{m_k}$, where $m = \dim _{\bbC} B$ and $m_k = \binom{m}{k}$.  The $\fH$-valued $k$-form is smooth precisely when the local forms $\alpha ^i$ and the local sections $\ff _i$ are smooth.  %We will really only care about smooth forms anyway.
\red
\end{rmk}

\begin{rmk}
Note that since the set of smooth sections restricts to dense subsets of the fibers, the same is automatically true of the set of smooth forms.
\red
\end{rmk}

From the `forms' point of view a connection $\nabla$ can be seen as a map 
\[
\nabla : \Gamma (B,\sC^{\infty}(\fH)) \to \Gamma(B, \sC^{\infty}(\Lambda ^1_B \tensor \fH)),
\]
and we can extend $\nabla$ to higher twisted exterior derivatives 
\[
\nabla : \Gamma(B,\sC^{\infty}(\Lambda ^k _B \tensor \fH)) \to \Gamma(B,\sC^{\infty}(\Lambda ^{k+1}_B \tensor \fH))
\]
in the usual way.  A connection is not $\sC^{\infty}(B)$-linear, but the difference of two connections is.  And if the difference $\nabla _2 - \nabla _1$ is also tensorial then 
\[
\nabla _2 - \nabla _1 \in \Gamma (B, \sC^{\infty} ( \Lambda ^1_B \tensor {\rm Lin}(\fH,\fH))).
\]
Similarly, if the curvature $\nabla \nabla$ of a connection $\nabla$ is tensorial then it can be identified with a $2$-form with values in ${\rm Lin}(\fH, \fH)$:
\[
\Theta (\nabla) \in \Gamma (B , \sC^{\infty} ( \Lambda ^{2}_B \tensor {\rm Lin}(\fH,\fH))).
\]
As noted in Proposition \ref{metric=>tensorial}, this is the case for metric connections.

\subsection{Smooth metrics} 

\begin{defn}
Let $\fH \to B$ be a quasi-Hilbert field.
\begin{enumerate}
\item[i.] A metric $\fh$ for $\fH \to B$ is an assignment to each fiber $\fH _t$ of a complete inner product $\fh _t$ so that $(\fH_t, \fh_t)$ and $\fH _t$ are quasi-isometric for all $t \in B$.
\item[ii.] If $\sC^{\infty} (\fH ) \to B$ is a smooth structure for $\fH \to B$ then the metric $\fh$ is \emph{smooth} if 
\[
\fh (\ff_1, \ff_2 )  \in \sC^{\infty} (U)
\]
for any smooth sections $\ff_1, \ff_2 \in \Gamma (U, \sC^{\infty} (\fH))$.
\red
\end{enumerate}
\end{defn}

\subsection{Smooth (quasi-)Hilbert fields}
\begin{defn}\label{smooth-(q)h-fld-defn}
\begin{enumerate}
\item A \emph{smooth quasi-Hilbert field} is a quasi-Hilbert field with a smooth structure that admits a connection.  
\item A smooth quasi-Hilbert field $\fH \to B, \sC^{\infty} (\fH)$ over $B$ is said to be a \emph{smooth Hilbert field} if it has a connection $\mathring{\nabla}$, called a \emph{reference connection}, and a smooth metric $\fh$ with the following property.  For each $t \in B$, $\xi \in T_{B,t}$ and $\ff \in \sC^{\infty} (\fH)_t$ there is a constant $C=C(t, \xi, \ff)$ such that 
\begin{equation}\label{rrt-cty}
\left |(\mathring{\nabla} _{\xi} \fh) (\ff, \fg)) (t) \right |^2  \le C  \fh (\fg, \fg)(t).
\end{equation}
for each $\fg \in \sC^{\infty} (\fH )_t$.
\red
\end{enumerate}
\end{defn}

\begin{prop}\label{metric-compatible-connection-prop}
A smooth Hilbert field $(\fH, \fh)$ has a metric compatible connection.
\end{prop}

\begin{proof}
Let $\mathring{\nabla}$ be a reference connection.  Any other connection is related to $\mathring{\nabla}$ by some $\sC^{\infty}$-linear map of sheaves $\mathbf{A}: \sC^{\infty} (\fH) \to \sC^{\infty} (T^*_B \tensor \fH)$, i.e., it is of the form $\nabla^{\mathbf{A}} := \mathring{\nabla} + \mathbf{A}$.   To find a metric compatible connection, one needs to find $\mathbf{A}$ satisfying $\nabla ^{\mathbf{A}} \fh = 0$ (c.f. definition \eqref{nabla-h}), or equivalently 
\[
\mathring{\nabla} \fh (\ff, \fg) = \fh (\mathbf{A}\ff, \fg ) + \fh (\ff, \mathbf{A}\fg).
\]
First, by \eqref{rrt-cty} we see that for each $\ff \in \sC^{\infty}(\fH)_t$ there exists $\mathbf{A}\ff \in \sC^{\infty}(T^*_B\tensor \fH)_t$ such that 
\[
\tfrac{1}{2}\mathring{\nabla}\fh (\ff, \fg) = \fh (\mathbf{A}\ff, \fg).
\]
Since $\overline{\mathring{\nabla} \fh (\ff, \fg)} = \mathring{\nabla} \fh (\fg, \ff)$, we see that if $\fg \in \sC^{\infty}(\fH)_t$ then 
\[
\left |(\mathring{\nabla} _{\xi} \fh) (\ff, \fg)) (t) \right |^2  \lesssim  \fh (\ff, \ff)(t)
\]
for every $\ff \in {\rm Domain}(\mathbf{A})$.  Thus $\fg \in {\rm Domain} (\mathbf{A} ^{\dagger})$, the Hilbert space adjoint of $\mathbf{A}$, and moreover 
\[
\fh (\mathbf{A} \ff, \fg) = \fh (\ff , \mathbf{A}^{\dagger}\fg) 
\]
for all $\ff, \fg \in \sC^{\infty} (\fH) _t$.  And we also have 
\[
\fh (\mathbf{A}\ff, \fg) = \mathring{\nabla} \fh (\ff, \fg) = \tfrac{1}{2} \overline{\mathring{\nabla} \fh (\fg, \ff)} = \tfrac{1}{2} \overline{\fh (\mathbf{A}\fg, \ff)} = \fh (\ff, \mathbf{A}\fg),
\]
so that in fact $\mathbf{A}$ is self-adjoint on its domain $\sC^{\infty} (\fH)_t$.  Therefore 
\[
 \fh (\mathbf{A}\ff, \fg) +  \fh (\ff, \mathbf{A}\fg) =  \fh (2\mathbf{A}\ff, \fg) = \mathring{\nabla}\fh (\ff, \fg),
\]
and the proof is complete.
\end{proof}

\begin{rmk}
Note that, conversely, the estimates \eqref{rrt-cty} on $\mathring{\nabla}$ and $\fh$ are necessary for the existence of $\mathbf{A}$.  Moreover, the $\fh$-Hermitian operator $\mathbf{A}$ must be tensorial.  There may be other solutions; but all of them must be of the form $\mathbf{A} + \mathbf{B}$ such that $\mathbf{B} ^{\dagger} = - \mathbf{B}$, with appropriately defined domains which we will not discuss here.
\red
\end{rmk}
 
 %\begin{rmk}
 %The reference connection $\mathring{\nabla}$ of a smooth quasi-Hilbert field should not be thought of as canonical.  However, for a fixed metric $\fh$ one is not completely free to choose a reference connection.  Indeed, if $\mathbf{A}: \sC^{\infty} (\fH) \to \sC^{\infty} (T^*_B \tensor \fH)$ is a $\sC^{\infty}$-linear map of sheaves and $\mathring{\nabla}$ is a reference connection satisfying \eqref{rrt-cty} then the connection $\mathring{\nabla} ^{\mathbf{A}}:= \mathring{\nabla} +\mathbf{A}$ satisfies 
 %\[
 %\mathring{\nabla} ^{\mathbf{A}} \fh (\ff, \fg) = \mathring{\nabla} \fh ( \ff, \fg) + \fh (\mathbf{A}\ff, \fg) + \fh (\ff, \mathbf{A}\fg) =  \mathring{\nabla} \fh ( \ff, \fg) + \fh (\mathbf{A}^S\ff, \fg) + \fh (\ff, \mathbf{A}^S\fg) 
 %\]
 %Hence $\mathring{\nabla} ^{\mathbf{A}}$ satisfies \eqref{rrt-cty} if and only if the symmetric part $\mathbf{A}^S$ of $\mathbf{A}$ is tensorial and the map of Hilbert fields associated to $\mathbf{A}^Sis bounded. 
%\red
% \end{rmk}

%An key notion for smooth Hilbert fields is that of smooth morphism of Hilbert fields.

\subsection{Holomorphic Quasi-Hilbert fields and BLS fields}

Since we are interested in the situation in which the base is a complex manifold, we consider the obvious complexification of the notion of connection: In Definition \ref{connection-defn} one allows the functions $f \in \sC^{\infty} (B)$ to be complex valued, and for a complex vector field $\xi \in \Gamma (B, \sC^{\infty} (T_B \tensor \bbC))$ one defines 
\[
\nabla _{\xi} := \nabla _{\re \xi} + \ii \nabla _{\im \xi}.
\]
These notions, which just use the natural complex structures of the targets, only require $B$ to be a smooth manifold.  But if $B$ is a complex manifold (or more generally an almost complex manifold) then there is a splitting $T_B \tensor \bbC = T^{1,0} _B \oplus T^{0,1}_B$.  Thus one has a decomposition 
\[
\nabla = \nabla ^{1,0} + \nabla ^{0,1}
\]
induced by the splitting 
\[
\Gamma (B, \sC^{\infty}(\Lambda ^{k+1} _B \tensor \fH))=\Gamma (B, \sC^{\infty}(\Lambda ^{1,0}_B \wedge \Lambda ^k_B \tensor \fH))\oplus \Gamma (B, \sC^{\infty}(\Lambda ^{0,1}_B \wedge \Lambda ^k _B \tensor \fH)),
\]
and one can talk about $\fH$-valued $(p,q)$-forms as sections of $\Lambda ^{p,q}_B \tensor \fH \to B$.  Such forms are smooth if they are smooth as $\fH$-valued $(p+q)$-forms.  For $0\le p,q\le \dim_{\bbC} B$ the set of smooth $\fH$-valued $(p,q)$-forms $\Gamma (B, \Lambda ^{p,q}_B \tensor \fH) \cap \Gamma (B, \sC^{\infty}(\Lambda ^{p+q} _B \tensor \fH))$ shall be denoted $\Gamma (B, \sC^{\infty}(\Lambda ^{p,q} _B \tensor \fH))$.

%Next we introduce the notion of a complex structure for a smooth quasi-Hilbert field.  

\begin{defn}\label{general-dbar-and-holo-defn}
Let $B$ be an almost complex manifold and let $\fH \to B$ be a smooth quasi-Hilbert field.  
\begin{enumerate}
\item[{\rm i.}] An almost complex structure for $\fH \to B$ is a family of linear maps of sheaves 
\[
\dbar : \sC^{\infty} (\Lambda ^{p,q} _B \tensor \fH) \to  \sC^{\infty} (\Lambda ^{p,q+1}_B \tensor \fH), \quad 0 \le p,q \le n, 
\]
compatible with the $\dbar$-operator on $B$ in the sense that 
\[
\dbar (f \wedge \fb) = \dbar f \wedge \fb +(-1) ^{p'+q'} f \wedge \dbar \fb, \quad f \in \sC^{\infty} (\Lambda ^{p',q'}_B).
\]
The almost complex structure is \emph{involutive} or \emph{complex} or \emph{integrable} if $\dbar \dbar = 0$.
\item[{\rm ii.}]  A smooth section $\ff$ of the smooth quasi-Hilbert field $\fH \to B$ with almost complex structure $\dbar$ is said to be \emph{holomorphic} on an open subset $U \subset B$ if $\dbar \ff = 0$ on $U$.  The set of holomorphic sections on $U$ is denoted $\Gamma (U, \cO (\fH))$.
\item[{\rm iii.}]  An \emph{(almost) holomorphic quasi-Hilbert field}\footnote{Since we are now working with complex connections, when we talk about a smooth Hilbert field, in Definition \ref{smooth-(q)h-fld-defn}.b we replace the vectors $\xi \in T_{B,t}$ with vectors vectors $\xi \in T_{B,t}\tensor \bbC$.} is a smooth quasi-Hilbert field with an (almost) complex structure.
\item[{\rm iv.}]  A \emph{Berndtsson-Lempert-Sz\H{o}ke} (BLS) field is an almost holomorphic smooth Hilbert field admitting a reference connection $\mathring{\nabla}$ such that $\mathring{\nabla}^{0,1} = \dbar$.
\item[{\rm v.}] A BLS field is said to be \emph{integrable} if its almost complex structure is integrable.
\end{enumerate}
\end{defn}

\begin{rmk}
If $\fH$ is an almost holomorphic smooth Hilbert field then $\fH$ is BLS if and only if for any reference connection $\mathring{\nabla}$ the map of sheaves $\mathring{\nabla}^{0,1}-\dbar$ is tensorial and the associated map of Hilbert fields is a morphism.
\red
\end{rmk}

\begin{prop}
On every BLS field $(\fH, \fh, \dbar) \to B$ there exists a unique connection $\nabla$, called the \emph{BLS Chern connection}, or simply the \emph{Chern connection}, such that  
\[
\nabla \fh = 0 \quad \text{and} \quad \nabla ^{0,1} = \dbar.
\]
\end{prop}

\begin{proof}
Let $\mathring{\nabla} = \mathring{\nabla}^{1,0} + \dbar$ be the reference connection of the underlying quasi-Hilbert field.  We need to find a linear map of sheaves $\mathbf{A} : \sC^{\infty} (\fH) \to \sC^{\infty} (T^{*1,0} _B \tensor \fH)$ such that the connection $\nabla ^{\mathbf{A}} := \mathring{\nabla} + \mathbf{A}$ is metric compatible, i.e., satisfies $\nabla _{\mathbf{A}}\fh= 0$.  Our assumption is that for each $\xi \in T_{B,t} \tensor \bbC$ and each $\ff \in \sC^{\infty} (\fH)$ the linear functional $(\mathring{\nabla} _{\xi} \fh)(\ff , \cdot)$ is bounded.  In particular, if we take $\xi \in T^{1,0} _{B,t}$, we find that the linear functional $(\mathring{\nabla} _{\xi} \fh)(\ff , \cdot)$ defined by 
\[
(\mathring{\nabla} _{\xi} \fh)(\ff , \fg) = \di _{\xi} (\fh (\ff,\fg)) - \fh (\mathring{\nabla} ^{1,0}_{\xi}  \ff, \fg) - \fh (\ff, \dbar _{\bar \xi} \fg)
\]
is bounded.  Hence there exists a linear map of sheaves $\mathbf{A} : \sC^{\infty} (\fH) \to \sC^{\infty} (T^{*1,0} _B \tensor \fH)$ such that
\[
\di _{\zeta} (\fh (\ff,\fg)) - \fh (\mathring{\nabla} ^{1,0}_{\zeta}  \ff, \fg) - \fh (\ff, \dbar _{\bar \zeta} \fg)= \fh (\left < \mathbf{A}, \zeta\right >\ff, \fg) \quad \text{ for all }\zeta \in T^{1,0} _{B,t}.
\]
Letting $\nabla := \mathring{\nabla}^{1,0} + \mathbf{A} + \dbar$, for $T_{B,t} \ni \xi = \eta + \bar \eta \in T_{B,t}^{1,0} \oplus T_{B,t}^{0,1} = T_{B,t} \tensor \bbC$ and $\ff, \fg \in \sC^{\infty} (\fH)_t$ we compute that 
\begin{eqnarray*}
\nabla_{\xi} \fh (\ff, \fg) &=& d_{\xi} (\fh (\ff,\fg)) - \fh (\nabla _{\xi}  \ff, \fg) - \fh (\ff,\nabla _{\xi}\fg)  \\
&=& \di _{\zeta} (\fh (\ff,\fg)) + \dbar _{\bar \zeta} (\fh (\ff,\fg)) - \fh (\nabla ^{1,0} _{\zeta}  \ff, \fg) - \fh (\dbar _{\bar \zeta}  \ff, \fg) - \fh (\ff,\nabla ^{1,0}_{\zeta} \fg) - \fh (\ff,\dbar_{\bar \zeta} \fg) \\
&=& \di _{\zeta} (\fh (\ff,\fg)) - \fh (\nabla ^{1,0} _{\zeta}  \ff, \fg)  - \fh (\ff,\dbar_{\bar \zeta} \fg)  + \overline{\di _{\zeta} (\fh (\fg,\ff)) - \fh (\nabla ^{1,0}_{\zeta} \fg, \ff) - \fh (\fg, \dbar _{\bar \zeta}  \ff)}  \\
&=& \di _{\zeta} (\fh (\ff,\fg)) - \fh (\mathring{\nabla} ^{1,0} _{\zeta}  \ff, \fg)  - \fh (\ff,\dbar_{\bar \zeta} \fg) - \fh \left  (\left < \mathbf{A}, \zeta \right > \ff, \fg\right )   \\
&&\qquad + \overline{\di _{\zeta} (\fh (\fg,\ff)) - \fh (\nabla ^{1,0}_{\zeta} \fg, \ff) - \fh (\fg, \dbar _{\bar \zeta}  \ff) - \fh \left  (\left < \mathbf{A}, \zeta \right > \fg, \ff\right )} \\
=0.
\end{eqnarray*}
Therefore $\nabla := \mathring{\nabla} + \mathbf{A}$ is the Chern connection.  Moreover, we have shown that 
\begin{equation}\label{equation-for-nabla-10}
\fh (\nabla^{1,0} \fb _1 , \fb _2) = \di (\fh (\fb _1, \fb _2)) - \fh (\fb _1, \dbar \fb _2)
\end{equation}
on all smooth $\fH$-valued forms.  Since the right hand side does not depend on the connection, we see that the Chern connection is unique.
\end{proof}

\begin{rmk}
In view of the uniqueness of the Chern connection, we shall often refer to the BLS field $(\fH, \fh)$ and omit explicit reference to the connection.  Often we even omit reference to the metric when the latter is clear from the context. 
\red
\end{rmk}

\begin{prop}\label{chern-curvature-is-11}
The curvature $\Theta (\fh)$ of the Chern connection of a BLS field $(\fH, \fh,\dbar) \to B$ is a smooth ${\rm Lin}(\fH, \fH)$-valued $2$-form that is $\fh$-anti-Hermitian.  Moreover, if $\dbar \dbar = 0$ then the curvature is a ${\rm Lin}(\fH, \fH)$-valued $(1,1)$-form.
\end{prop}

\begin{proof}
By Proposition \ref{metric=>tensorial} $\Theta (\fh)$ is tensorial.  If $\fb _1$ and $\fb _2$ are smooth $k$-forms then 
\begin{eqnarray*}
0 &=& dd \fh (\fb _1, \fb _2) = d \left ( \fh (\nabla \fb _1, \fb _2) + (-1) ^{k} \fh (\fb _1,\nabla \fb _2) \right )\\
&=& \fh (\Theta (\fh) \fb _1, \fb _2) + (-1)^{k+1} \fh (\nabla \fb _1 , \nabla \fb _2) + (-1) ^{k} \fh (\nabla \fb_1, \nabla \fb _2) + \fh (\fb _1 ,\Theta (\fh) \fb _2)\\
&=& \fh (\Theta (\fh) \fb _1, \fb _2) + \fh (\fb _1 ,\Theta (\fh) \fb _2),
\end{eqnarray*}
so every metric-compatible connection has anti-Hermitian curvature.  In particular, the $(0,2)$-component of a metric-compatible connection is the Hermitian conjugate of the $(2,0)$-component, so if also $\dbar \dbar = 0$ then $\nabla ^{1,0}\nabla ^{1,0}= 0$. 
\end{proof}

\subsection{Smooth (Quasi-)Hilbert Subfields}

\begin{defn} \label{smooth-subfield-defn}
Let $(\fL, \fh ^{\fL}) \to B$ be a smooth Hilbert field.  
\begin{enumerate}
\item[i.] A smooth Hilbert subfield of $(\fL, \fh ^{\fL})$ is a smooth Hilbert field $(\fH ,\fh ^{\fH})$ such that 
\begin{enumerate}
\item[a)] $\fH _t$ is a closed subspace of $\fL _t$ for every $t \in B$, 
\item[b)] $\sC^{\infty} (\fH) = \{ \ff \in  \sC ^{\infty} (\fL))\ ;\ \ff (t) \in \fH _t \text{ for all } t \in {\rm Domain}(\ff)\}$, and 
\item[c)] $\fh ^{\fL} |_{\fH} = \fh ^{\fH}$.
\end{enumerate}
\item[ii.] To any smooth Hilbert subfield $\fH \subset \fL$ there is associated the \emph{orthogonal projector}: the unique morphism $P = P _{\fH} : \fL \to \fH$ defined by fiberwise orthogonal projection.
\item[iii.] We say that the smooth Hilbert subfield $\fH \subset \fL$ is \emph{regular} if $P_{\fH} : \fL \to \fH$ is smooth.
\end{enumerate}
\end{defn}

\begin{rmk}
A smooth quasi-Hilbert subfield $\fH \subset \fL$ is a smooth quasi-Hilbert field that satisfies (i.a) and (i.b) of Definition \ref{smooth-subfield-defn}.
\red
\end{rmk} 

\begin{rmk}
To any smooth Hilbert subfield $\fH \subset \fL$ one can associate its orthogonal complement $\fH ^{\perp}$.  In general $\fH ^{\perp}$ need not be a smooth Hilbert subfield of $\fL$, but this is certainly the case if $\fH$ is regular.
\red
\end{rmk}

Let $\fL \to B$ be a smooth (quasi-)Hilbert and $\fH \subset \fL$ a smooth (quasi-)Hilbert subfield.  Suppose both fields are equipped with connections $\nabla ^{\fL}$ and $\nabla ^{\fH}$.  One can consider the map
\[
\mathbf{N} ^{\fL/\fH} := \nabla ^{\fL}|_{\fH} - \nabla ^{\fH} : \sC^{\infty} (\fH) \to \sC^{\infty} (\Lambda ^1 _B \tensor \fL),
\]
which depends of course on the choice of connections.  One computes that 
\[
\mathbf{N} ^{\fL/\fH} (f\ff)= \nabla ^{\fL} (f\ff) - \nabla ^{\fH} (f\ff) = df \tensor \ff + f \nabla ^{\fL} \ff - (df \tensor \ff +f \nabla ^{\fH} \ff) = f \mathbf{N} ^{\fL/\fH} \ff \quad \text{ for all }\ff \in \sC^{\infty} (\fH),
\]
which shows that $\mathbf{N} ^{\fL/\fH}$ is a linear map of sheaves.  

\begin{defn}\label{sff-bls-defn}
The linear map of sheaves $\mathbf{N} ^{\fL/\fH}$ is called the \emph{second fundamental map} of the (quasi)-Hilbert fields with connections.
\end{defn}

\begin{prop}\label{sff-form-prop}
If the connections $\nabla ^{\fL}$ and $\nabla ^{\fH}$ are metric-compatible then the map $\mathbf{N} ^{\fL/\fH}$ is tensorial.
\end{prop}

\begin{proof}
If $\ff , \fg \in \sC^{\infty} (\fH)_t$ and $\ff (t) = 0$ then 
\[
\fh (\mathbf{N} ^{\fL/\fH}\ff, \fg) = \fh ((\nabla ^{\fL} - \nabla ^{\fH})\ff, \fg)= \fh (\ff, (\nabla ^{\fL} - \nabla ^{\fH})\fg),
\]
and the right hand side vanishes at $t$.  Therefore by density $(\mathbf{N} ^{\fL/\fH}\ff)(t) = 0$.
\end{proof}

Therefore one has a densely defined map of Hilbert fields $\two^ {\fL / \fH} : \fH \to \Lambda ^1_B \tensor \fL$ such that 
\[
\mathbf{N}^{\fL/\fH} \ff = \two ^{\fL/\fH}\ff \quad \text{for all } \ff \in \sC^{\infty} (\fH).
\]

\begin{defn}
The map of Hilbert fields $\two ^{\fL/\fH}$  is called the \emph{second fundamental form} of the Hilbert fields with metric-compatible connections $(\fL, \nabla ^{\fL})$ and $(\fH, \nabla ^{\fH})$.
\end{defn}

If $\fH \subset \fL$ is a regular Hilbert subfield then any connection $\nabla ^{\fL}$ for $\fL$ defines a connection $\nabla ^{\fH} := P_{\fH} \nabla ^{\fL}$.  In this case, $\mathbf{N}^{\fL/\fH} = P_{\fH} ^{\perp} \nabla ^{\fL}$.  And if $\nabla ^{\fL}$ is metric-compatible then for any $\ff, \fg \in \sC^{\infty} (\fH)$ one has 
\[
\fh ( \nabla ^{\fH} \ff, \fg) = \fh ( \nabla ^{\fL} \ff, \fg) = d (\fh (\ff, \fg)) - \fh (\ff, \nabla ^{\fL} \fg) = d (\fh (\ff, \fg)) -  \fh (\ff, \nabla ^{\fH} \fg),
\]
so that $\nabla ^{\fH}$ is also metric compatible.  Therefore $\mathbf{N}^{\fL/\fH}$ is tensorial, and one has a second fundamental form $\two^ {\fL / \fH}$, which moreover takes its values in $\Lambda ^1_B \tensor \fH ^{\perp}$.

\begin{rmk}
In the case of our main examples $\sH \subset \sL$ discussed in the introduction, if the Hilbert field $\sH$ is smooth then it is a regular smooth Hilbert subfield of $\sL$.  This fact is a special case of a result of Kodaira and Spencer, to the effect that smoothness of the harmonic projections is equivalent to constancy of the dimensions of the harmonic spaces.  The latter in turn means that the harmonic spaces fit together to form a vector bundle, and in the case of harmonic $(p,0)$-forms, which are automatically holomorphic, this vector bundle is even holomorphic.  Thus for $\sH$, which is the special case $p=n$, Kodaira-Spencer tells us that if $\sH$ is smooth then $\sH$ is a holomorphic vector bundle, and then the projection $P _{\sH}$ is smooth.  
\red
\end{rmk}

\subsection{BLS Subfields and the Gauss-Griffiths Formula}

\begin{defn}\label{bls-subfld-defn}
Let $\fL \to B, \dbar ^{\fL}$ be an almost holomorphic quasi-Hilbert field.
\begin{enumerate}
\item[i.]  An almost holomorphic quasi-Hilbert field $\fH \to B, \dbar ^{\fH}$ is said to be a \emph{subfield} of $\fL$ if
\begin{enumerate}
\item[a)] $\fH_t$ is a closed subspace of $\fL_t$ for every $t \in B$,
\item[b)] $\sC^{\infty} (\fH) = \{ \ff \in  \sC ^{\infty} (\fL))\ ;\ \ff (t) \in \fH _t \text{ for all } t \in {\rm Domain}(\ff)\}$, and 
\item[c)] $\dbar ^{\fL} |_{\sC^{\infty}(\fH)} = \dbar ^{\fH}$.
\end{enumerate}
\item[ii.] A subfield $\fH \to B$ of $\fL \to B$ is said to be a \emph{BLS subfield} of $\fL$ if $P_{\fH} \sC^{\infty}(\fL) \subset \sC^{\infty} (\fH)$.
\end{enumerate}
\end{defn}

\begin{s-rmk}
The inclusions $P_{\fH} \sC^{\infty}(\fL) \subset \sC^{\infty} (\fH) \subset \sC^{\infty} (\fL)$ imply that $\fH \subset \fL$ is a regular subfield and that $P_{\fH} \sC^{\infty}(\fL) =\sC^{\infty} (\fH)$.  In particular, a smooth section of $\fL$ that takes all its values in $\fH$ is a smooth section of $\fH$.

Conversely, if $\fL$ is a BLS field and $\fH \subset \fL$ is a regular subfield such that the subspace $\{ (P_{\fH} \ff)(p) \ ;\ \ff \in \sC^{\infty}(\fL)_p\} \subset \fH _p$ is dense for every $p \in B$ then $\fH \subset \fL$ is a BLS subfield.
\red
\end{s-rmk}

\begin{prop}\label{subbundle-chern-connection-prop}
If $\fH \subset \fL$ is a BLS subfield then the BLS Chern connection for $\fH \to B$ is $\nabla ^{\fH} = P_{\fH} \nabla ^{\fL} |_{\fH}$, where $\nabla ^{\fL} |_{\fH}$ denotes the restriction of the connection $\nabla ^{\fL}$ to sections of $\fH$.
\end{prop}

\begin{proof}
If $\ff _1, \ff _2 \in \Gamma (B,\fH)$ then 
\begin{eqnarray*}
\di \fh ^{\fH}(\ff _1, \ff _2) - \fh ^{\fH} ( \ff _1, \dbar ^{\fH} \ff _2) &=& \di \fh ^{\fL}(\ff _1, \ff _2) - \fh ^{\fL} ( \ff _1, \dbar ^{\fL} \ff _2)\\
&=& \fh ^{\fL} (\nabla ^{\fL 1,0}\ff _1, \ff _2)  = \fh ^{\fL} (P_{\fH}  \nabla ^{\fL 1,0}\ff _1, \ff _2) = \fh ^{\fH} (P_{\fH}  \nabla ^{\fL 1,0}\ff _1, \ff _2), 
\end{eqnarray*}
and since \eqref{equation-for-nabla-10} uniquely determines the Chern connection, the result follows.  
\end{proof}

By property (c) of Definition \ref{bls-subfld-defn} one also sees that 
\[
\mathbf{N}^{\fL /\fH} = P^{\perp}_{\fH} \nabla ^{\fL 1,0}|_{\fH}.
\]
Consequently the second fundamental map $\mathbf{N}^{\fL / \fH}$ of a BLS subfield $\fH \subset \fL$ is tensorial, and its associated second fundamental form $\two^{\fL/\fH}$ takes values in $\Lambda ^{1,0}_B \tensor \fH ^{\perp}$.

\begin{prop}[Gauss-Griffiths Formula]\label{gg-formula}
Let $\fL$ be a BLS field and let $\fH \subset \fL$ be a BLS subfield.
\begin{enumerate}
\item[{\rm a.}]  For any smooth sections $\ff_1, \ff_2 \in \Gamma (B, \sC^{\infty}(\fH))$ and any $(1,0)$-vectors $\xi , \eta \in T^{1,0}_{B, t}$, $t \in B$ one has 
\begin{equation}\label{gauss-griffiths-bls-eval}
\fh(\Theta (\nabla ^{\fL}) (\xi, \bar \eta) \ff_1 , \ff_2) = \fh(\Theta (\nabla ^{\fH})(\xi, \bar \eta)\ff_1, \ff_2) + \fh  (\two ^{\fL/\fH}(\xi) \ff_1 , \two ^{\fL/\fH}(\eta) \ff_2)
\end{equation}
\item[{\rm b.}] If moreover the quadratic form
\[
\bbG _{\xi \bar \eta} (f_1, f_2) := \fh  (\two ^{\fL/\fH}(\xi) f_1 , \two ^{\fL/\fH}(\eta) f_2)
\]
defines a bounded conjugate-linear functional $\bbG _{\xi, \bar \eta} (f, \cdot )$ for every $t \in B$, every $f \in \fH_t$ and every pair of $(1,0)$-vectors $\xi , \eta \in T^{1,0}_{B, t}$, then 
\begin{equation}\label{gauss-griffiths-bls-forms}
\Theta (\nabla ^{\fL})|_{\fH} = \Theta (\nabla ^{\fH}) - (\two ^{\fL/ \fH})^{\dagger} \wedge \two ^{\fL/ \fH}.
\end{equation}
\end{enumerate}
\end{prop}

\begin{proof}
Fix $\ff_1 , \ff_2 \in \Gamma (B, \sC^{\infty}(\fH))$.  Then by orthogonality 
\begin{eqnarray*}
\fh ( \Theta (\nabla ^{\fL}) \ff_1 , \ff_2)&=& \fh \left (((P_\fH \nabla ^{\fL 1,0} + P^{\perp} _{\fH} \nabla ^{\fL1,0}) \dbar + \dbar (P_\fH \nabla ^{\fL 1,0} + P^{\perp} _{\fH} \nabla ^{\fL1,0}) )\ff_1 , \ff_2\right )\\
&=& \fh ((\nabla ^{\fH 1,0} \dbar + \dbar \nabla ^{\fH 1,0})\ff_1 , \ff_2)  + \fh  (\two ^{\fL/\fH} \dbar + \dbar \two ^{\fL/\fH})\ff_1, \ff_2)\\
&=& \fh ((\nabla ^{\fH 1,0} \dbar + \dbar \nabla ^{\fH 1,0})\ff_1 , \ff_2)  + \fh  (\dbar (\two ^{\fL/\fH}\ff_1), \ff_2).
\end{eqnarray*}
But  
\[
\dbar \fh (\two ^{\fL/\fH} \ff_1 ,\ff_2 ) = \fh (\dbar (\two ^{\fL/\fH} \ff_1) ,\ff_2 ) -  \fh (\two ^{\fL/\fH} \ff_1 ,\nabla ^{\fL 1,0}\ff_2 )= \fh (\dbar (\two ^{\fL/\fH} \ff_1) ,\ff_2 ) -  \fh (\two ^{\fL/\fH} \ff_1 ,P_{\fH} ^{\perp} \nabla ^{\fL 1,0}\ff_2 ),
\]
and since $\fh (\two ^{\fL/\fH} \ff_1 ,\ff_2 ) = 0$, 
\[
\fh (\dbar \two ^{\fL/\fH} \ff_1 ,\ff_2 ) =  \fh (\two ^{\fL/\fH} \ff_1 ,\two ^{\fL/\fH}\ff_2 ).
\]
Thus Formula \eqref{gauss-griffiths-bls-eval} is proved.   

To obtain Formula \eqref{gauss-griffiths-bls-forms} from Formula \eqref{gauss-griffiths-bls-eval} we need to show that the domain of the Hilbert space adjoint of the densely defined operator $\two ^{\fL/\fH}(\eta)$ contains the image of $\two ^{\fL/\fH}(\xi)$.  Now, the domain of $\two ^{\fL/\fH}(\eta)$ consists of all $\gamma \in \fH_t^{\perp}$ such that 
\[
\left | \fh (\gamma, \two ^{\fL/\fH}(\eta)g)\right |^2 \le C_{\gamma,\eta} \fh (g,g) \quad \text{for all }g \in (\sC^{\infty} (\fH) _t) (t).
\]
Since by hypothesis
\[
\left | \bbG _{\xi\bar \eta}( f,g)\right |^2  = \left | \fh (\two ^{\fL/\fH}(\xi)f, \two ^{\fL/\fH}(\eta)g) \right |^2 \lesssim \fh (g,g),
\]
we see that for each $f \in \fH _t$ 
\[
\two ^{\fL/\fH}(\xi)f \in {\rm Domain}( \two ^{\fL/\fH}(\eta)^{\dagger}) \quad \text{and} \quad \fh ( \two ^{\fL/\fH}(\eta)^{\dagger}  \two ^{\fL/\fH}(\xi) f , g) = \fh ( \two ^{\fL/\fH}(\xi) f,  \two ^{\fL/\fH}(\eta)g).
\]
Therefore for each $f \in \fH _t$
\[
\left <  (\two ^{\fL/\fH})^{\dagger} \wedge  \two ^{\fL/\fH} f , \xi \wedge \bar \eta \right > = -  \two ^{\fL/\fH}(\eta)^{\dagger} \two ^{\fL/\fH}(\xi)f,
\]
which shows that \eqref{gauss-griffiths-bls-eval} implies \eqref{gauss-griffiths-bls-forms}.  The proof is complete.
\end{proof}

\begin{rmk}\label{curvature-degree-rmk}
Note that if $\fH \to B$ has an \emph{integrable} structure, i.e., $\dbar ^{\fH}\dbar ^{\fH}= 0$, then the right hand side of \eqref{gauss-griffiths-bls-forms} is a $(1,1)$-form.  Thus, even if $\dbar ^{\fL} \dbar ^{\fL} \neq 0$, the restriction of the curvature of $\fL$ to an integrable BLS subbundle is a $(1,1)$-form.

In general, the complex structure of $\fH$ need not be integrable.  In that case, Proposition \ref{gg-formula} computes only the $(1,1)$-part of the curvature operator of $\fH$.  It is not hard to extend Proposition \ref{gg-formula} to compute the full curvature form of $\fH$.  The part of the formula coming from the form $\bbG$ remains of type $(1,1)$, and the type $(2,0)$ and $(0,2)$ parts come from the curvature form of $\fL$.

In our application, in which $\fH$ is the Hilbert field $\sH$ defined in the introduction, the complex structure of the subfield is always integrable.
\red
\end{rmk} 

\begin{defn}
Let $\fL$ be a BLS field and let $\fH \subset \fL$ be a BLS subfield.
\begin{enumerate}
\item[a.] The quadratic form $\bbG$ defined in Proposition \ref{gg-formula}.b is called the \emph{Gauss-Griffiths form} of the BLS subfield $\fH \subset \fL$.
\item[b.] If the Gauss form satisfies the continuity condition of Proposition \ref{gg-formula}.b then we say that $\fH \subset \fL$ is a \emph{strongly BLS subfield}.
\end{enumerate}
\end{defn}

\subsection{iBLS Fields}\label{iBLS-paragraph}

If $\fL$ is a BLS field and $\fH$ is a BLS subfield then Formula \eqref{gauss-griffiths-bls-forms} is an identity of Hilbert field maps.  Equivalently, Formula \eqref{gauss-griffiths-bls-eval} depends only on the value of the sections $\ff_1,\ff_2$ at a point, and not on higher derivatives of these sections, provided of course that $\ff_1,\ff_2 \in \sC^{\infty}(\fH)_t$.  Similarly, the second fundamental form of a BLS subfield is a pointwise object.

If $\fL \to B$ is a BLS field and $\fH  \subset \fL$ is a Hilbert subfield that is not necessarily BLS, one might try to use the Gauss-Griffiths Formula \eqref{gauss-griffiths-bls-forms} to define the curvature of $\fH \to B$.  There are several issues that arise:
\begin{enumerate}
\item[1.] For the resulting curvature of $\fH$ to be a densely-defined operator, one must know that the set of germs of smooth sections of $\fL$ at $t \in B$ whose value at $t$ lies in $\fH _t$ is dense in $\fH_t$.   
\item[2.] One must define the second fundamental form of $\fH$ in $\fL$ in a manner that does not require a connection for $\fH$.  
\item[3.]\label{third-curv-issue}  If one wishes to use Formula \eqref{gauss-griffiths-bls-forms} to define the curvature of $\fH$ as a tensorial operator then the second fundamental form defined in 2 must be tensorial, as must the wedge product of  the second fundamental form and its Hilbert space adjoint, i.e., the object corresponding to the second quantity on the right hand side of \eqref{gg-formula} must be well-defined.
\item[4.]  Ideally the curvature of $\fH$ should be independent of the ambient BLS field $\fL$. 
\end{enumerate}
We shall introduce a geometric theory of iBLS fields that resolves the first three issues.  Unfortunately, we are unable at present to resolve the fourth issue at the abstract level.  However, in the concrete setting of Berndtsson's iBLS subfield $\sH$, the large collection of ambient BLS fields $\sL^{\theta}$ that we define (in fact, we could not come up with any other examples of ambient BLS fields for $\sH$) all yield the same curvature operator.

\subsubsection*{\sf Substalk bundles}\

\noi For each $t \in B$ let 
\begin{equation}\label{restriction-subspace}
\tilde \fH _{\fL,t} := \{\ff \in \sC^{\infty} (\fL) _t\text{ and } \ff (t) \in \fH_t\} \subset \sC^{\infty} (\fL) _t. 
\end{equation}
The space $\tilde \fH _{\fL,t}$ is a $\bbC$-vector subspace of the stalk $\sC^{\infty} (\fL)_t$.  Define the family $\tilde \fH _{\fL} \to B$ whose fiber over $t$ is $\tilde \fH _{\fL,t}$.  

\begin{s-rmk}
The `family' notation is only introduced for convenience; we give this family no additional structure.  In particular, we note that $\tilde \fH _{\fL}$ is not a sheaf, since in general there is no notion of a smooth local section of $\tilde \fH _{\fL}$.

However, we \emph{can} define the notion of a smooth section of $\tilde \fH _{\fL}$ over $U \subset B$; it is a section $\ff \in \Gamma (U,\sC^{\infty} (\fL))$  such that $\ff (t) \in \fH _t$ for all $t \in U$.
\red
\end{s-rmk}  

\begin{defn}
Let $\fL \to B$ be a smooth quasi-Hilbert field.  We say that $\fH \subset \fL$ is a \emph{bona fide} quasi-Hilbert subfield if for each $t \in B$ the space $\tilde \fH _{\fL, t}(t) := \{ \ff (t)\ ;\ \ff \in \tilde \fH _{\fL, t}\}$ is a dense subspace of $\fH _t$. 
\end{defn}

Bona fide subfields resolve Issue 1 above.  

\begin{defn}\label{substalk-defns}
Let $\fL \to B$ be a BLS field and let $\fH \subset \fL$ be a (not necessarily smooth) Hilbert subfield.  
\begin{enumerate}
\item A substalk bundle $\cS$ of $\sC^{\infty} (\fL)$ is a choice of $\bbC$-vector subspace $\cS_t \subset \sC^{\infty} (\fL)_t$ for each $t \in B$; we write $\cS \subset \tilde \fH_{\fL}$.  We say that $\cS$ is a \emph{tuning} of $\fH$ (in $\fL$), or that $(\fH, \cS)$ is a \emph{tuned} Hilbert subfield.
\item The family $\tilde \fH _{\fL} \to B$   of subspaces of the stalks of $\sC^{\infty} (\fL)$ defined by \eqref{restriction-subspace} is called the maximal substalk bundle of $\fH$.
\item A substalk bundle $\Sigma \subset \tilde \fH_{\fL}$ is said to \emph{formalize $\fH$} if 
\begin{enumerate}
\item $\Sigma _t \subset \tilde \fH _{\fL, t}$ is a dense subspace for every $t$, and 
\item for every $t \in B$, if $\ff \in \Sigma _t$ and $\ff (t) = 0$ then $(\nabla ^{\fL} \ff )(t) \in \fH _t$.
\end{enumerate}
In this case we say that $\Sigma$ is a \emph{formal} tuning, or that $(\fH , \Sigma)$ is formally tuned.
\end{enumerate}
\end{defn}

\begin{ex}\label{BLS-is-iBLS}
If $\fH$ is a BLS subfield of a BLS field $\fL$ then the tuning $\sC^{\infty}(\fH)$ of $\fH$ is formal.  In fact, Property i of Definition \ref{substalk-defns}.c holds because density is part of the BLS hypothesis, and the proof of Proposition \ref{sff-form-prop} gives Property ii.  
\red
\end{ex}

As Example \ref{BLS-is-iBLS} suggests, tunings of $\fH$ in $\fL$ are meant to be `non-integrable' replacements for the non-existent smooth structure $\sC^{\infty} (\fH)$ of $\fH$.

\subsubsection*{\sf Infinitesimal Second Fundamental Form}

\begin{defn}
Let $\fL$ be a BLS field and let $(\fH,\cS)$ be a tuned quasi-Hilbert subfield.  The map 
\[
\mathbf{N} ^{\fL/ (\fH, \cS)} := P_{\fH} ^{\perp} \nabla ^{\fL}|_{\cS}
\]
is called the second fundamental map (of $(\fH, \cS)$ in $\fL$).
\end{defn}

In some sense, $\mathbf{N} ^{\fL / (\fH,\cS)}$ is $\sC^{\infty} _B$-linear.  Indeed, for $f \in \sC^{\infty} _{B,t}$ and $\ff \in \cS$ one has 
\[
P^{\perp} _{\fH} \nabla ( f \ff ) = P^{\perp} _{\fH} (df \tensor \ff + f \nabla \ff ) = fP^{\perp} _{\fH} \nabla \ff.
\]
However, the projection $P^{\perp} _{\fH}$ need not be smooth, so one has to understand this linearity in an appropriate sense.

Obviously $(\mathbf{N} ^{\fL/(\fH, \tilde \fH _{\fL})} \ff)(t) \in \fH _t ^{\perp}$.  However, $\tilde \fH _{\fL}$ is \emph{never} a formal tuning of $\fH$, so the map $\mathbf{N} ^{\fL/(\fH, \tilde \fH _{\fL})}$ is never tensorial\footnote{We are slightly abusing language here, since there is no sheaf $\tilde \fH _{\fL}$ whose stalk at every $t \in B$ is $\tilde \fH _{\fL, t}$.  However, $\tilde \fH _{\fL, t}$ is a $\sC^{\infty}_{B,t}$-module, so $\sC^{\infty}$-linear means $\sC^{\infty}_{B,t}$-linear, and by \emph{tensorial} we mean that if $\ff \in \tilde \fH _{\fL, t}$ and $\ff (t) = 0$ then $(\mathbf{N} ^{\fL/\fH}\ff )(t) = 0$.}; not even if $\fH \subset \fL$ is a BLS subfield.  The latter claim might seem confusing, in view of Proposition \ref{sff-form-prop}.  However,  the second fundamental map $\mathbf{N} ^{\fL/\fH}$ in Proposition \ref{sff-form-prop} corresponds here to the map $\mathbf{N}^{\fL/ (\fH, \sC^{\infty}(\fh))}$, and there is no analogue of the sheaf $\sC^{\infty}(\fH)$ in the present context, since we are not assuming that $\fH$ is a smooth quasi-Hilbert field.

\begin{prop}\label{iBLS-sff-prop}
If $\fL$ is a BLS field and $(\fH, \Sigma)$ is a formally tuned Hilbert subfield then there exists a $(1,0)$-form $\two ^{\fL/(\fH, \Sigma)}$ with values in ${\rm Lin}(\tilde \fH _{\fL}, \fH ^{\perp})$ such that 
\begin{equation}\label{sff-form-prop-iBLS}
(\mathbf{N}^{\fL /(\fH, \Sigma)}\ff )(t) = \two ^{\fL/(\fH, \Sigma)}(t) \ff (t)
\end{equation}
for all $t \in B$ and all $\ff \in \Sigma _{t}$.
\end{prop}

\begin{proof}
The formality of $\Sigma$ implies that the linear map of sheaves $\mathbf{N}^{\fL /(\fH, \Sigma)}$ is formal, and thus $\two ^{\fL/(\fH,\Sigma)}(t)$ is well-defined by \eqref{sff-form-prop-iBLS}.
\end{proof}

\begin{rmk}
The name \emph{formal} is thus revealed to be a play on words: if $(\fH, \Sigma)$ is formal then its second fundamental map $\mathbf{N}^{\fL /(\fH, \Sigma)}$ is tensorial, and hence defined by a \emph{form}.
\red
\end{rmk}

\subsubsection*{\sf Infinitesimal (almost) holomorphic subfield}

\begin{defn}
Let $\fL \to B$ be a BLS field, let $\fH \subset \fL$ be a quasi-Hilbert subfield and let $\cS$ be a tuning of $\fH$, i.e., a substalk bundle of $\tilde \fH _{\fL}$.  We say that $(\fH, \cS)$ is an \emph{infinitesimally almost holomorphic subfield} of $\fL$ if for each $t \in B$, each $\tau \in T^{1,0} _{B,t}$ and each $\ff \in \tilde \cS_t$ one has 
\[
\left < \dbar ^{\fL} \ff ,   \bar \tau \right >(t)  \in \fH _t.
\]
If moreover 
\[
\left < \dbar ^{\fL} \dbar ^{\fL} \ff ,   \bar \sigma  \wedge \bar \tau \right > (t)= 0
\]
for each $t \in B$, all $\sigma, \tau \in T^{1,0} _{B,t}$ and all $\ff \in \tilde \cS_t$ then we say that $(\fH, \cS)$ is an \emph{infinitesimally holomorphic subfield} of $\fL$.
\end{defn}

%\begin{rmk}
%The idea to introduce the notion of infinitesimal holomorphic subfield is due to Upadrashta.  We thank him for allowing us to include it here.
%\red
%\end{rmk}

\begin{rmk}
In practice, given a Hilbert subfield $\fH$ of a BLS field $\fL$, one finds that there are always sections $\ff \in \tilde \fH _{\fL, t}$ such that $(\dbar ^{\fL} \ff )(t) \not \in T^{0,1} _{B,t} \tensor \fH _t$.  Thus any useful notion of almost holomorphic subfield must involve a choice of tuning.
\red
\end{rmk}

%\begin{rmk}
%Even if the ambient BLS field $\fL$ is only \emph{almost} holomorphic, a tuned subfield $\fH \subset \fL$ might well be holomorphic.  Again, this happens with the Berndtsson Hilbert field $\sH$, as we shall see.
%\end{rmk}

If the quasi-Hilbert subfield $\fH \subset \fL$ is moreover infinitesimally almost holomorphic, which is the case of primary interest for us, then one trivially has the following proposition.

\begin{prop}\label{infinitesimal-abstract-sff-formula}
Let $\fL$ be a BLS field and let $(\fH,\cS) \subset \fL$ be an infinitesimally almost holomorphic subfield.  Then
\[
\mathbf{N}^{\fL /(\fH,\cS)} = P^{\perp}_{\fH} \nabla ^{\fL 1,0}|_{\cS}.
\]
\end{prop}

\subsubsection*{\sf iBLS fields}\

\bigskip

\noi Propositions \ref{iBLS-sff-prop} and \ref{infinitesimal-abstract-sff-formula} take care of the first aspect of issue number 3 discussed on page \pageref{third-curv-issue}.  The second issue, namely that of showing that $(\mathbf{N}^{\fL / \fH}) ^{\dagger} \wedge\mathbf{N}^{\fL/\fH}$ is well-defined and tensorial, is better handled by considering the quadratic form \eqref{gauss-griffiths-bls-forms} rather than the formula \eqref{gauss-griffiths-bls-eval} in Proposition \ref{gg-formula}.  

\begin{defn}
Let $\fL$ be a BLS field, let $\fH \subset \fL$ be a bona fide quasi-Hilbert subfield and let $\Sigma$ be a formal tuning of $\fH$.  The \emph{Gauss-Griffiths form} of $(\fL, \fH, \Sigma)$ is the sesquilinear form $\bbG : \Sigma \times \Sigma \to \sC^{\infty} (\Lambda ^2 _B)$ defined by 
\[
\bbG (\ff, \fg) := \fh \left ( \mathbf{\two} ^{\fL/ (\fH, \Sigma)}\ff, \mathbf{\two} ^{\fL/ (\fH, \Sigma)}\fg \right ).
\]
\end{defn}
Note that the Gauss-Griffiths form applied to $\ff, \fg$ is a complex $2$-form, and its contraction with two vectors $\sigma , \tau \in T_B \tensor \bbC$ is given by 
\[
\bbG_{\sigma \bar \tau}(\ff,\fg) = \fh \left ( \mathbf{\two} ^{\fL/ (\fH, \Sigma)}(\sigma) \ff, \mathbf{\two} ^{\fL/ (\fH, \Sigma)}(\tau)\fg \right ).
\]
Just as we required the continuity \eqref{rrt-cty} of $\mathring{\nabla}\fh$ in the setting of smooth Hilbert fields in order to guarantee that such Hilbert fields have a metric-compatible connection (c.f. the proof of Proposition \ref{metric-compatible-connection-prop}), if we want the curvature of a bona-fide, formally tuned quasi-Hilbert subfield of a BLS field to be given by a map of sheaves then we must require the analogous continuity of the Gauss-Griffiths form.

Finally, we can define \emph{infinitesimal BLS} fields, or iBLS fields, as follows.

\begin{defn}\label{iBLS-defn}
Let $\fH \to B$ be a quasi-Hilbert field.  
\begin{enumerate}
\item[{\rm i.}] An iBLS field is a triple $(\fH, \fL, \Sigma)$ where $\fL \to B$ is a BLS field, called the \emph{ambient} BLS field (for $\fH$), and $\Sigma \subset \tilde \fH _{\fL}$ is a substalk bundle, called the \emph{formalizing} substalk bundle, such that 
\begin{enumerate}
\item[(a)] $\fH \subset \fL$ is a bona fide subfield, i.e., $\tilde \fH _{\fL, t}(t)$ is dense in $\fH _t$ for every $t \in B$,
\item[(b)] $(\fH, \Sigma)$ is an infinitesimally almost holomorphic Hilbert subfield of $\fL$, and 
\item[(c)] $(\fH, \Sigma) \subset \fL$ is formally tuned.
\end{enumerate}
\item[{\rm ii.}] If in addition to (a), (b) and (c) one has that 
\begin{enumerate} 
\item[(d)] for each $t \in B$, every $\sigma , \tau \in T_{B, t}\tensor \bbC$ and each $f \in \Sigma _t(t)$ the conjugate-linear functional $\bbG _{\sigma \bar \tau}(\ff (t), \cdot) _t$ is bounded on $\Sigma _t (t)$ then we say that $(\fH, \fL, \Sigma)$ is \emph{strongly iBLS}.
\end{enumerate}
\item[iii.] And if moreover $(\fH, \Sigma)$ is an infinitesimally holomorphic subfield of $\fL$ then we shall say that $(\fH, \fL, \Sigma)$ is an \emph{integrable} (strongly) iBLS field.
\end{enumerate}
\end{defn}

We shall often abuse notation by repressing the triple $(\fH, \fL, \Sigma)$ and referring only to $\fH$ as an iBLS field when the rest of the data is clear from the context.

\begin{prop}\label{Gauss-form-is-form}
Let $(\fH, \fL, \Sigma)$ be an iBLS field.  Then there exists a unique linear map of substalk bundles $\mathbf{A}: \Sigma \to \Lambda ^2 _B \tensor \Sigma$ such that 
\[
\bbG _{\sigma\bar \tau} (\ff, \fg) = \fh \left (\left < \mathbf{A}, \sigma \wedge \bar \tau\right > \ff, \fg\right )= - \fh \left (\ff, \left < \mathbf{A}, \sigma \wedge \bar \tau\right > \fg\right ).
\]
\end{prop}
\begin{proof}
The map $\mathbf{A}$ defined by the Riesz Representation Theorem satisfies the first equality.  Since $\overline{\bbG _{\sigma \bar \tau}(\ff ,\fg)}= \bbG _{\tau \bar \sigma}(\fg, \ff)$, we have
\[
\bbG _{\sigma\bar \tau} (\ff, \fg) = \overline{\bbG _{\tau \bar \sigma} (\fg, \ff)} = \overline{\fh \left (\left < \mathbf{A}, \bar \tau \wedge \sigma\right > \fg, \ff\right )} = - \overline{\fh \left (\left < \mathbf{A}, \sigma \wedge \bar \tau\right > \fg, \ff\right )} = - \fh \left (\ff, \left < \mathbf{A}, \sigma \wedge \bar \tau\right > \fg\right ).
\]
The proof is complete.
\end{proof}

In analogy with the Gauss-Griffiths formula, we define the operator $\two ^{\fL/(\fH, \Sigma)})^{\dagger} \wedge \two ^{\fL/ (\fH, \Sigma)}$ by 
\begin{eqnarray*}
\left < \mathbf{A} , \sigma \wedge \bar \tau \right > &=:& \left < (\two ^{\fL/(\fH, \Sigma)})^{\dagger} \wedge \two ^{\fL/ (\fH, \Sigma)}, \sigma \wedge \bar \tau \right > \\
&=& \left < (\two ^{\fL/(\fH, \Sigma)})^{\dagger} , \bar \tau \right > \left <\two ^{\fL/ (\fH, \Sigma)} , \sigma \right > - \left <\two ^{\fL/ (\fH, \Sigma)} , \sigma \right > \left < (\two ^{\fL/(\fH, \Sigma)})^{\dagger} , \bar \tau \right >.
\end{eqnarray*}

\begin{defn}\label{sff-and-curv-iBLS-defn}
The \emph{curvature} of an iBLS field $(\fH, \fL, \Sigma)$ is the $2$-form
\[
\Theta ^{\text{(rel  $\fL$)}}(\fH) := \Theta (\nabla ^{\fL})|_{\fH} + (\two ^{\fL/ (\fH,\Sigma)})^{\dagger} \wedge \two ^{\fL/ (\fH,\Sigma)}.
\]
Equivalently, for any $t \in B$ 
\[
\fh(\Theta^{\text{(rel  $\fL$)}}(\fH)(\xi, \eta)\ff_1, \ff_2) =  \fh(\Theta (\nabla ^{\fL}) (\xi, \eta) \ff_1 , \ff_2) -  \fh  (\two^{\fL/(\fH,\Sigma)}(\xi) \ff_1 , \two^{\fL/(\fH,\Sigma)}(\eta) \ff_2)
\]
for all $\ff _1,\ff _2 \in \sC^{\infty}(\fH)_t$ and all $(1,0)$-vectors $\xi , \eta \in \bbC \tensor T_{B_t}$.

\end{defn}

\noi (We omit reference to $\Sigma$ in $\Theta ^{(\text{rel $\fL$})}$, but the curvature does depend on the choice of $\Sigma$.)

\begin{prop}\label{H-integrable=>L-11}
If $\fH$ is an iBLS subfield of a BLS field $\fL$ and $\fH$ is integrable, i.e., $(\dbar^2\ff )(t)= 0$ for every $f \in \Sigma _t$, then $\Theta (\nabla ^{\fL})|_{\fH}$ has $(1,1)$-coefficients.  More precisely, for each $f_1,f_2 \in \Sigma_t(t)$ and each $\xi , \eta \in T^{1,0}_{B, t}$ one has $\fh \left ( \Theta (\nabla ^{\fL}) (\xi, \eta) f_1,f_2\right ) = \fh \left ( \Theta (\nabla ^{\fL}) (\bar \xi, \bar \eta) f_1,f_2\right ) = 0$.  Consequently, if $\fH$ is integrable then $\Theta^{({\rm rel}\ \fL)} (\fH)$ has $(1,1)$-coefficients, i.e., 
\[
\Theta^{({\rm rel}\ \fL)} (\fH) (\xi, \eta)= \Theta^{({\rm rel}\ \fL)} (\fH) (\bar \xi, \bar \eta)= 0
\]
for all $\xi , \eta \in T^{1,0}_{B, t}$.
\end{prop}

\begin{proof}
The statement about $\Theta^{({\rm rel}\ \fL)} (\fH)$ follows from the statement about $\Theta (\nabla ^{\fL})|_{\fH}$ and the fact that $\two ^{\sL /\fH}$ has $(1,0)$-form coefficients.  Finally, since $\Theta (\nabla ^{\fL})$ is anti-Hermitian, it suffices to show that $\Theta (\nabla ^{\fL})^{0,2} \ff = 0$ for any $\ff \in \fH$.  The latter holds because $\Theta (\nabla ^{\fL})^{0,2} \ff  = \dbar ^{\fH} \dbar ^{\fH} \ff = 0$ for all $\ff \in \fH$.
\end{proof}

A definition of the curvature of an iBLS subfield that depends heavily on the ambient BLS field is rather undesirable, but we have not found a reasonable way to overcome this issue. On the other hand, the iBLS field $\sH \to B$, which is our main concern, has the interesting property that it has a huge number of natural ambient BLS fields $\sL ^{\theta}$, where $\theta$ is the parameter `enumerating' these BLS fields, such that 
\begin{enumerate}
\item[(i)]  $\sC^{\infty} (\sL ^{\theta})$ is independent of $\theta$, 
\item[(ii)] the formal tuning $\Sigma$ of $\sH \subset \sL^{\theta}$ is independent of $\theta$, and 
\item[(iii)] the curvature $\Theta ^{{\rm rel}\ \sL ^{\theta}}(\sH)$ is independent of $\theta$.
\end{enumerate}
In fact, the only dependence on the parameter $\theta$ is through the $\dbar$-operator $\dbar ^{\theta}$.  We will use these properties to define the curvature of the iBLS field $\sH \to B$.

\subsection{Positivity}\label{positivity-par}

All of the definitions of positivity of holomorphic Hermitian vector bundles have direct analogues in the setting of BLS fields and iBLS fields, with one important change.  Since one wants curvature to measure an infinitesimal quantity related to local variation, it is natural to evaluate the curvature only on vectors that can be continued to a local section.  For a BLS field $(\fH, \fh) \to B$ the curvature $\Theta(\fh)$ is evaluated at the fiber $\fH _t$ only on vectors in $\sC^{\infty}(\fH)_t (t)$, while for iBLS fields $(\fH, \fL, \Sigma)$ over $B$ the curvature $ \Theta ^{(\text{rel }\fL)}(\fh)$ is evaluated at the fiber $\fH _t$ only on vectors in $\Sigma _t (t)$.     For simplicity we shall state the definitions of positivity only for BLS fields.  For iBLS fields, simply replace $\Theta(\fh)$ by $\Theta ^{(\text{rel }\fL)}(\fH)$ and $\sC^{\infty}(\fH)_t(t)$ by $\Sigma _t (t)$.

Let $\fH \to B, \fh$ be a BLS field.  Using the metric $\fh$, one defines Hermitian forms $\{ \ ,\ \}_{\fh,\Theta (\fh)}$ on the fibers $T^{1,0}_{B,t} \tensor \sC^{\infty}(\fH)_t (t)$ by 
\[
\left \{ \tau_1 \tensor \ff_1, \tau_2 \tensor \ff_2 \right \}_{\fh, \Theta (\fh)} := \fh(\Theta (\fh)_{\tau_1, \bar \tau_2}\ff_1,\ff_2) 
\]
for indecomposable tensors in $T^{1,0}_{B,t}\tensor \sC^{\infty}(\fH)_t(t)$, and extending sesqui-linearly on each fiber.

\begin{defn}\label{k-pos-curv-defn}
Let $\fH \to B$ be a BLS field with smooth Hermitian metric $\fh$, and fix a smooth Hermitian metric $g$ on $B$.
\begin{enumerate}
\item[(i)]We say that $\fh$ has positive curvature in the sense of Griffiths at a point $t \in B$ if there exists $c > 0$ such that 
\[
\left \{ \tau \tensor \ff , \tau \tensor \ff \right \}_{\fh, \Theta (\fh)} \ge c \fh(\ff,\ff) g(\tau,\tau)
\]
for all $\tau \tensor \ff \in  T^{1,0}_{B,t} \tensor \sC^{\infty}(\fH)_t(t)$.
\item[(ii)]We say that $\fh$ has positive curvature in the sense of Nakano at a point $t \in B$ if there exists $c > 0$ such that 
\[
\left \{ \sum _{j=1} ^n \tau_j\tensor \ff_j, \sum _{k=1} ^n \tau_k\tensor \ff_k\right \}_{\fh, \Theta (\fh)} \ge c \sum _{j=1}^n\fh(\ff_j,\ff_j) g(\tau_j,\tau_j)
\]
for all $\tau_1\tensor \ff _1,...,\tau_n\tensor \ff _n \in T^{1,0}_{B,t}\tensor \sC^{\infty}(\fH)_t(t)$, where $n =\min \left ({\rm dim}_{\bbC} B, {\rm Rank }(\sC^{\infty}(\fH)_t(t))\right )$.
\item[(iii)] Let $k$ be an integer between $1$ and $\min \left ({\rm dim}_{\bbC} B, {\rm Rank }(\sC^{\infty}(\fH)_t(t))\right )$.  We say that $\fh$ has $k$-positive curvature at a point $t \in B$ if there exists $c > 0$ such that 
\[
\left \{ \sum _{j=1} ^k \tau_j\tensor \ff_j, \sum _{\ell=1} ^k \tau_{\ell}\tensor \ff_{\ell}\right \}_{\fh, \Theta (\fh)} \ge c \sum _{j=1}^k\fh(\ff_j,\ff_j) g(\tau_j,\tau_j)
\]
for all $\tau_1\tensor \ff _1,...,\tau_k\tensor \ff _k \in T^{1,0}_{B,t}\tensor \sC^{\infty}(\fH)_t(t)$
\red
\end{enumerate}
\end{defn}

\noi Clearly these notions of positivity do not depend on the choice of the metric $g$.  Note furthermore that $1$-positivity is just Griffiths positivity, that $n$-positivity is just Nakano positivity, and that $m$-positivity implies $k$-positivity if $m \ge k$.

There are also corresponding notions of nonnegative curvature (one takes $c$ nonnegative) and negative curvature (one reverses the inequalities and takes $c<0$).

\section{Relative Canonical Bundle, Lie Derivatives and Horizontal Lifts}\label{lie-deriv-section}

The present section introduces the basic definitions and tools needed to to define our Hilbert fields of interest and to carry out computation of the various geometric objects on these Hilbert fields; objects that were discussed more abstractly in the previous section.

\subsection{Families of twisted canonical sections}\label{relative-canonical-sections-paragraph}

Fix a proper holomorphic family of vector bundles $(E,h) \to X \stackrel{p}{\to}B$.  In the Introduction we defined the Hilbert field $\sH \to B$, whose fibers and metric are respectively
\[
\sH _t := H^0 (X_t, \cO _{X_t}(K_{X_t} \tensor E)) \quad \text{and} \quad (f_1,f_2)_t := (-1)^{\frac{n^2}{2}} \int _{X_t} \left < f_1 \wedge \bar f_2 , h\right >.
\]
In Section \ref{BLS-section-fields-section} we shall present the iBLS structures for $\sH \to B$.   To do so, one must have a way to vary $E$-valued canonical sections on $X_t$ with respect to $t$.  Following Berndtsson, we define the sections of $\sH \to B$ to be fiberwise-holomorphic sections of $K_{X/B} \tensor E \to X$, where $K_{X/B}$ denotes the relative canonical bundle.  Since $p :X \to B$ is a submersion, there are two equivalent ways of defining the relative canonical bundle.  The first is as a quotient 
\[
K_{X/B} := K_X \tensor p^* K_B^*
\]
of the canonical bundle of $X$ by the pullback of the canonical bundle of $B$, and the second is by equivalence classes of $E$-valued $(n,0)$-forms.

\subsubsection{\sf Representation of twisted relative canonical sections as twisted canonical sections}  

Given a point $t \in B$ and a coordinate system $s = (s^1,...,s^m)$ in an open set $U$ containing $t$, the restricted projection map $p : X_U \to U$, where $X_U := p^{-1} (U)$, can be expressed in terms of coordinates as $p= (p^1,..., p^m)$ for global holomorphic functions $p^i \in \cO (X_U)$.  Since $p$ is a submersion, i.e., $dp^1 \wedge ... \wedge dp ^m$ is nowhere-zero on $X_U$, any section $F$ of $K_X|_{X_t}  \to X_t$ can be written uniquely as $F = f \wedge dp^1 \wedge ... \wedge dp ^m$ for some section $f$ of $K_{X_t} \to X_t$.  This first perspective on the relative canonical bundle is well-suited to extension of canonical sections from a the fiber $X_t$  to $X$, in the sense of Ohsawa-Takegoshi.  But for the purpose of differentiating sections of $\sH \to B$ the `canonical sections' approach has drawbacks.  Therefore, we will not be using this point of view except when we discuss the question of local triviality of $\sH \to B$ in Section \ref{loc-triv-section}.

\subsubsection{\sf Representation by equivalence classes of relative top-forms}\label{n0-forms-description}

Consider the holomorphic vector bundle $\Lambda ^{n,0} _X \to X$ whose local sections are $(n,0)$-forms on $X$, and the subbundle $Z(p) ^{n,0} \to X$ whose fiber at $x \in X$ is defined by 
\[
Z(p) ^{n,0}_x = \left \{ f \in \Lambda ^{n,0}_{X,x}\ ;\ \left < f, \xi _1 \wedge ... \wedge \xi _n \right > = 0 \text{ for all }\xi _1\wedge ...\wedge \xi _n \in \Lambda ^n ({\rm Kernel}(dp(x)))\right \}.
\] 
Thus $Z(p)^{n,0} \to X$ is the annihilator of $\Lambda ^n T^{1,0}_{X/B}$, where $T^{1,0}_{X/B} := {\rm kernel} (dp)$ is the so-called \emph{vertical $(1,0)$-tangent bundle} of $p :X \to B$.  As $p$ is a submersion, $T^{1,0}_{X/B}$ is a holomorphic vector bundle of rank $n$, and thus $\Lambda ^n T^{1,0}_{X/B}$ has rank $1$.  Therefore there is a line bundle $L$ such that 
\[
0 \to Z(p)^{n,0} \to \Lambda ^{n,0} _X \to L \to 0
\]
is an exact sequence of vector bundles, and the line bundle $L$ is precisely $K_{X/B}$.  From this perspective a section of $K_{X/B} \tensor E \to X$ is an equivalence class of sections of $\Lambda _X ^{n,0} \tensor E \to X$, where two sections are equivalent if their difference is a section of $Z(p)^{n,0} \tensor E \to X$.  We note also that a section $u$ of $\Lambda _X ^{n,0} \tensor E  \to X$ is a section of $Z(p)^{n,0} \tensor E \to X$ if and only if $\iota _{X_t} ^*u = 0$ for every $t \in B$.

%The `canonical' nature of this representation of relative canonical forms is formulated in the following proposition.  

\begin{prop}\label{relative-restriction-prop}
Let $p: X \to B$ be a proper holomorphic submersion of complex manifolds.  For each $t \in B$ there is an isomorphism 
\[
\iota _t : \Gamma (X_t, (K_{X/B}\tensor E)|_{X_t}) \to \Gamma (X_t, K_{X_t}\tensor E|_{X_t} )
\]
such that, with $\iota _{X_t} : X_t \emb X$ denoting the natural inclusion of the fiber, 
\[
\iota _t (f) = \iota _{X_t} ^* u 
\]
for any $(n,0)$-form $u$ representing a section $f$ of $K_{X/B}\tensor E$ over an open subset $U$ of $t$.
\end{prop}

\begin{rmk}\label{local-in-the-base-coordinates-rmk}
In computation one often works locally on the base $B$, in some coordinate chart $U$ with local coordinates $s= (s^1,..., s^m)$.  In terms of these coordinates the projection $p$ of the holomorphic family has the expression $p = (p^1,..., p^m)$.  The functions $p ^i$ are of course constant on the fibers of $p$, and their differentials $dp ^i$ and $d\bar p ^i$ are global forms on $p^{-1} (U)$. Every section of $Z(p)^{n,0}|_U$ is of the form $dp^1 \wedge v_1+...+ dp ^m\wedge v_m$ for some smooth $E$-valued $(n-1,0)$-forms $v_1,...,v_m \in \Gamma (p^{-1} (U), \sC^{\infty}_X(\Lambda ^{n-1,0}_X \tensor E))$.  Thus two sections $u_1, u_2$ of $\Lambda ^{n,0} _X \tensor E|_U$ define the same $E$-valued relative canonical section over $U$ if and only if there exist $v_1,...,v_m \in \Gamma (p^{-1} (U), \sC^{\infty}_X(\Lambda ^{n-1,0}_X \tensor E))$ such that 
\[
u_2 - u_1 =  dp ^1 \wedge v_1+ \cdots + dp ^m \wedge v_m, 
\]
or equivalently $u_1 \wedge dp ^1 \wedge \cdots \wedge dp ^m = u_2 \wedge dp ^1 \wedge \cdots \wedge dp ^m$.  
\red
\end{rmk}

\subsection{Lie derivatives of twisted forms}\label{Lie-deriv-paragraph}

We shall need to compute derivatives of integrals of the form \eqref{bls-inner-product}.  The usual tool for taking the derivative `through the integral' is the Lie derivative.  When the objects being differentiated have natural action by diffeomorphisms, the Lie derivative has a very natural definition, namely differentiation of the pullback by the flow of the vector field along which one is computing derivatives.  For example, differential forms, vector fields and the tensors one can construct from these behave well with respect to this definition.

In the twisted setting, i.e., when one looks at differential forms with values in a vector bundle, one can still make use of the classical Lie derivative by working in locally trivial open sets, but many formulas become local and keeping track of these local formulas can be cumbersome.  It is clearer and more efficient to identify certain combinations of terms with a global meaning, and to work with those combinations of terms as the basic objects.

\begin{defn}\label{cplx-lie-deriv}
Let $X$ be a complex manifold, let $E \to X$ be a holomorphic vector bundle with connection $\nabla$, let $u$ be an $E$-valued $(p,q)$-form and let $\xi$ be a complex vector field on $X$, i.e., a section of $T_X \tensor \bbC \to X$.  
\begin{enumerate}
\item[i.] The Lie derivative along $\xi$, which acts on complex-valued forms, is defined by 
\[
L_{\xi} := L _{\re \xi} + \ii L_{\im \xi},
\]
where $L$ on the right-hand side denotes the real Lie derivative of real vector fields.
\item[ii.] The twisted Lie derivative of an $E$-valued $(p,q)$-form $u$  along $\xi$ is 
\[
L^{\nabla }_{\xi} u= L _{\xi} u := \nabla (\xi \lrcorner   u ) + \xi \lrcorner \nabla u,
\]
\end{enumerate}
\end{defn}

\begin{rmk}
In the untwisted case, i.e., when the underlying vector bundle $E$ is the trivial line bundle and the connection $\nabla$ is the usual exterior derivative, Definition \ref{cplx-lie-deriv}.ii agrees with the aforementioned definition of Lie derivative because of the formula 
\[
L_{\xi} u = d(\xi \lrcorner u) + \xi \lrcorner (du)
\]
of H. Cartan.
\red
\end{rmk}

Via the splitting $\nabla = \nabla ^{1,0} + \nabla ^{0,1}$ one can define the $(1,0)$- and $(0,1)$-Lie derivative of a $(p,q)$-form $u$ with values in a holomorphic vector bundle along a complex vector field $\xi$ by 
\[
L^{\nabla 1,0}_{\xi} u  := \nabla^{1,0} (\xi \lrcorner   u ) + \xi \lrcorner \nabla^{1,0} u \quad \text{and} \quad L^{\nabla 0,1}_{\xi} u  := \nabla^{0,1} (\xi \lrcorner   u ) + \xi \lrcorner \nabla^{0,1} u.
\]
When the connection $\nabla$ is clear from the context, we shall write $L^{1,0}_{\xi}$ instead of $L^{\nabla 1,0}_{\xi}$ (and similarly for $(0,1)$) but we shall retain the notation $L^{\nabla}$ for the full twisted Lie derivative for both clarity and emphasis.  

\begin{s-rmk}
Note that $L^{1,0} _{\xi}$ (resp. $L^{0,1}_{\xi}$) preserves the bi-degree $(p,q)$ if and only if $\xi$ is a $(1,0)$ (resp. $(0,1)$) vector field or $q=0$ (resp. $p=0$).
\red
\end{s-rmk}

\begin{s-rmk}
One can also define $L^{0,1} _{\xi} u := \xi \lrcorner \dbar u + \dbar (\xi \lrcorner u)$.  In the present article these two meanings of $L^{0,1}_{\xi}$ do not conflict because we only use connections $\nabla$ satisfying $\nabla ^{0,1}= \dbar$.
\red
\end{s-rmk}

\begin{prop}\label{restricted-product-rule}
Let $X$ be a K\"ahler manifold of dimension $n+m$, let $Z \subset X$ be an $n$-dimensional submanifold, let $E \to X$ be a holomorphic vector bundle with Hermitian metric $h$ and its Chern connection, and let $u , u'$ be $E$-valued $(n,0)$-forms on $X$.  If $\xi$ is a $(1,0)$-vector field on $X$ then 
\[
\iota _Z^* L _{\xi} \left < u \wedge \bar u ' , h\right > = \iota _Z^*  \left < \left ( L ^{1,0}_{\xi} u \right ) \wedge \bar u ' , h\right >  +   \iota _Z^*\left < u\wedge \overline{L^{0,1}_{\bar \xi} u'} , h \right >
\]
and
\[
\iota _Z^* L _{\bar \xi} \left < u \wedge \bar u ' , h\right > = \iota _Z^*  \left < \left ( L ^{0,1}_{\bar \xi} u \right ) \wedge \bar u ' , h\right >  +   \iota _Z^*\left < u\wedge \overline{L^{1,0}_{\xi} u'} , h \right >.
\]
%In particular, if $u'$, resp. $u$, is holomorphic then 
%\[
%\iota _Z^* L _{\xi} \left < u \wedge \bar u ' , h\right > = \iota _Z^*  \left < \left ( L ^{1,0}_{\xi} u \right ) \wedge \bar u ' , h\right >,   \quad \text{resp.} \quad \iota _Z^* L _{\bar \xi} \left < u \wedge \bar u ' , h\right > =  \iota _Z^*\left < u\wedge \overline{L^{1,0}_{\xi} u'} , h \right >.
%\]
\end{prop}

\begin{proof}
We compute that 
\begin{eqnarray*}
&& L _{\xi}  \left < u \wedge \bar u ' , h\right > \\
&=& d \left ( \xi \lrcorner  \left < u \wedge \bar u ' , h\right >\right ) + \xi \lrcorner d  \left < u \wedge \bar u ' , h\right > \\
&=& d \left (\left < (\xi \lrcorner   u )\wedge \bar u ' , h\right >  + (-1) ^n  \left < u \wedge \overline{(\bar \xi \lrcorner u ')} , h\right >  \right )+ \xi \lrcorner \left ( \left < \nabla u \wedge \bar u' , h \right > + (-1)^n  \left < u\wedge \overline{\nabla u'} , h \right >\right )\\
&=& \left < \nabla (\xi \lrcorner   u )\wedge \bar u ' , h\right >  + (-1) ^{n-1} \left < (\xi \lrcorner   u )\wedge \overline{\nabla u '} , h\right >  +(-1) ^n \left < \nabla u \wedge \overline{(\bar \xi \lrcorner u')}, h\right > \\
&& + \left <  u \wedge \overline{\nabla (\bar \xi \lrcorner u')}, h\right >  + \left < (\xi \lrcorner (\nabla u)) \wedge \bar u' , h \right > + (-1)^{n+1} \left < \nabla u \wedge \overline{(\bar \xi \lrcorner u')} , h \right >  \\
&& + (-1) ^{n} \left < \xi \lrcorner u \wedge \overline{\nabla u'}, h\right > + \left < u , \overline{(\bar \xi \lrcorner (\nabla u'))}, h\right >\\
&=& \left < \nabla (\xi \lrcorner   u )\wedge \bar u ' , h\right >   + \left <  u \wedge \overline{\nabla (\bar \xi \lrcorner u')}, h\right >  + \left < (\xi \lrcorner (\nabla u)) \wedge \bar u' , h \right >  + \left < u , \overline{(\bar \xi \lrcorner (\nabla u'))}, h\right >\\
&=& \left < \nabla (\xi \lrcorner   u )\wedge \bar u ' , h\right >  + \left < (\xi \lrcorner (\nabla u)) \wedge \bar u' , h \right >  + \left < u \wedge \overline{(\bar \xi \lrcorner (\nabla u'))}, h\right >,
\end{eqnarray*}
where the last equality holds because $\bar \xi \lrcorner u' = 0$.  Therefore 
\begin{eqnarray*}
\iota _Z ^*  L _{\xi}  \left < u \wedge \bar u ' , h\right > &=&  \iota _Z ^* \left < \nabla (\xi \lrcorner   u )\wedge \bar u ' , h\right >   + \iota _Z ^* \left < (\xi \lrcorner (\nabla u)) \wedge \bar u' , h \right >  + \iota _Z ^* \left < u , \overline{(\bar \xi \lrcorner (\nabla u'))}, h\right > \\
&=&  \iota _Z ^* \left < \nabla^{1,0} (\xi \lrcorner   u )\wedge \bar u ' , h\right >   + \iota _Z ^* \left < (\xi \lrcorner (\nabla^{1,0} u)) \wedge \bar u' , h \right >  + \iota _Z ^* \left < u , \overline{(\bar \xi \lrcorner (\dbar u'))}, h\right >,
\end{eqnarray*}
the last equality being a consequence of the fact that upon restriction to $Z$ all of the forms whose bidegree is not $(n,n)$ must vanish.  
\end{proof}

\begin{lem}\label{lie-deriv-lemma}
Let $p :X \to B$ be a proper holomorphic family, let $E \to X, h$ be a Hermitian holomorphic vector bundle equipped with its Chern connection and let $u$ and $u'$ be $E$-valued $(n,0)$-forms on $X$.  If $\tau$ is a $(1,0)$-vector field on $B$ and $\xi_{\tau}$ is a $(1,0)$-vector field on $X$ such that $dp \xi _{\tau} = \tau$ then
\[
\tau \int _{X_t} \left < u\wedge \bar u', h \right > = \int _{X_t} \left < L ^{1,0} _{\xi _{\tau}} u\wedge \bar u', h \right > + \int _{X_t} \left < u\wedge \overline{ L^{0,1} _{\bar\xi _{\tau}}u'}, h \right >,
\]
where $X_t := p ^{-1} (t)$.
\end{lem}

\begin{proof}
Write $\tau = \tau _1 + \ii \tau _2$ in terms of its real and imaginary parts, and similarly write $\xi _{\tau}= \xi_1 + \ii \xi_2$.  Since $p$ is holomorphic, $dp \xi _i = \tau _i$.  Then 
\begin{eqnarray*}
\tau _i \int _{X_t} \left < u\wedge \bar u', h \right >  &=& \left . \frac{\di}{\di r} \right |_{r=0} \int _{\Phi ^r_{\xi _i}(X_t)} \left < u\wedge \bar u', h \right >= \left . \frac{\di}{\di r} \right |_{r=0} \int _{X_t} \left ( \Phi ^r_{\xi _i} \right )^* \left < u\wedge \bar u', h \right > \\
&=&  \int _{X_t} L_{\xi _i} \left < u\wedge \bar u', h \right >,
\end{eqnarray*}
so the lemma follows from Definition \ref{cplx-lie-deriv} and Proposition \ref{restricted-product-rule}.
\end{proof}

Similarly one has the following lemma.  

\begin{lem}\label{(n,1)v(n-1,0)-lemma}
Let $X$, $Z$, $E$ and $h$ be as in Proposition \ref{restricted-product-rule}, let $\alpha$ and $\beta$ be $E$-valued forms of bidegree $(n,1)$ and $(n-1,0)$ repectively.  If $\xi$ is a $(1,0)$-vector field on $X$ then 
\[
 L^{1,0}_{\xi}   \left < \alpha\wedge \bar \beta, h\right >  = \left < (L^{1,0} _{\xi}\alpha ) \wedge \bar \beta, h\right > +  \left < \alpha \wedge \overline{L^{0,1} _{\bar \xi} \beta}, h \right >,
\]
and hence  
\[%begin{equation}\label{01xi}
\int _Z  L^{1,0}_{\xi}   \left < \alpha\wedge \bar \beta, h\right >  = \int_Z \left < (L^{1,0} _{\xi}\alpha ) \wedge \bar \beta, h\right > + \int _Z \left < \alpha \wedge \overline{L^{0,1} _{\bar \xi} \beta}, h \right >.
\]%end{equation}
\end{lem}

\begin{proof}
We compute that 
\begin{eqnarray*}
L^{1,0} _{\xi} \left < \alpha \wedge \bar \beta ,h\right > &=& \xi \lrcorner \di \left < \alpha \wedge \bar \beta ,h\right > + \di \left ( \xi \lrcorner \left < \alpha \wedge \bar \beta ,h\right >\right )\\
&=&  \xi \lrcorner \left <\nabla ^{1,0} \alpha \wedge \bar \beta ,h\right > +  (-1) ^{n+1} \xi \lrcorner \left < \alpha \wedge \overline{\dbar \beta},h\right > + \di \left < \xi \lrcorner \alpha \wedge \bar \beta , h \right >\\
&=& \left <  (\xi \lrcorner \nabla ^{1,0} \alpha) \wedge\bar  \beta ,h\right >  + (-1) ^{n+1} \left < (  \xi \lrcorner \alpha) \wedge \overline{\dbar \beta} ,h\right > +  \left < \alpha \wedge\overline{\bar   \xi \lrcorner \dbar \beta} ,h\right > \\
&& + \left < \nabla ^{1,0} (\xi \lrcorner \alpha) \wedge \bar \beta , h \right >+ (-1)^n \left < \xi \lrcorner \alpha \wedge \overline{\dbar \beta} , h \right >\\
&=& \left <  (\xi \lrcorner \nabla ^{1,0} \alpha) \wedge\bar  \beta ,h\right >  + \left < \nabla ^{1,0} (\xi \lrcorner \alpha) \wedge \bar \beta , h \right > + \left < \alpha \wedge\overline{\bar   \xi \lrcorner \dbar \beta} ,h\right > \\
&=& \left < (L^{0,1} _{\xi}\alpha ) \wedge \bar \beta, h\right > + \left < \alpha \wedge \overline{L^{0,1} _{\bar \xi} \beta}, h \right >,
\end{eqnarray*}
and the proof is completed by integrating over $Z$.
\end{proof}

\begin{prop}\label{restricted-product-rule-2}
Let $X$, $Z$, $E$ and $h$ be as in Proposition \ref{restricted-product-rule}, let $u$ and $u'$ be $E$-valued $(n,0)$-forms on $X$ and let $\gamma$ be an $E$-valued $(n-1,1)$-form on $X$.  If $\xi$ is a $(1,0)$-vector field on $X$ then 
\begin{equation}\label{10xibar}
 L^{1,0}_{\bar \xi}   \left < \gamma \wedge \bar u, h\right >  = \left < L^{1,0} _{\bar \xi} \gamma \wedge \bar u, h\right > +  \left < \gamma \wedge \overline{L^{0,1} _{\xi}u} , h \right >
\end{equation}
and 
\begin{equation}\label{01xi}
L^{0,1}_{\xi}   \left < u\wedge \bar u', h\right > = \left < (L^{0,1} _{\xi}u) \wedge \bar u', h\right >.
\end{equation}
In particular, 
\begin{equation}\label{10xibar-integrated}
\iota _Z^* L_{\bar \xi}   \left < \gamma \wedge \bar u, h\right >  = \iota _Z^* \left < L^{1,0} _{\bar \xi} \gamma \wedge \bar u, h\right > +  \iota _Z^* \left < \gamma \wedge \overline{L^{0,1} _{\xi}u} , h \right >
\end{equation}
\end{prop}

\begin{proof}
As in the proof of Proposition \ref{restricted-product-rule} we have 
\[
L _{\bar \xi}  \left < \gamma \wedge \bar u  , h\right > = \left < \nabla (\bar \xi \lrcorner   \gamma )\wedge \bar u  , h\right >  + \left < (\bar \xi \lrcorner (\nabla \gamma)) \wedge \bar u , h \right >  + \left <  \gamma \wedge \overline{\nabla (\xi \lrcorner u)}, h\right > + \left < \gamma \wedge \overline{(\xi \lrcorner (\nabla u))}, h\right >
\]
and 
\[
L _{\xi}  \left < u \wedge \bar u'  , h\right > = \left < \nabla (\xi \lrcorner   u )\wedge \bar u'  , h\right >  + \left < (\xi \lrcorner (\nabla u)) \wedge \bar u' , h \right >  + \left <  u \wedge \overline{\nabla (\bar \xi \lrcorner u')}, h\right > + \left < u \wedge \overline{(\bar \xi \lrcorner (\nabla u'))}, h\right >.
\]
Taking the $(n,n)$-components of both sides of the first identity yields \eqref{10xibar}, and taking the $(n-1,n+1)$-parts of the second identity yields \eqref{01xi}.  Finally, since $L^{0,1}_{\bar \xi}\left < \gamma \wedge \bar u , h \right >$ is of bidegree $(n-1, n+1)$, \eqref{10xibar-integrated} follows from \eqref{10xibar}
\end{proof}

\begin{lem}\label{dbar-of-sff}
Let $(E,h) \to X \stackrel{p}{\to} B$ be a holomorphic family of Hermitian holomorphic vector bundles, let $\tau$ be a holomorphic vector field on $B$, let $\xi _{\tau}$ be a lift of $\tau$ to $X$, and let $u$ be an $E$-valued $(n,0)$-form on $X$.  Then with respect to the Chern connection of $(E,h)$ 
\[
\dbar \iota _{X_t} ^* \left ( L^{1,0}_{\xi  _{\tau}} u \right ) = \iota _{X_t} ^* \left (- (\xi _{\tau} \lrcorner \Theta (h)) u - \nabla ^{1,0} ((\dbar \xi _{\tau}) \lrcorner u) +L^{1,0}_{\xi  _{\tau}} \dbar u\right ),
\]
where the $(n,1)$-form $\iota _{X_t} ^*( (\xi _{\tau} \lrcorner \Theta (h)) u)$ is defined by 
\[
\bar \zeta \lrcorner \left ( \iota _{X_t} ^* \left ( (\xi _{\tau} \lrcorner \Theta (h)) u\right ) \right ) :=  \iota _{X_t} ^* \left ( \Theta (h) _{\xi _{\tau} \bar \zeta}u\right ).
\]
\end{lem}

\begin{proof}
One has 
\begin{eqnarray*}
\dbar _{X_t}\iota _{X_t} ^*L^{1,0}_{\xi _{\tau}}u &=&  \iota _{X_t} ^* \dbar L^{1,0}_{\xi _{\tau}}u = \iota _{X_t} ^* \left ( \dbar \nabla ^{1,0} (\xi _{\tau} \lrcorner u )+ \dbar (\xi _{\tau} \lrcorner \nabla ^{1,0} u)\right )\\
&=& \iota _{X_t} ^* \left ( \dbar \nabla ^{1,0} (\xi _{\tau} \lrcorner u) - \xi _{\tau} \lrcorner \dbar \nabla ^{1,0} u+ (\dbar \xi _{\tau}) \lrcorner \nabla^{1,0}u \right )\\
&=&  \iota _{X_t} ^* \left ( \Theta (h) (\xi _{\tau} \lrcorner u) - \nabla ^{1,0} \dbar (\xi_{\tau} \lrcorner u) - (\xi _{\tau} \lrcorner \Theta (h) u) +\xi_{\tau} \lrcorner (\nabla ^{1,0} \dbar u) + (\dbar \xi _{\tau}) \lrcorner \nabla ^{1,0}u \right )\\
&=&\iota _{X_t}^*\left (-(\xi_{\tau}\lrcorner\Theta (h))u-\left ( \nabla^{1,0} \dbar (\xi _{\tau}\lrcorner u) - \xi_{\tau} \lrcorner (\nabla ^{1,0} \dbar u) \right ) + (\dbar \xi _{\tau}) \lrcorner \nabla ^{1,0}u\right )\\
&=&\iota _{X_t} ^* \left (- (\xi _{\tau} \lrcorner \Theta (h)) u - \nabla ^{1,0} ((\dbar \xi _{\tau}) \lrcorner u) + \nabla ^{1,0} (\xi  _{\tau} \lrcorner \dbar u) + \xi  _{\tau} \lrcorner \nabla ^{1,0} \dbar u + (\dbar \xi _{\tau}) \lrcorner \nabla ^{1,0}u\right )\\
&=& \iota _{X_t} ^* \left (- (\xi _{\tau} \lrcorner \Theta (h)) u - \nabla ^{1,0} ((\dbar \xi _{\tau}) \lrcorner u) +L^{1,0}_{\xi  _{\tau}} \dbar u + (\dbar \xi _{\tau}) \lrcorner \nabla ^{1,0}u\right ) \\
&=& \iota _{X_t} ^* \left (- (\xi _{\tau} \lrcorner \Theta (h)) u - \nabla ^{1,0} ((\dbar \xi _{\tau}) \lrcorner u) +L^{1,0}_{\xi  _{\tau}} \dbar u\right ).
\end{eqnarray*}
Indeed, $(\dbar \xi _{\tau}) \lrcorner \nabla ^{1,0}u= 0$, because $\dbar \xi  _{\tau}$ is vertical, and $\nabla ^{1,0}u$ has bidegree $(n+1,0)$.
\end{proof}

\begin{prop}\label{int-ptwise-holo}
Let $(E,h) \to X \stackrel{p}{\to} B$ be a holomorphic family of holomorphic Hermitian vector bundles, let $t \in B$ and let $\tau \in H^0(B,\cO (T^{1,0} _B))$.  Let $\xi _{\tau}$ be a lift of $\tau$, i.e., a $(1,0)$-vector field $\xi _{\tau}$ on $X$ such that $dp \xi_{\tau} = \tau$.  Then with respect to the Chern connection for $(E,h)$, one has the following statements.
\begin{enumerate}
\item[{\rm a.}] If $u, \tilde u \in H^0(X, \sC^{\infty} (\Lambda ^{n,0} _X \tensor E))$ agree on the fibers of $X \to B$, i.e., they represent the same smooth section $[u]= [\tilde u]$ of $K_{X/B} \tensor E \to X$, then 
\[
\iota ^*_{X_t} L^{1,0} _{\xi_{\tau}}u = \iota ^*_{X_t} L^{1,0} _{\xi _{\tau}}\tilde u \quad \text{and} \quad \iota ^*_{X_t} L^{1,0} _{\xi_{\tau}}\dbar u = \iota ^*_{X_t} L^{1,0} _{\xi _{\tau}}\dbar \tilde u.
\]
\item[{\rm b.}] If moreover $\iota ^*_{X_t} u  \in \sH_t$ then
\[
\iota ^*_{X_t} L^{1,0} _{\tilde \xi_{\tau}}\dbar u = \iota ^*_{X_t} L^{1,0} _{\xi _{\tau}}\dbar u
\]
for any other  lift $\tilde \xi _{\tau}$ of $\tau$.
\item[{\rm c.}]  Finally, if $\iota _{X_t} ^* u = 0$ and $\iota _{X_t} ^*L^{1,0} _{\xi _{\tau}}\dbar u = 0$ then $\iota _{X_t} ^* L^{1,0} _{\xi _{\tau}} u \in \sH _t$.
\end{enumerate}
\end{prop}

\begin{proof}
Since $v := \tilde \xi _{\tau}- \xi _{\tau}$ is tangent to the fibers of $p$, 
\[
\iota ^*_{X_t} L^{1,0} _{\tilde \xi _{\tau}}\dbar u_1 - \iota ^*_{X_t} L^{1,0} _{\xi _{\tau}}\dbar u_1 = \iota _{X_t} ^* L^{1,0}_{v} \dbar u = L^{1,0}_{v_t} \dbar \iota _{X_t} ^* u = 0,
\]
where $v_t := v|_{X_t}$.  Thus (b) is proved.  

To prove (a) it suffices to show that if $u \in H^0(X, \sC^{\infty}(\Lambda ^{n,0}_X \tensor E))$ restricts to $0$ on fibers then $L^{1,0}_{\xi _{\tau}} u$ and $L^{1,0}_{\xi _{\tau}} \dbar u$ restrict to $0$ on fibers.  Now, using coordinates $s = (s^1,...,s^m)$ on the base, in terms of which we write $p = (p^1,..., p^m)$ and $\tau = \tau ^{\mu} \frac{\di}{\di s^{\mu}}$ (see Remark \ref{local-in-the-base-coordinates-rmk}), the section $u$ vanishes on fibers if and only if locally in the base one has 
\[
u = dp ^{\mu} \wedge v _{\mu}
\]
for some $E$-valued $(n-1,0)$-forms $v_1,..., v_m$.  But then 
\[
L^{1,0} _{\xi _{\tau}}u =  L^{1,0} _{\xi_{\tau}} dp ^{\mu} \wedge v_{\mu} =   ( p^* \di \tau ^{\mu} \wedge v _{\mu} ) + dp ^{\mu} \wedge   \iota _{X_t} ^*L^{1,0} _{\xi_{\tau}}  v _{\mu}  = dp ^{\mu} \wedge \left ( p ^*  \left ( \tfrac{\di \tau ^{\nu}}{\di s^{\mu}} \right ) v_{\nu}  + L^{1,0} _{\xi _{\tau}} v_{\mu} \right )
\]
and 
\[
- L^{1,0} _{\xi _{\tau}}\dbar u =  L^{1,0} _{\xi_{\tau}} dp ^{\mu} \wedge \dbar v_{\mu} =   ( p^* \di \tau ^{\mu} \wedge \dbar v _{\mu} ) + dp ^{\mu} \wedge   \iota _{X_t} ^*L^{1,0} _{\xi_{\tau}}  \dbar v _{\mu}  = dp ^{\mu} \wedge \left ( p ^*  \left ( \tfrac{\di \tau ^{\nu}}{\di s^{\mu}} \right ) \dbar v_{\nu}  + L^{1,0} _{\xi _{\tau}} \dbar v_{\mu} \right ),
\]
so that $\iota _{X_t} ^*L^{1,0} _{\xi _{\tau}} u = 0$ and $\iota _{X_t} ^*L^{1,0} _{\xi _{\tau}}\dbar u = 0$ for all $t \in B$.  

Finally, if $\tau$ is a holomorphic vector field, $\dbar \xi _{\tau}$ is a $(0,1)$-form with values in the vertical bundle $T^{1,0} _{X/B}$.  Therefore if $\iota _{X_t} ^*u = 0$ then $\iota _{X_t}^*\left ( \nabla ^{1,0} (\dbar \xi _{\tau}) \lrcorner u)\right ) = 0$ and $\iota _{X_t}^*( \xi _{\tau}  \lrcorner \Theta (h) )u$.  If also $\iota _{X_t} ^*L^{1,0}_{\xi _{\tau}} \dbar u = 0$ then by Lemma \ref{dbar-of-sff} $\dbar \iota _{X_t} ^* \left ( L^{1,0}_{\xi  _{\tau}} u \right ) = 0$, so (c) is proved.
\end{proof}

\subsection{Horizontal lifts}\label{lifts-paragraph}

In differentiating sections sections of $\sH \to B$ or the ambient $L^2$ bundle $\sL \to B$ one must choose \emph{horizontal} lifts of $(1,0)$-vector fields from $B$ to $X$.  

\begin{defn}
A horizontal distribution for the smooth family $p :X \to B$ is a distribution $H \subset T_X$ such that $dp |_H :H \to T_B$ is an isomorphism of vector bundles.  Equivalently, there is a morphism of vector bundles $\theta : T_B \to T_X$ such that $dp \theta (\tau) = \tau$.
\red
\end{defn}

We shall abuse notation and write $\theta$ instead of $H$ for the horizontal distribution.  Thus $\theta$ is both a subbundle of $T_X$ and a right inverse to $dp : T_X \to T_B$.

In our situation the total space $X$ and base manifold $B$ are complex manifolds and $p :X \to B$ is holomorphic, i.e., $dp J = J dp$, where $J$ denotes the almost complex structures.  Regardless of whether $J$ is integrable, there is a map that assigns to any horizontal distribution $\theta \subset T_X$ a unique horizontal $(1,0)$-distribution $\theta ^{1,0} \subset T^{1,0} _X$, defined as $\theta ^{1,0} = s^{1,0} \theta$, where 
\[
s^{1,0} : T_X \ni \xi \mapsto \tfrac{1}{2} (\xi - \ii J \xi) \in T^{1,0} _X.
\]
The correspondence $\theta \leftrightarrow \theta ^{1,0}$ is bijective, and since $p :X \to B$ is holomorphic, 
\[
dp \theta ^{1,0} \cong T^{1,0} _B \quad \text{and} \quad dp \theta ^{0,1} \cong T^{0,1} _B,
\]
where $\theta ^{0,1} := \overline{\theta ^{1,0}}$.

To slightly shrink the sizes of many already-too-large formulas, we shall write 
\label{horiz-lift-notation-page}
\[
\xi ^{\theta} _{\tau} := \theta ^{1,0} (\tau) \quad \text{and} \quad \bar \xi ^{\theta}_{\tau} := \theta ^{0,1} (\bar \tau) = \overline{\theta ^{1,0}(\tau)}, \qquad \tau \in T^{1,0} _B.
\]

\begin{prop}\label{integrable-horizontal-distribution}
Let $p :X \to B$ be a holomorphic submersion and let $\theta$ be a horizontal distribution for $p$.  Then $\theta$ is integrable if and only if
\begin{equation}\label{integ-crit}
\left [ \xi ^{\theta}_{\sigma}, \xi ^{\theta}_{\tau}\right ] = \xi ^{\theta}_{[\sigma, \tau]}
\end{equation}
for all local vector fields $\sigma, \tau$ on $B$.
\end{prop}

\begin{proof}
If \eqref{integ-crit} holds then the Lie brackets of vector fields tangent to $\theta$ are tangent to $\theta$, i.e., $\theta$ is integrable.  

Conversely, suppose $\theta$ is integrable.   Then $\left [ \xi ^{\theta}_{\sigma}, \xi ^{\theta}_{\tau}\right ] = \xi ^{\theta} _{\eta}$ for some vector field $\eta$ on $B$, and we must show that $\eta = [\sigma, \tau]$.  By the Inverse Function Theorem each point $P \in X$ has a local coordinate system $(t,x)$ such that $p (t,x) = t$ is a local coordinate at $p(P)$.  In these coordinates, if $\sigma = \sigma ^{\mu} \frac{\di}{\di t ^{\mu}}$ and $\tau = \tau ^{\nu} \frac{\di}{\di t ^{\nu}}$ then 
\[
\xi ^{\theta}_{\sigma} = \sigma ^{\mu} \frac{\di }{\di t ^{\mu}} + \alpha ^j \frac{\di}{\di x ^j} \quad \text{and} \quad \xi ^{\theta}_{\tau} = \tau ^{\nu} \frac{\di }{\di t ^{\nu}} + \beta ^k \frac{\di}{\di x ^k} 
\]
and therefore 
\[
\left [ \xi ^{\theta}_{\sigma},\xi ^{\theta}_{\tau} \right ] = [\sigma, \theta]^{\mu} \frac{\di}{\di t ^{\mu}} + \gamma ^i \frac{\di}{\di x ^i}.
\]
Thus 
\[
\eta = dp \left ( \left [ \xi ^{\theta}_{\sigma},\xi ^{\theta}_{\tau} \right ]  \right ) = [\sigma, \tau],
\]
as claimed.
\end{proof}

\section{Hilbert Fields of Sections}\label{BLS-section-fields-section}

In this section we discuss the structure of the Hilbert fields $\sH \to B$ and $\sL \to B$ associated to a holomorphic family of vector bundles $(E,h) \to X \stackrel{p}{\to} B$ that we defined in the introduction.  

\subsection{Berndtsson's Hilbert field}\label{structure-of-H} 

\begin{defn}
The Hilbert field $\sH \to B$ with its $L^2$-inner product \eqref{bls-inner-product} is called the \emph{Berndtsson Hilbert field of the family $(E,h) \to X \stackrel{p}{\to} B$}.
\red
\end{defn}

%Let us recall Berndtsson's definition of smooth and holomorphic sections of $\sH \to B$.  

\subsubsection*{\sf Smooth and holomorphic sections of $\sH \to B$}

We begin with Berndtsson's definition of sections of $\sH \to B$.  However, we phrase the definition in terms of sections of the relative canonical bundle, and for the moment repress the fact that these sections are equivalence classes of twisted $(n,0)$-forms.

\begin{defn}\label{smooth-sections-of-holo-bundle-defn}
A section $\ff$ of $\sH \to B$ is a section $f \in\Gamma (X, K_{X/B} \tensor E)$ such that for each $t \in B$ $\iota _t f \in \sH_t = H^0(X_t, \cO (K_{X_t} \tensor E|_{X_t}))$ (c.f. Proposition \ref{relative-restriction-prop}).  We write 
\[
\ff \in \Gamma (B,\sH).
\]
We also use the notation 
\[
\ff = \fri (f) \quad \text{ and } \quad f = \fa (\ff),
\]
and say that $\ff \in \Gamma (B, \sH)$ is induced by $f$, and that $f \in \Gamma (X, K_{X/B} \tensor E)$ is associated to $\ff$.  
We say $\ff :B \to \sH$ is holomorphic (resp. $\sC^r$, measurable , etc.) if $\fa(\ff)$ is holomorphic (resp. $\sC^r$, measurable , etc.).  We denote by $\Gamma (U, \sC^{\infty} (\sH))$ and $\Gamma (U, \cO (\sH))$ the set of smooth and holomorphic sections over an open set $U \subset B$.  Finally, we let $\sC^{\infty} (\sH)$ and $\cO (\sH)$ denote the sheaves of germs of smooth  and holomorphic sections of $\sH \to B$ respectively.
\red
\end{defn}

\begin{rmk}
We emphasize that there may well exist $t \in B$ and $f \in \sH _t$ such that there is no \emph{smooth} (or even continuous) section $\ff$ of $\sH \to B$ satisfying $\ff (t) = f$.  This is the case precisely when $\sH \to B$ is not locally trivial near $t \in B$.
\red
\end{rmk}

Unfortunately the sheaf $\sC^{\infty} (\sH)$ does not always define a smooth structure on $\sH$.  The next proposition characterizes precisely when this is the case.

\begin{prop}\label{H-is-BLSS-iff-VB}
Suppose that the holomorphic family $(E, h) \to  X \stackrel{p}{\to} B$ is proper, or more generally that all of the fibers of the resulting Berndtsson Hilbert field $\sH$ are finite-dimensional.  Then the sheaf $\sC^{\infty} (\sH)$ is a smooth structure for the Hilbert field $\sH$, in the sense of Definition {\rm \ref{smooth-structure-defn}.i}, if and only if $\sH$ is a holomorphic vector bundle. 
\end{prop}

\begin{proof}
If $\sH \to B$ is locally trivial then the smooth sections clearly form a dense subset on every fiber; in fact, every element of every fiber is the evaluation of a smooth (and even holomorphic) section.  Conversely, if $\sH \to B$ has a smooth structure then $t \mapsto \dim \sH _t$ is lower semi-continuous.  Indeed, if $t_o \in B$ then we can choose smooth sections $f_1,..., f_N \in \sC^{\infty} (\sH)_{t_o}$ such that $\{ f_1 (t_o),..., f_N(t_o)\} \subset \sH _{t_o}$ is a basis, since the smooth sections are dense.  By continuity, $\{ f_1 (t),..., f_N(t)\} \subset \sH _t$ is linearly independent for every $t$ in some neighborhood $U$ of $t_o$.  Thus $\dim \sH _t \ge \dim \sH _{t_o}$ for all $t \in U$.  

On the other hand, it is well-known that $t \mapsto \dim \sH _t$ is upper semi-continuous.  (One can deduce upper semi-continuity from Grauert's theorem on the coherence of proper direct images, but in the present setting of a proper holomorphic family this upper semi-continuity is a relatively simple consequence of Montel's Theorem.)  Therefore the rank of the fibers of $\sH \to B$ is constant, and thus, by a well-known result of Kodaira and Spencer, $\sH \to B$ is a holomorphic vector bundle.  
\end{proof}

\subsubsection*{\sf Berndtsson's $\dbar$ operator for $\sH \to B$}

Let us now turn to the definition of the $\dbar$ operator for $\sH \to B$.  We begin with the following simple but important fact.

\begin{prop}\label{H_t-dbar-is-lift-independent}
Let $u$ be an $E$-valued $(n,0)$-form on the total space of the family $X\to B$, let $t \in B$, and suppose that $\iota_{X_t}^*u  \in \sH _t$.  Then for any lifts $\xi^1 _{\tau}$ and $\xi^2 _{\tau}$ of the $(1,0)$-vector $\tau \in T^{1,0}_{B,t}$ one has 
\[
\bar \xi ^1 _{\tau} \lrcorner \dbar u = \bar \xi ^2 _{\tau} \lrcorner \dbar u 
\]
as $E$-valued $(n,0)$-forms on $X$ along $X_t$.
\end{prop}

\begin{proof}
The proposition is equivalent to the requirement that $\bar \eta \lrcorner \dbar u = 0$ for any $\eta \in T^{1,0} _{X_t}$.  Since $\eta$ is tangent to $X_t$, $\bar \eta \lrcorner \dbar u = 0$ if and only if $\bar \eta \lrcorner \dbar \iota _{X_t} ^*u = 0$, and the latter holds because $\iota _{X_t} ^*u$ is holomorphic.
\end{proof}

By Proposition \ref{H_t-dbar-is-lift-independent} if $\ff \in \Gamma (B,\sH)$ then the associated section $f := \fa (\ff)$ satisfies 
\begin{equation}\label{dbar=0-claim}
\dbar f (\bar \zeta) = 0 \quad \text{ for all }\zeta \in T^{1,0}_{X/B},
\end{equation}
where $T^{1,0}_{X/B} := {\rm kernel}(dp) \cap T^{1,0}_X  \subset T^{1,0}_X$ is the vertical holomorphic tangent bundle.  One has an exact sequence of anti-holomorphic vector bundles 
\[
0 \to  T^{0,1}_{X/B} \to T^{0,1} _X \to p^*T^{0,1}_B \to 0,
\]
and therefore $\dbar f$ induces a unique section $\bar \eth f \in \Gamma (X,  p ^* T^{*0,1}_B \tensor K_{X/B} \tensor E)$ defined by 
\[
\bar \eth f ((x, \bar \tau)) = \dbar f (x) (\bar \xi_{\tau}) \quad \text{ for any } \xi_{\tau} \in T^{1,0} _{X,x} \text{ such that }dp (x)\xi _{\tau} = \tau \in T^{1,0}_{B, p(x)}.
\]
We therefore obtain the following definition, which is equivalent to the definition of $\dbar$ for $\sH \to B$ given in \cite{bo-annals}.
\begin{defn}[$\dbar$-operator for $\sH$]\label{holo-dbar-defn}
The operator 
\[
\dbar : \Gamma (B,\sC^{\infty}(\sH)) \to \Gamma (B,\sC^{\infty}( \Lambda ^{0,1} _B \tensor \sH))
\]
defined by 
\[
\dbar \ff := \fri (\bar \eth \fa (\ff))
\]
is the $\dbar$-operator for $\sH \to B$.
\end{defn}

\begin{rmk}
Evidently $\ff \in \Gamma (B,\cO(\sH))$ in the sense of Definition \ref{smooth-sections-of-holo-bundle-defn} if and only if $\dbar \ff = 0$, i.e., if and only if $\ff \in \Gamma (B,\cO(\sH))$ in the sense of Definition \ref{general-dbar-and-holo-defn}.  
\red
\end{rmk}

\begin{rmk}\label{H-is-holo-hh}
For higher-bidegree forms $\fb$, expressed in terms of a $\dbar$-closed local frame $\alpha ^1, ... , \alpha ^N$ for $\Lambda ^{p,q} _B$ as $\fb = \alpha ^i \tensor \fb _i$, the operator $\dbar$ is given by 
\[
\dbar \fb  = \alpha ^i \tensor \dbar \fb _i,
\]
and therefore $\dbar \dbar = 0$.  Thus $\sH \to B$ is a holomorphic Hilbert field.
\red
\end{rmk}

\begin{rmk}\label{bo-dbar-defn-rmk}
At first glance, the definition of the $\dbar$ operator for $\sH$ appearing in \cite{bo-annals} might look somewhat different than the definition we presented here.   Here we have used the representation of sections of $\sH$ as relative canonical sections $f \in p_* \cO_X (K_{X/B}\tensor E)_t$, whereas in Berndtsson's approach one chooses an $E$-valued $(n,0)$-form $u \in p_*\cO _X(\Lambda ^{n,0}_X \tensor E)$ representing $f$.  Since $f$ is holomorphic on fibers, the section $\dbar f$ of $\Lambda ^{0,1} _X \tensor K_{X/B} \tensor E$ vanishes on the vertical bundle $T^{0,1} _{X/B}$.  In terms of the representative $u$ of $f$, $\dbar u$, which is a section of $\Lambda ^{n,1} _X \tensor E$, need not vanish on $T^{0,1} _{X/B}$, but the restriction of $\dbar u$ to every fiber of $p$ must vanish.  Therefore, if $\xi _{\tau}$ is any lift of $\tau$ to $X$ one has 
\[
\iota _t (\bar \xi _{\tau} \lrcorner \dbar f) = \iota _{X_t} ^* (\bar \xi _{\tau} \lrcorner \dbar u),
\]
where $\iota _t$ is the isomorphism of Proposition \ref{relative-restriction-prop}.

In terms of the local coordinates discussed in Remark \ref{local-in-the-base-coordinates-rmk}, there exist $(n,0)$-forms $\beta _1,...,\beta _m$ and $(n-1,1)$-forms $\alpha _1,...,\alpha _m$ such that 
\[
\dbar u = dp^{\mu} \wedge \alpha _{\mu} + d\bar p^{\nu} \wedge \beta _{\bar \nu}.
\]
If $\tilde u$ is another representative of $f$ then $\tilde u - u = dp^{\mu} \wedge \gamma _{\mu}$ for some $(n-1,0)$-forms $\gamma _1,..., \gamma _m$, so 
\[
\dbar \tilde u = dp^{\mu} \wedge (\alpha _{\mu} - \dbar \gamma _{\mu})+ d\bar p^{\nu} \wedge \beta _{\bar \nu}.
\]
Therefore the section $\bar \eth f$ is 
\[
\bar \eth f = p^* d\bar s ^{\mu} \tensor [\beta _{\mu}],
\]
where $[\beta]$ means the relative canonical form represented by the $(n,0)$-form $\beta$.  The right-hand side of the latter identity is Berndtsson's definition of $\dbar$ \cite[page 544]{bo-annals}.
\red
\end{rmk}

\subsection{The $\mathbf{L^2}$ BLS fields}\label{structure-of-L}

We shall now construct a large collection of BLS fields each of which contains the Berndtsson Hilbert field $\sH \to B$ of the family $(E,h) \to X \stackrel{p}{\to} B$ as a subfield, and which makes $\sH \to B$ into an iBLS field.  The underlying smooth structure of all of these BLS fields is the same, and it is the complex structure that changes.

\subsubsection*{\sf The underlying Hilbert field}\ 

\medskip

\noi The fibers of the underlying Hilbert field are the $L^2$-spaces 
\[
\sL _t = \left \{ \text{measurable sections } f \in \Gamma (X_t, K_{X_t} \tensor E|_{X_t})\ ; \ (-1)^{\frac{n^2}{2}}\int _{X_t} \left < f \wedge \bar f , h\right > < +\infty \right \}.
\]
Of course, the Hermitian metric is just the $L^2$-metric.

\subsubsection*{\sf The smooth structure}\

\begin{defn} 
A smooth section $\ff$ of $\sL \to B$ is a section $f \in H^0(X, \sC^{\infty} (K_{X/B} \tensor E))$.  We write $\ff \in \Gamma (B, \sC^{\infty}(\sL))$.
\end{defn}
As with the Berndtsson Hilbert field $\sH \to B$, we write $f = \fa (\ff)$ and $\ff = \fri (f)$ to relate the sections of $\sL \to B$ and the sections of $K_{X/B} \tensor E \to X$.

As a consequence of the general theory of Section \ref{bls-section}, differential $k$-forms and $(p,q)$-forms with values in $\sL$ are sections of $\Lambda ^k_B \tensor \sL  \to B$ and $\Lambda ^{p,q}_B \tensor \sL  \to B$ respectively, and can be identified with sections of $p^* \Lambda ^{k}_B  \tensor K_{X/B} \tensor E \to X$ and $p^* \Lambda ^{p,q}_B \tensor K_{X/B} \tensor E \to X$. The discussion is analogous to that of Berndtsson's Hilbert field $\sH$.

It is a basic fact that 
\begin{enumerate}
\item[(i)] the collection of smooth members of $\sL _t$ is dense, and
\item[(ii)] every smooth member of $\sL _t$ has a smooth extension to a section of $K_{X/B} \tensor E \to X$. 
\end{enumerate}
It is also obvious that 
\begin{enumerate}
\item[(iii)] the $L^2$ metric is smooth with respect to $\sC^{\infty} (\sL)$.
\end{enumerate}
Thus $\sL$ is a smooth Hilbert field.  (In fact, as mentioned in Remark \ref{ehresmann-rmk}, by Ehresmann's Theorem $\sL \to B$ is even locally trivial, i.e., it is a smooth Hilbert bundle.)

\subsubsection*{\sf Some reference connections}\

\medskip

Let $D$ be any connection for $E \to X$.   We define the connection $\mathring{\nabla}$ for $\sL \to B$ by the formula 
\begin{equation}\label{D-ref-for-L}
\mathring{\nabla} _{\sigma + \bar \tau} \ff = \fri \left ( L^{D1,0}_{\xi ^{\theta} _{\sigma}}\fa (\ff) \right )+ \fri \left ( L^{D0,1}_{\bar \xi ^{\theta} _{\tau}}\fa (\ff) \right ),
\end{equation}
where $\sigma, \tau = T^{1,0}_{B,t}$. 

\begin{prop}
For any connection $D$ for $E \to X$ the connection $\mathring{\nabla}$ for $\sL \to B$ is a reference connection for the $L^2$ metric, i.e., it satisfies the continuity assumption {\rm \eqref{rrt-cty}} of Definition {\rm \ref{smooth-(q)h-fld-defn}.b}.
\end{prop}

\begin{proof}
Let $\fh$ denote the $L^2$ metrics for $\sL \to B$.  A direct computation shows that 
\[
\mathring{\nabla}_{\sigma + \bar \tau}\fh (\ff_1, \ff_2) = \int _X \left < u_1 \wedge \bar u_2, D h \right >,
\]
and condition \eqref{rrt-cty} obviously holds.
\end{proof}

\subsubsection*{\sf The almost complex structures}

\begin{defn}\label{dbar-theta-defn}
Let $(E,h) \to X \stackrel{p}{\to} B$ be a proper holomorphic family.  For a horizontal distribution $\theta \subset T_X$ and a smooth section $\ff \in \Gamma (B,\sC^{\infty}(\sL))$ define 
\[
\dbar^{\theta} _{\bar \tau} \ff := \fri \left ( \dbar _{\bar \xi ^{\theta}_{\tau}} \fa (\ff) \right ), \quad \tau \in T^{1,0} _B.
\]
\end{defn}
\noi The operator $\dbar ^{\theta}$ is extended to smooth $\sL$-valued differential forms in the usual way.

\begin{prop}\label{integrable-dbar}
Let $p: X \to B$ be a proper holomorphic submersion and let $\theta \subset T_X$ be a horizontal distribution.  Then the following are equivalent.
\begin{enumerate}
\item[{\rm a.}] The almost complex structure $\dbar ^\theta$ is a complex structure, i.e., $\dbar ^{\theta} \dbar ^{\theta} = 0$.
\item[{\rm b.}] The distribution $\theta$ is integrable.
\end{enumerate}
\end{prop}

\begin{proof}%[Proof of Proposition \ref{integrable-dbar}]
By definition of $\dbar^{\theta}$ 
\[
\left ( \dbar^{\theta} _{\bar \sigma}\dbar^{\theta} _{\bar \tau} - \dbar^{\theta} _{\bar \tau}\dbar^{\theta} _{\bar  \bar \sigma} - \dbar^{\theta} _{[\bar \sigma , \bar \tau]} \right ) \ff = \fri \left ( \left ( \dbar_{\bar \xi ^{\theta} _{\sigma}} \dbar_{\bar \xi ^{\theta} _{\tau}}  - \dbar_{\bar \xi ^{\theta} _{\tau}}  \dbar_{\bar \xi ^{\theta} _{\sigma}}  - \dbar_{\bar \xi ^{\theta} _{[\sigma,\tau]}}  \right ) \fa(\ff) \right )
\]
for any $(1,0)$-vector fields $\sigma , \tau$ on $B$.  Now, on any complex manifold one has  
\[
\dbar _{\bar \xi _1} \dbar _{\bar \xi _2} - \dbar_{\bar \xi _2} \dbar_{\bar \xi _1} - \dbar _{[\bar \xi _1,\bar \xi _2]} = 0
\]
for any $(1,0)$-vector fields $\xi _1, \xi _2$.  Since we are acting on sections of a holomorphic vector bundle, the first two terms are always well-defined, and the third term is well-defined because $[\xi _1, \xi _2]$ is also a $(1,0)$-vector field. Therefore 
\[
\left ( \dbar^{\theta} _{\bar \sigma}\dbar^{\theta} _{\bar \tau} - \dbar^{\theta} _{\bar \tau}\dbar^{\theta} _{\bar  \bar \sigma} - \dbar^{\theta} _{[\bar \sigma , \bar \tau]} \right ) \ff = \fri \left ( \left ( \dbar_{[\bar \xi ^{\theta} _{\sigma},\bar \xi ^{\theta} _{\tau}]} - \dbar_{\bar \xi ^{\theta} _{[\sigma,\tau]}}  \right ) \fa(\ff) \right ),
\]
and the right hand side vanishes on all smooth sections if and only if $[\xi ^{\theta}_{\sigma}, \xi ^{\theta}_{\tau}] = \xi ^{\theta}_{[\sigma ,\tau]}$ for all smooth $(1,0)$-vector fields $\sigma , \tau$ on $B$.  In view of Proposition \ref{integrable-horizontal-distribution} the latter holds if and only if $\theta$ is integrable.
\end{proof}

\subsubsection*{\sf The $L^2$ BLS field of $(E,h) \to X \stackrel{p}{\to} B$}\ 

\medskip

\noi Note that if in Formula \eqref{D-ref-for-L} we take the connection $D$ to be complex, i.e., such that $D^{0,1} = \dbar$, then the resulting reference connection $\mathring{\nabla}$ for $\sL ,\dbar^{\theta}$ satisfies $\mathring{\nabla}^{0,1} = \dbar ^{\theta}$.  Therefore all of the properties of a BLS field required by Definition \ref{general-dbar-and-holo-defn}.iv are satisfied for $\sL ,\dbar^{\theta}$.

\begin{defn}
The BLS field $(\sL, \fh , \dbar^\theta) \to B$ is called the \emph{$L^2$ BLS field of the family $(E,h) \to X \stackrel{p}{\to} B$} associated to the horizontal distribution $\theta \subset T_X$.
\end{defn}

The notation $(\sL , \fh, \dbar ^{\theta})$ for the BLS field of the family $(E,h) \to X \stackrel{p}{\to} B$ can become cumbersome.  Since for a family  $(E,h) \to X \stackrel{p}{\to} B$ (which is seen as fixed from the outset, and has already been suppressed from the notation) this BLS structure is completely determined by $\theta$, we shall often simply denote it $\sL ^{\theta}$.

%\begin{rmk}
%We emphasize that in Proposition \ref{integrable-dbar} the horizontal distribution $\theta^{1,0}$ need not itself be integrable!  (But integrable distributions do exist; see  Remark \ref{ehresmann-rmk}.)  
%\red
%\end{rmk}

\subsubsection*{\sf The BLS subfield $\sH$ in the regular case}\

\medskip

In Paragraph \ref{structure-of-H} we defined the $\dbar$-operator for $\sH \to B$ by 
\[
\dbar _{\bar \tau} \ff :=\fri ( \bar \xi ^{\theta} _{\tau} \lrcorner \dbar \fa (\ff)), \quad \tau \in T^{1,0} _B, 
\]
and observed (Proposition \ref{H_t-dbar-is-lift-independent}) that the right hand side is independent of the choice of lift.   Thus we have the following trivial but fundamental observation.

\begin{prop}\label{H-is-always-subbundle}
Let $(E,h) \to X \stackrel{p}{\to} B$ be a holomorphic family of Hermitian vector bundles.  Then for any distribution $\theta \subset T_X$ the operator $\dbar ^{\theta} |_{\sH}$ agrees with the $\dbar$-operator for $\sH$ given in Definition \ref{holo-dbar-defn}.  In particular, if $\sH \to B$ is locally trivial then for any horizontal distribution $\theta$ the holomorphic Hilbert field $(\sH , \dbar)$ is a BLS subfield of the BLS field $\sL ^{\theta}$.
\end{prop}

\subsection{The iBLS Structures of Berndtsson's Hilbert Field}

\begin{prop}\label{holo-infinitesimal}
Let $(E,h) \to X \stackrel{p}{\to} B$ be a holomorphic family of Hermitian vector bundles, fix $t_o \in B$, and let $\theta \subset T_X$ be any horizontal distribution.  Then for any $f \in \sH _{t_o}$ there exists a smooth section $\ff \in \Gamma (B, \sC^{\infty}(\sL^{\theta}))$ such that if $u \in \Gamma (X, \sC^{\infty}(\Lambda ^{n,0}_X \tensor E))$ represents $\fa (\ff)$ then 
\[
\iota _{X_{t_o}} ^* u = f, \quad \iota _{X_{t_o}} ^* \left ( \left < \dbar u , \bar \xi ^{\theta}_{\tau} \right >\right ) = 0 \quad \text{and} \quad \iota _{X_{t_o}} ^* (L^{1,0} _{\xi ^{\theta} _{\tau}}\dbar u ) = 0 \text{ for all }\tau \in \cO(T^{1,0}_B)_{t_o}.
\]
\end{prop}

\begin{proof}
First note that by Proposition \ref{int-ptwise-holo} $\iota _{X_t} ^* L^{1,0} _{\xi ^{\theta} _{\tau}} \dbar u$ is independent of the choice of representative $u$ and the horizontal lift $\xi ^{\theta} _{\tau}$, and that by Proposition \ref{H_t-dbar-is-lift-independent} $\iota _{X_{t_o}} ^* \left ( \left < \dbar u , \bar \xi ^{\theta}_{\tau} \right >\right )$ is independent of the horizontal lift $\xi ^{\theta} _{\tau}$.  It is also straightforward to see that in fact $\iota _{X_{t_o}} ^* \left ( \left < \dbar u , \bar \xi ^{\theta}_{\tau} \right >\right )$ does not depend on the choice of representative $u$ of $f$.  Indeed, in local coordinates on the base $B$ as in Remark \ref{local-in-the-base-coordinates-rmk}, if $u_1$ and $u_2$ are two representatives of $f$ then $u_1-u_2 = dp ^1 \wedge \gamma_1 + ... + dp ^m \wedge \gamma_m$.  Hence $\left < \dbar u_1 , \bar \xi ^{\theta}_{\tau} \right > - \left < \dbar u_2 , \bar \xi ^{\theta}_{\tau} \right > = dp^1 \wedge \left < \dbar \gamma_1 , \bar \xi ^{\theta}_{\tau} \right > + ... + dp ^m \wedge \left < \dbar \gamma_m , \bar \xi ^{\theta}_{\tau} \right >$, and the latter vanishes on the fibers of $p:X \to B$.

By \cite[Proposition 9.5]{voisin} there exists an open set $U \subset B$ containing ${t_o}$ and diffeomorphism $F := ( \psi, p) : p^{-1} (U) \to X_{t_o} \times U$ that is tangent to the identity along ${t_o}$, such that for each $x \in X_{t_o}$ the submanifold $U_x := F^{-1} (\{x\} \times U) \subset p^{-1} (U)$ is complex and $p|_{U_x} : U_x \to U$ is biholomorphic.  And though it's not part of the statement in \cite[Proposition 9.5]{voisin}, $F$ is holomorphic at the fiber $X_{t_o}$, i.e., if $x \in X_{t_o}$ and $h \in \cO _{X,x}$ then $\dbar (F^*h)((x,t_o)) = 0$.  (Indeed, vanishing in the directions of $F^{-1} (\{x\} \times U)$ follows from the construction, while in the vertical directions the map is the identity.)  Let $\theta$ be the integrable distribution on $p^{-1} (U)$ obtained from that tangent bundles $T_{V_x}$, $x \in X_{t_o}$, via the diffeomorphism $F$.  We set $\Phi := F^{-1}$.  

We show now that there is a smooth vector bundle isomorphism $\Psi : \Phi^* E  \to p_1 ^* E_{t_o}$ such that $d\Psi (x)\bar \xi ^{\theta} _{\tau} (x) = 0$ for all $x \in X_{t_o}$.  The idea is similar to the proof of Ehresmann's Theorem: one defines vector fields on the total space of $\Phi^*E$ that project to constant vector fields on $U$.  In order to get a vector bundle map, which is linear on fibers, one constructs vector fields that are linear on fibers.  Fortunately the flows of vector fields that are linear on fibers cannot acquire enough vertical velocity to escape to infinity in finite time in the vertical direction.  Thus lifts of constant vector fields (locally) on the base to vector fields on the vector bundle that are linear along fibers have flows by diffeomorphisms.  The precise construction is as follows.  

Fix a sufficiently small coordinate chart $U \subset B$ and holomorphic coordinates $t$ on $U$ such that $t _o$ is the origin.  Choose a finite open cover $\{O_1,..., O_N\}$ of $X_{t_o}$ by open sets on which $p_1 ^* E_{t_o}$ is trivial, where $p_1 :X_t \times U \to X_t$ is the projection to the first factor, and let $U_i = O_i \times U$. Let $\{\chi _1,..., \chi _N\}$ be a partition of unity subordinate to $\{O_1,..., O_N\}$.  We will not need local coordinates on $X_{t_o}$, but we take coordinates on the fibers by fixing trivializing frames for $p_1 ^* E_{t_o}$ over each $U_j$; if $U$ is small enough then such frames exist.  We write $v_j = (v_j^1,..., v_j ^r)$ for the components of a vector $v \in \Phi ^*E$ in terms of these frames. We let $t_j= (t_j ^1,..., t_j ^m)$ be coordinates in the horizontal direction; they are the global coordinates $t$ coming from the base $U$, but we keep the index because the corresponding constant vector fields $\frac{\di}{\di t_j^{\mu}}$ on the trivial vector bundle over the chart $U_j$ depend on $j$.  The transition functions for these coordinates on the total space of $\Phi ^* E$ are 
\[
(x, t_j , v_j) = (x, t_k , F_{jk} (x,t_k)v_k),
\]
where $F_{jk}(x, t_k)$ are invertible linear transformations that depend smoothly on $x$ and $t_k$, and satisfy 
\[
\frac{\di }{\di \bar t^{\nu} _k} F_{jk}(x, t_k) = 0
\]
for all $1 \le \nu \le m$ and $x \in O_k$.  Now define the vector fields 
\[
\xi _{\nu} ^{(j)} (x, t_j, v_j) := \frac{\di}{\di t ^{\nu} _j}. 
\]
Then on the fibers over $U_j\cap U_k$ one has 
\[
\frac{\di}{\di t ^{\nu} _k} = \frac{\di}{\di t^{\nu}_j} + \frac{\di (F_{jk} (x,t_k)v_k)^{\mu}}{\di t_j ^{\nu}} \frac{\di }{\di v_j ^{\mu}}= \frac{\di}{\di t^{\nu}_j} + \left ( \frac{\di F_{jk} (x,t_k)}{\di t_j ^{\nu}}F_{kj}(x,t_j) v_j\right )^{\mu} \frac{\di }{\di v_j ^{\mu}}.
\]
Therefore on fibers over $U_j\cap U_k$ the difference of vector fields 
\[
\xi ^{(k)}_{\nu} (x,t,F_{j,k} (x,t)v_j) - \xi ^{(j)}_{\nu} (x,t,v_j)
\]
is linear in $v_j$ and holomorphic in $t$.  It follows that the vector fields
\[
\eta _{\nu} (v) := \sum _j \chi _j (x) \xi ^{(j)} _{\nu} (x,t_j, v_j)
\]
depend holomorphically on $t$ and vary linearly in $v$ in the fiber directions.  By construction, these vector fields project onto the coordinate vector fields, and by integrating these vector fields we therefore get diffeomorphisms $\Psi ^s_{\nu} \in {\rm Diff}((\Phi ^{-1}) ^*E)$ that are linear on fibers and that preserve the trivial family $X \times U \to U$.  In fact, if $\phi :\Phi^*E \to X \times U$ is the vector bundle projection then $\phi \Psi^s_{\nu} (x,t, v) = (x, t+s)$.  Consider the diffeomorphism $\Psi : p_1 ^* E_t \to (\Phi ^{-1}) ^* E$ defined by 
\[
\Psi (v_o, t) := \Psi _1 ^{t^1}\circ \cdots \circ \Psi _m ^{t^m} (v_o).
\]
Evidently 
\[
\left . \frac{\di}{\di \bar t ^{\mu}} \Psi (v_o , t) \right | _{t=0}= 0
\]
and 
\begin{eqnarray*}
\frac{\di}{\di t^{\mu}}\frac{\di}{\di \bar t^{\nu}} \Psi (v_o, t) &=&\frac{\di}{\di \bar t^{\nu}}  \frac{\di}{\di t^{\mu}}\Psi (v_o, t)\\
&=& \frac{\di}{\di \bar t^{\nu}}\left ( \prod _{\lambda =1} ^{\nu -1} d\Psi^{t^{\lambda}} (\Psi _{\lambda+1} ^{t^{\lambda+1}}\circ \cdots \circ \Psi _m ^{t^m} (v_o))  \right ) \eta _{\mu} (\Psi _{\mu} ^{t^{\mu}}\circ \cdots \circ \Psi _m ^{t^m} (v_o)),
\end{eqnarray*}
and since $\Psi_j ^0 = {\rm Id}$ we have 
\[
\left . \frac{\di}{\di t^{\mu}}\frac{\di}{\di \bar t^{\nu}} \Psi (v_o, t) \right |_{t=0}= 0.
\]
Finally, to get the desired section $u$, we take the constant extension of $f$ in the trivial family and pull it back to an $E$-valued $(n,0)$-form via the diffeomorphisms $\Psi$ and $F$.  We need to check that $\iota _{X_{t_o}}^*(\bar \xi ^{\theta} _{\tau} \lrcorner \dbar u) = 0$ and that $L^{1,0} _{\xi ^{\theta} _{\tau}} \dbar u = 0$.  Writing $u_o (x,t) := f(x)$ for the trivial extension of $f$, so that $u = F^* \Psi ^* u_o$, our construction yields $\iota _{X_{t_o}}^*u = f$ and, at the points of $X_{t_o}$, 
\[
\bar \xi ^{\theta} _{\tau}\lrcorner \dbar u = \bar \xi ^{\theta} _{\tau}\lrcorner (\dbar (F^* \Psi ^* u_o)) = F^* \Psi ^*\left ( \tfrac{\di }{\di \bar \tau} \lrcorner \dbar u_o\right ) = 0
\]
and
\[
L^{1,0} _{\xi ^{\theta} _{\tau}} \dbar u = L^{1,0} _{\xi ^{\theta} _{\tau}} F^* \Psi ^*\dbar u_o = F^* \Psi ^*L^{1,0} _{\tau ^{\mu} \frac{\di}{\di t_{\mu}}} \dbar u_o =0.
\]
The proof is complete.  
\end{proof}

We can now collect all we have done into the following simple but important theorem. (See Definitions \ref{substalk-defns} and \ref{iBLS-defn}.i.)

\begin{thm}\label{iBLS-structures-for-H}
Let $(E,h) \to X \stackrel{p}{\to} B$ be a proper holomorphic family of Hermitian holomorphic vector bundles and let $\theta \subset T_X$ be any horizontal distribution.  Define the substalk bundle $\Sigma \subset \tilde \sH _{\sL^{\theta}}$ by 
\[
\Sigma _t := \left \{ \ff \in \sC^{\infty} (\sL)\ ;\ \iota _{X_t} ^* u \in \sH_t \text{ and } L^{1,0}_{\xi ^{\theta} _{\tau}} \dbar u = 0 \text{ for any }u \in \sC^{\infty} (\Lambda ^{n,0}_X \tensor E)_t \text{ representing }\ff \right \}.
\]
Then $\Sigma$ is independent of $\theta$, and formalizes $\sH$.  Moreover $(\sH,\Sigma, \sL^{\theta})$ is an integrable iBLS field. 
\end{thm}

\begin{proof}
By Proposition \ref{int-ptwise-holo}.a $\Sigma$ is well-defined and by Proposition \ref{int-ptwise-holo}.b $\Sigma$ is independent of $\theta$.  By Proposition \ref{holo-infinitesimal}, every $f \in \sH _t$ extends to a smooth section of $\sL \to B$ that is holomorphic to second order at $t$.  Therefore $\sH \subset \sL ^{\theta}$ is bona fide, and by Proposition \ref{int-ptwise-holo}.c $\Sigma$ formalizes $\sH \subset \sL ^{\theta}$. Moreover, by Proposition \ref{holo-infinitesimal} $\iota ^*_{X_t} \bar \xi ^{\theta}_{\tau} u = 0$ for any representative $u$ of $f$, so $(\dbar ^{\theta} \ff )(t) = 0 \in \Lambda ^{0,1} _{B,t} \tensor \sH _t$.  Finally, the fact that $\sH$ is integrable follows the proof of Proposition \ref{holo-infinitesimal}; the reader can check that any representative $u$ of the section $\ff \in \Sigma _{t_o}$ satisfies $\left < \iota _{X_{t_o}}^* \dbar \dbar u ,\bar \xi ^{\theta}_{\sigma} \wedge \bar \xi ^{\theta}_{\tau} \right > = 0$ for all $\sigma , \tau \in T^{1,0} _{B, t_o}$.
\end{proof}

In fact, $(\sH , \Sigma, \sL ^{\theta})$ is strongly iBLS.  That is to say, the Gauss-Griffiths form 
\[
\bbG (\ff, \fg) := \left ( \two ^{\fL^{\theta} / (\sH, \Sigma)}\ff, \two ^{\fL^{\theta} / (\sH, \Sigma)}\fg \right )
\]
has the continuity property stated in Definition \ref{iBLS-defn}.d.  This fact is a consequence of the identity \eqref{bly-formula}, whose proof is completed in Section \ref{sff-section}.

\subsection{Lift-independence of the relative curvature}\label{lift-indep-par}

For a choice of horizontal lift $\theta$ the curvature of the BLS field $\sL^{\theta}$ is 
\[
\Theta ^{\sL^{\theta}}  = \nabla ^{\sL ^{\theta}} \nabla ^{\sL ^{\theta}},
\]
where $\nabla ^{\sL^{\theta}}$ denotes the Chern connection of $\sL^{\theta}$.   One then has the relative curvature $\Theta ^{({\rm rel}\ \sL ^{\theta})}(\sH)$ (see Definition \ref{iBLS-defn}.iii) given by 
\[
\left ( \Theta ^{({\rm rel}\ \sL ^{\theta})}(\sH) _{\sigma \bar \tau}f_1,f_2\right ) := \left ( \Theta^{\sL^{\theta}} _{\sigma \bar \tau}f _1,f_2\right ) - \left ( \two ^{\sL^{\theta}/\sH}(\sigma)f_1, \two ^{\sL^{\theta}/\sH}(\tau) f_2 \right ), \quad t \in B,\ f_1,f_2 \in \sH _t.
\]
The next result is key to a natural definition of the curvature of the iBLS field $\sH \to B$.

\begin{thm}\label{indep}
Let $(E,h) \to X \stackrel{p}{\to} B$ be a holomorphic family of vector bundles. If $\theta _1$ and $\theta _2$ are two horizontal lifts for $p$ then 
\[
\Theta ^{({\rm rel}\ \sL ^{\theta_1})}(\sH) =  \Theta ^{({\rm rel}\ \sL ^{\theta_2})}(\sH).
\]
%for all $t \in B$, all $\sigma , \tau \in T^{1,0} _{B,t}$ and all smooth sections $\ff _1, \ff_2 \in \Gamma (B,\sC^{\infty}(\sL^{\theta_1})) = \Gamma (B,\sC^{\infty}(\sL^{\theta_2}))$ such that $\ff_1(t), \ff_2(t) \in \sH_t$.
\end{thm}

\noi The proof of Theorem \ref{indep} is given in Section \ref{B2-thm-pf-section}, after an explicit formula for the Chern connection of $\sL^{\theta}$ is obtained.  

\begin{defn}\label{H-curvature-definition}
The curvature of $\sH \to B$ with its $L^2$ metric is 
\[
\Theta ^{\sH}  :=  \Theta ^{({\rm rel}\ \sL ^{\theta})}(\sH),
\]
where $\theta$ is any horizontal lift for $p :X \to B$.
\end{defn}

\subsection{Positivity of curvature for section fields}

In this section we present some known but slightly less common multilinear algebra useful in the study of the curvature of $\sH \to B$.

\subsubsection*{\sf Quadratic forms on $2$-tensors}

Positivity of the curvature $\Theta (\fh)$ of a holomorphic vector bundle $\fH \to B$ with metric $\fh$ is a positivity property of the Hermitian form $\{ \cdot , \cdot \}_{\fh , \Theta (\fh)}$ on the vector bundle $T^{1,0} _B \tensor \fH$ defined in Paragraph \ref{positivity-par}.   One can generalize the latter notion of positivity to any Hermitian quadratic form $A$ on a tensor product of vector spaces $M \tensor F$.  (We use the letters $M$ and $F$ to indicate that the first and second factors correspond respectively to vectors in the manifold and fiber.)  The elements of $M \tensor F$ can be thought of as linear maps $M^* \to F$ (or $F^*\to M$), and as such have a well-defined notion of rank.  One can then write $M \tensor F$ as an increasing union of subvarieties 
\[
\{0\} =:D_o(M \tensor F) \subset D_1(M \tensor F) \subset  \cdots \subset D_k(M \tensor F) \subset  ... ,
\]
where $D_k(M \tensor F)$ denotes the set of tensor of rank $\le k$.  Since the rank is at most $\min \{ \dim M, \dim F)$, the subvarieties $D_k(M \tensor F)$ stabilize.  The sub-linearity of the rank also implies that 
\[
D_k(M \tensor F) + D_{k'}(M \tensor F) \subset D_{k+k'}(M \tensor F).
\]
The subvarieties $D_k(M \tensor F)$ are cones, and are non-linear if $k \neq 0$ or $\min (\dim M, \dim F)$.  Moreover, the generic element of $M \tensor F$ has maximal rank, i.e., $D_{\min (\dim M, \dim F)}(M \tensor F)$ is Zariski open.  

\begin{defn}
Let $M$ and $F$ be finite-dimensional vector spaces and suppose $F$ is equipped with a Hermitian metric $\phi$.
\begin{enumerate}
\item[a.]  A Hermitian quadratic form $A$ on $M \times F$ is said to be $k$-positive if it is positive on $D_k (M \tensor F)$, i.e., 
\[
D_k (M \tensor F) \ni \xi \mapsto A(\xi ,\xi) \ge 0,
\]
with equality if and only if $\xi = 0$.  (If the latter is omitted, we say $A$ is $k$-nonnegative.)
\item[b.] A tensor $Q \in {\rm Hom} (F,F) \tensor M \wedge \bar M$ is said to be Hermitian for $\phi$ if for all $\xi, \eta \in M$ the endomorphism $Q _{\xi \bar \eta} := \left < Q, \xi \tensor \bar \eta\right >$  satisfies 
\[
\phi (Q _{\xi \bar \eta} v, w) = \phi (v, Q _{\xi \bar \eta} w).
\]

\item[c.] A Hermitian tensor $Q \in {\rm Hom} (F,F) \tensor M \wedge \bar M$ is said to be $k$-positive if the Hermitian quadratic form 
\[
(M \tensor F) \tensor( \bar M \tensor \bar F) \ni (\xi \tensor v, \bar \eta \tensor \bar w) \mapsto \{ (\xi \tensor v, \bar \eta \tensor \bar w)\} _{Q, \phi} := \phi ( Q_{\xi \bar \eta} v, w)
\]
is $k$-positive.
\item[d.] A $1$-positive Hermitian form is also called Griffiths-positive, and a Hermitian form that is $k$-positive for every $k \ge 1$ is also called Nakano-positive.
\end{enumerate}
\end{defn}

\noi Thus the vector bundle $(\fH , \fh) \to B$ has $k$-positive curvature (Definition \ref{k-pos-curv-defn}) if and only if its curvature $\Theta (\fh)$ is $k$-positive if and only if $\{ \cdot , \cdot \}_{\fh , \Theta (\fh)}$ is $k$-positive. 

These definitions can be extended to the infinite-rank case, but doing so requires some care.  We shall not use these ideas in the infinite rank case.

\subsubsection*{\sf Positivity for Schur complements}

We wish to consider the situation in which the subspace $M$ has a direct sum decomposition $M = M_1 \oplus M_2$.  In this `split' context we shall need a positivity statement that is best expressed in terms of the inertia map.

Recall that for a Hermitian vector space $W$ with Hermitian metric $h$, one can associated to a Hermitian form $\Phi$ its \emph{inertia endomorphism} $\fI (\Phi) = \fI _h(\Phi)$ defined by 
\[
h ( \fI (\Phi) v, w) = \Phi (v,w).
\]
The \emph{inertia map} $\fI _h$ is an isomorphism from the set of quadratic Hermitian forms on $W$ onto the set of $h$-Hermitian endomorphisms of $W$.

We take the vector space $W$ to be $M \tensor F$, where $M= M_1 \oplus M_2$.  Suppose $F$ has a Hermitian inner product $\phi$ as before, and assume $M$ is equipped with a Hermitian inner product $g$ such that $M_1$ and $M_2$ are $g$-orthogonal.  (Write $g_i := g|_{M_i}$.)  We equip $M \tensor F$ with the product metric $g \boxtimes \phi$.  Then the inertia map $\fI = \fI _{g \boxtimes \phi}$ decomposes into four maps $\fI _{ij}$, $1 \le i,j \le 2$, such that for any Hermitian quadratic form $\Phi$ the maps 
\[
\fI _{ij} (\Phi) : M_i \tensor F \to M_j \tensor F
\]
satisfy $\fI _{ij} (\Phi) ^{\dagger} = \fI_{ji}(\Phi)$, where $\dagger$ denotes Hermitian conjugation with respect to the relevant Hermitian metrics.

\begin{defn}
The Schur complement of $(\Phi, M_1\tensor F)$ is the Hermitian form on $M_1 \tensor F$ whose inertia map with respect to the metric $g_1 \boxtimes \phi$ is
\[
\fI _{11}(\Phi) - \fI_{12} (\Phi)\fI_{22}(\Phi)^{-1} \fI _{21}(\Phi) : M_1 \tensor F \to M_1 \tensor F.
\]
\end{defn}

%The result for which we will have particular use can be stated as follows.

\begin{prop}\label{positivity-of-dets}
Let $\Phi$ be an $m$-positive Hermitian form on $(M_1 \oplus M_2) \tensor F$ such that $\Phi|_{M_2 \tensor F}$ is Nakano-positive.  Then the Schur complement of $(\Phi, M_1\tensor F)$ is $m$-positive.
\end{prop}

\begin{proof}
Choose $f_i, f_j \in F$, and define $f_j^{\dagger} \in F^*$ by $f_j ^\dagger f = \phi (f,f_j)$ for all $f \in F$.  Also for $\xi \in M$ define $\xi^{\ddagger} \in M^*$ by $\xi ^\ddagger \eta = g(\eta,\xi)$ for all $\eta \in M$.  For $\xi_k= (\tau_k, w_k),\xi_{\ell} = (\tau _{\ell},w_{\ell}) \in M = M_1 \oplus M_2$ we have 
\begin{eqnarray*}
\Phi (\xi _k \tensor f_i, \xi _{\ell} \tensor f_j) &=& \tau _{\ell} ^{\ddagger} \left (f_j ^{\dagger} \fI _{11} f_i\right ) \tau _k  +  w_{\ell} ^{\ddagger} \left (f_j ^{\dagger} \fI _{12} f_i\right ) \tau _k +  \tau_{\ell} ^{\ddagger} \left (f_j ^{\dagger} \fI _{21} f_i\right ) w _k +  w_{\ell} ^{\ddagger} \left (f_j ^{\dagger} \fI _{22} f_i\right ) w_k \\
&=& \tau _{\ell} ^{\ddagger} \left (f_j ^{\dagger} \fI _{11} f_i\right ) \tau _k  +  w_{\ell} ^{\ddagger} \left (f_j ^{\dagger} \fI _{12} f_i\right ) \tau _k +  \tau_{\ell} ^{\ddagger} \left (f_j ^{\dagger} \fI _{12}^{\ddagger} f_i\right ) w _k +  w_{\ell} ^{\ddagger} \left (f_j ^{\dagger} \fI _{22} f_i\right ) w_k,
\end{eqnarray*}
where the second equality holds because $\fI(\Phi)$ is Hermitian.  We can therefore define the $(1,1)$-form (a.k.a. sesquilinear form)
\[
\Phi _{i\bar j} (\xi _k,\bar \xi _{\ell}) := \Phi (\xi _k \tensor f_i,\xi _{\ell} \tensor f_j)
\]
Now let us take $w_k := - \fI_{22} (\Phi)^{-1} \fI_{21}(\Phi) \tau _k$, $1 \le k \le m$.  Then 
\[
\Phi _{i \bar j} (\xi _k, \xi _{\ell}) = \tau _{\ell} ^{\ddagger} \tensor f_j^{\dagger}\left ( \fI _{11}(\Phi) - \fI_{12}(\Phi) \fI_{22} ^{-1}(\Phi) \fI_{21}(\Phi) \right ) f_i \tensor \tau _k.
\]
Setting $k=i$ and $\ell = j$, and then summing, we obtain 
\[
\sum _{i,j=1} ^m \Phi _{i \bar j} (\xi _i , \bar \xi _j) = \sum _{i,j=1} ^m  \tau _{j} ^{\ddagger} \tensor f_j^{\dagger}\left ( \fI _{11} (\Phi)- \fI_{12} (\Phi)\fI_{22} ^{-1}(\Phi) \fI_{21}(\Phi) \right ) f_i \tensor \tau _i,
\]
and the result follows.
\end{proof}

\subsection{Semi-local descriptions of sections and representatives}\label{local-reps-paragraph}

For later computation it is useful to describe the actions of $\nabla ^{1,0}$ and $\dbar^{\theta}$ in terms of the description of smooth sections of $\sL \to B$, locally over the base, as $E$-valued $(n,0)$-forms on the total space $X$.  We fix a local coordinate chart $U\subset B$ and write $p= (p^1,...,p^m)$ as in Remark \ref{local-in-the-base-coordinates-rmk}.  Fix a representative $E$-valued $(n,0)$-form $u$ of a smooth section $\ff$ of $\sL \to B$.  Then $\nabla ^{1,0}$ is an $E$-valued $(n+1,0)$-form, so there exist global $(n,0)$-forms $\phi _1,...,\phi _m$ in $p^{-1} (U)$ such that 
\begin{equation}\label{nabla10-of-L-section-decomposition}
\nabla ^{1,0} u = dp ^1 \wedge \phi _1 + \cdots + dp ^{\mu} \wedge \phi _{\mu}.
\end{equation}
The forms $\phi _{\mu}$ depend on the representative $u$, but their restrictions to the fibers, i.e., 
\[
\iota _{X_t} ^* \phi _{\mu}, \quad \mu =1,...,m,\  t \in U,
\]
are uniquely determined by $\ff$, i.e., they are independent of the representative $u$ of $\ff$.

For $\dbar$ the situation is just slightly more subtle.  One has 
\begin{equation}\label{dbar-of-L-section-decomposition}
\dbar u = \psi + dp ^1 \wedge \alpha _{1} + \cdots + dp ^m \wedge \alpha _{m} +  d\bar p ^1 \wedge \beta _{\bar 1} + \cdots + d\bar p ^m \wedge \beta _{\bar m}
\end{equation}
for some $E$-valued $(n-1,1)$-forms  $\alpha _{1},..., \alpha _{m}$ and $E$-valued $(n,0)$-forms $\beta _{\bar 1},...,\beta _{\bar m}$ (that depend on the choice of local coordinates in $B$ around $t$) and some $(n,1)$-form $\psi$.  The forms $\alpha _1,..., \alpha _m, \beta_{\bar 1}, ... \beta _{\bar m}$ and $\psi$ are not unique, but for each $t \in U$ the forms 
\[
\iota _{X_t} ^*\left (\alpha _{\mu} +  \xi ^{\theta} _{\tfrac{\di}{\di t^{\mu}}} \lrcorner \psi \right ) \quad \text{and} \quad \iota _{X_t} ^*\left ( \beta _{\bar \mu} + \bar \xi ^{\theta} _{\tfrac{\di}{\di t^{\mu}}} \lrcorner \psi \right ), \qquad \mu =1,..., m,
\]
are uniquely determined.

The expression for $\dbar$ gets slightly simpler if one is looking at sections of $\sH \to B$.  (The case of sections of $\sH$ was explained in Remark \ref{bo-dbar-defn-rmk}, but we collect it here again for convenience.)
\begin{prop}\label{rep-of-holo}
Let $E \to X \stackrel{p}{\to} B$ be a holomorphic family of holomorphic vector bundles and let $\ff \in \Gamma (B, \sL)$.  Let $u \in H^0(X, \sC^{\infty}(\Lambda ^{n,0}_X \tensor E))$ be a representative of $\ff$.  Fix a coordinate neighborhood $U \subset B$ and write $p = (p^1,..., p^m)$ in these coordinates.
\begin{enumerate}
\item[{\rm a.}]  For some $t \in U$ one has $\ff (t) \in \sH _t$ if and only if $\iota _{X_t} ^* \psi = 0$ in \eqref{dbar-of-L-section-decomposition}.
\item[{\rm b.}]  The section $\ff$ of $\sL$ restricts to a section of $\sH$ over $U$ if and only if for any (and hence every) representative $u$ one has \eqref{dbar-of-L-section-decomposition} with $\psi \equiv 0$.
\item[{\rm c.}] The section $\ff$ of $\sL$ is a holomorphic section of $\sH|_U$ if and only if for any (and hence every) representative $u$ one has \eqref{dbar-of-L-section-decomposition} with $\psi \equiv 0$, $\beta _{\bar \nu} \equiv 0$, $1 \le \nu \le m$.
\end{enumerate}
\end{prop}

\begin{proof} Since $\iota _{X_t} ^* \psi = \iota _{X_t} ^* \dbar u =\dbar \iota _{X_t} ^*u$, (a) follows.  From (a) we see that $\ff$ restricts to a section of $\sH$ over $U$ if and only if $\iota _{X_t} ^* \psi = 0$ for every $t \in U$.  We claim that the latter holds if and only if $\psi = dp ^{\mu} \wedge \alpha' _{\mu}+d\bar p ^{\nu} \wedge \beta'_{\bar \nu}$ for some $E$-valued $(n-1,1)$-forms  $\alpha' _{\mu}$ and $E$-valued $(n,0)$-forms $\beta' _{\bar \nu}$, $1 \le \mu,\nu \le m$.  Indeed, this is clearly true locally (but with the global forms $dp^{\mu}$ and $d\bar p ^{\nu}$, $1 \le \mu,\nu \le m$), and the local coefficient forms may be patched together using a partition of unity.  This proves (b).  Finally, suppose $\ff$ is a section of $\sH$ over $U$. By (b) we can write $\dbar u = dp ^{\mu} \wedge \alpha _{\mu}  +  d\bar p ^{\nu} \wedge \beta _{\bar \nu}$.  Then $\ff$ is a holomorphic section if and only if for each $1 \le \mu \le m$ and every lift $\xi _{\mu}$ of the $(1,0)$-vector field $\tfrac{\di}{\di t^{\mu}}$ on $U$ one has 
\[
0 = \bar \xi _{\mu}\lrcorner \dbar u = \beta _{\bar \mu} - d\bar p ^{\nu}  \wedge (\bar \xi _{\mu} \lrcorner \beta _{\bar \nu}).
\]
Thus $\iota _{X_t} ^* \beta _{\bar \mu} \equiv 0$ for all $t \in U$.  But for an $(n,0)$-form we already know that this condition is equivalent to the existence of a decomposition $\beta _{\bar \mu} = dt ^{\nu} \wedge \gamma _{\nu \bar \mu }$, which proves (c).
\end{proof}

\noi We emphasize that the forms $\phi _{\mu}$ associated to $\nabla ^{1,0} u$ and $\alpha _{\mu}$,$\beta_{\bar \mu}$, and $\psi$ associated to $\dbar u$ are \emph{global} on the preimage $p^{-1} (U)$ of the coordinate chart $U$.

\section{Local Triviality}\label{loc-triv-section}

\subsection{$L^2$ Extension}\label{OT-par}

The question of local triviality of vector bundles like $\sH \to B$ has been of interest since the beginning of the theory of deformations of complex manifolds introduced by Kodaira and Spencer.  The traditional approach is based on the observation that if the dimensions of the fibers are constant then the fibration is locally trivial.  

While this perspective is broad in its scope, its application can be very difficult in most cases.  But a new approach has become very popular in the past two decades or so.  It is well-known that the dimensions of the fibers of $\sH \to B$ vary upper semi-continuously with the point of the base $B$.  Thus to prove local triviality, one need only prove lower semi-continuity of these dimensions.  The latter is a consequence of the existence of linear extension operators.  A powerful tool that has many other applications is the Ohsawa-Takegoshi Theory of $L^2$-Extension.  This theory consists of a huge and growing collection of extension theorems obtained by $L^2$ methods of one kind or another.  For applications in the present article the following result can be used.

\begin{thm}\label{ot-tak}
Let $Y$ be a weakly pseudoconvex K\"ahler manifold of complex dimension $N$, let $Z \subset Y$ be a closed complex submanifold of codimension $1$, and let $E \to Y$ be a holomorphic vector bundle.  Assume there exist smooth Hermitian metrics $h$ for $E$ and $e^{-\lambda}$ for the line bundle $L_Z \to Y$ associated to $Z$, a holomorphic section $T \in H^0 (Y, \cO (L_Z))$ whose zero divisor is $Z$ counting multiplicity, and a number $\delta > 0$ such that 
\[
\sup _Y |T|^2e^{\lambda} \le 1 \quad \text{and} \quad \ii \Theta (h) \ge \delta  \ii \di \dbar \lambda \tensor {\rm Id}_E \ge 0 \text{ in the sense of Nakano}.
\]
Then for every holomorphic section $f \in H^0(Z, \cO (K_Z \tensor E|_Z ))$ such that 
\[
\int _Z \ii ^{(N-1)^2}\left < f \wedge \bar f, h \right > < +\infty
\]
there exists $F \in H^0(Y, \cO (K_Y \tensor E \tensor L_Z))$ such that 
\[
F|_Z = f \wedge dT \quad \text{and} \quad \int _Y \ii ^{N^2} \left < F \wedge \bar F, h \tensor e^{-\lambda} \right > \le C_{\delta}  \int _Z \ii ^{(N-1)^2}\left < f \wedge \bar f, h \right > .
\]
\end{thm}

If the manifold $Y$ in Theorem \ref{ot-tak} is compact K\"ahler then perhaps the shortest proof of Theorem \ref{ot-tak} is via Berndtsson's method of currents; see the appendix of \cite{bo-annals} in the case in which $E$ has rank $1$.  A proof in the higher rank case can be found in \cite{raufi-OT}.  %A different proof, based on a beautiful and slightly more recent construction of Berndtsson and Lempert \cite{bl}, can be found in \cite{v-bo-1-vb}.

\subsection{Sufficient conditions for local triviality}

We start with Berndtsson's local triviality result mentioned in the introduction.

\begin{prop}\label{ot-triv}
Let $(E,h) \to X \stackrel{p}{\to} B$ be a proper K\"ahler family of holomorphic Hermitian vector bundles.  If the metric $h$ has non-negative curvature in the sense of Nakano then $\sH$ is a holomorphic vector bundle.
\end{prop}

\begin{proof}
Any of the $L^2$ extension theorems mentioned in the previous paragraph imply that the dimensions of the fibers $\sH _t$ vary lower semicontinuously with $t$.  Upper semicontinuity, which holds for all smooth metrics $h$ regardless of their curvature, is a consequence of Montel's Theorem.
\end{proof}

\begin{cor}\label{nak-pos-fam}
Let $(E,h) \to X \stackrel{p}{\to} B$ be a proper K\"ahler family of holomorphic Hermitian vector bundles.  Assume that $\Theta (\fh) \ge - C p^*\omega \tensor {\rm Id}_{E}$ for some positive Hermitian form $\omega$ on $B$ and some positive continuous function $C$ on $X$.  Then $\sH$ is a holomorphic vector bundle.
\end{cor}

\begin{proof}
We may assume that the base $B$ is the unit ball and that the metric $h$ is defined up to the boundary $p^{-1} (\di B)$ of $X$.  Then for $M >> 0$ the metric $h_M := e^{-M|t|^2} h$ for $E \to X$ has Nakano-nonnegative curvature, and hence by Proposition \ref{ot-triv} the associated Hilbert field $\sH _M \to B$ is locally trivial.  But the underlying vector bundles for $\sH_M$ and $\sH$ are the same.
\end{proof}

\subsection{Examples of Non-locally trivial $\sH$}\label{non-loc-triv-examples}

There are many examples of smooth proper holomorphic families $E \to X \stackrel{p}{\to} B$ for which $\sH \to B$ is not locally trivial.  

\subsubsection*{\sf Trivial family of manifolds} The simplest examples come from trivial holomorphic families $p: X_o \times \bbD \to \bbD$ with $X_o$ a compact  Riemann surface of genus $g \ge 1$ together with a non-trivial family of holomorphic line bundles $E \to X_o \times \bbD$, where $\bbD$ denotes the unit disk.  The examples are based on the following simple

\medskip 

 \noi {\sf Fact:}  \emph{There is no meromorphic function on $X_o$ with exactly one pole, whose order is $1$.}  
 
 \medskip
 
 \noi Indeed, since the number of preimages of any point under a nonconstant holomorphic map of compact Riemann surfaces is independent of the point, the existence of such a meromorphic function would yield an isomorphism $X_o \cong \bbP_1$, which is impossible.

It follows that the line bundle $L_{D}$ associated to the divisor $D=p-q$ is either trivial, i.e., $p=q$, in which case $H^0 (X_o, \cO (L_D))= \bbC$, or else  $p \neq q$, in which case $H^0 (X_o, \cO_{X_o}(L_D))= \{0\}$.  Indeed, suppose $p \neq q$ and $s \in H^0(X_o, \cO (L_D))$.  Let $s_D$ be a meromorphic section with divisor $D$.  If $s$ does not vanish at $p$ then the function $s/s_D$ is meromorphic with exactly one pole, which is impossible.  Therefore $s$ vanishes at $p$, in which case $s/s_D$ is holomorphic, hence constant.  But $s/s_D$ vanishes at $q$, so $s/s_D \equiv 0$.  Hence $s=0$ as claimed.

Now let $f : \bbD \to X_o$ be an injective holomorphic map with $f(0) = p$ and consider the holomorphic family of line bundles $E_f \to X_o \times \bbD$ defined by 
\[
E _f :=  K^*_{X_o \times \bbD} \tensor L_{D_f}, 
\]
where $D_f$ is the divisor on $X_o \times \bbD$ defined by 
\[
D_f := \{p\} \times \bbD - \{ (f(t),t)\ ;\ t \in \bbD\}
\]
and $p_1 : X_o \times \bbD \to X_o$ is the Cartesian projection.  Then 
\[
\sH _t = H^0(X_o, \cO (L_{p - f(t)})) = \left \{ 
\begin{array}{l@{,\ }l}
\bbC & t=0\\
\{0\} & t \in \bbD - \{0\}
\end{array}
\right . .
\]
Thus $\sH \to \bbD$ is not locally trivial.

\subsubsection*{\sf Nontrivial families of manifolds}  Another very simple example is constructed as follows.  Let $\mathbf{H} \subset \bbC$ denote the (open) upper half plane and consider the family of tori 
\[
X := (\bbC \times \mathbf{H})/ \sim, 
\]
where $(z,t) \sim (\zeta , \tau) \iff t = \tau$ and $\zeta - z= m + n t$ for some $m,n  \in \bbZ$.  Of course $p :X \to \mathbf{H}$ is the map induced by the projection $\bbC \times \mathbf{H} \to \mathbf{H}$.  The fibers $X_t$ are tori, which we identify with the quotient $\bbC/(\bbZ + t \bbZ)$.  

Let $D_a$ denote the smooth hypersurface on $X$ obtained from $\{ a\} \times \mathbf{H} \subset \bbC \times \mathbf{H}$ after passing to the quotient.  We take $E$ to be the line bundle associated to the divisor $D := D_0 - D_{\ii}$.  The two smooth hypersurfaces $D_0$ and $D_{\ii}$ intersect at the points $[(0,t)] \in X$ whenever $m +nt  = \ii$ for some $m,n \in \bbZ$ (such $t$ form a closed discrete subset of $\mathbf{H}$), but are otherwise are disjoint.  Denote by $B \subset \mathbf{H}$ the set of such $t$.  Then the divisors $\Delta _t := D \cap X_t$ are trivial if $t \in B$, and otherwise are $[0]- [t]$.  

Since the canonical bundle of any torus is trivial, we see that 
\[
\sH _t = H^0(X_t, \cO (E_t)) = \left \{ 
\begin{array}{l@{,\ }l}
\bbC & t\in B\\
\{0\} & t \in \mathbf{H} \setminus B
\end{array}
\right . .
\]

\section{Elements of Hodge Theory}\label{Hodge-thy-par}

The twisted Laplace Beltrami operator for a Hermitian holomorphic vector bundle $(E,h)$ on a Hermitian manifold $(X,g)$ is 
\[
\Box := \dbar \dbar ^* + \dbar ^* \dbar;
\]
while the $\dbar$-operator depends only on $X$ and $E$, and not on the metrics, the formal adjoint $\dbar ^*$ does carry information about the metrics $g$ and $h$.  In particular, the Kernel 
\[
{\rm Harm} (g,h) = \bigoplus _{0 \le p,q\le n}{\rm Harm}^{p,q} (g,h)
\]
depends on these metrics.  

\subsection{The Hodge Theorem}
The Hodge Theorem can be stated as follows. 

\begin{thm}[Hodge Theorem for Vector Bundle-valued Forms]\label{hodge-vb}
Let $(X,g)$ be a compact complex Hermitian manifold and let $(E, h) \to X$ be a holomorphic Hermitian vector bundle.  Then the vector spaces ${\rm Harm}^{p,q} (g,h)$ are finite dimensional and consist of smooth $E$-valued $(p,q)$-forms.  Moreover, with
\[
\wp_{p,q} :L^2 _{p,q}(g,h) \to {\rm Harm}^{p,q}(g,h)
\]
denoting the orthogonal projection, there exists a compact self-adjoint collection of operators 
\[
G_{p,q} : L^2_{p,q}(g,h) \to L^2_{p,q}(g,h), \quad 0 \le p,q \le n,
\]
called the Green operator, such that 
\begin{enumerate}
\item[{\rm (a)}] $G _{p,q} W^{s,2}_{p,q}(X) \subset W^{s+2, 2}_{p,q}(X)$ for all $s \ge 0$, 
\item[{\rm (b)}] $G_{p,q} {\rm Harm} ^{p,q} (g,h) = 0$, $[G, \dbar ]= [G,\dbar ^*]= 0$ and $||G_{p,q} || \le \frac{1}{\lambda ^1_{p,q}}$, where $\lambda ^1_{p,q}$ is the smallest positive eigenvalue of $\Box$ in $L^2 _{p,q}(g,h)$, and 
\item[{\rm (c)}] one has the orthogonal decomposition ${\rm Id} = \wp_{p,q}+\Box  G_{p,q}$ on $L^2_{p,q}(g,h)$.
\end{enumerate}
\end{thm}

From the Chern connection $\nabla = \nabla ^{1,0} + \dbar$ one defines the $\nabla^{1,0}$-Laplacian 
\[
\Box ^{1,0} := \nabla ^{1,0}\nabla ^{1,0*} + \nabla ^{1,0*}\nabla ^{1,0}
\]
and its associated Green operator $G^{1,0}$ (whose action on $(p,q)$-forms will be unfortunately denoted $G^{1,0} _{p,q}$; we shall eliminate the subscript whenever there is no confusion).  An analogous Hodge Theorem holds for $\Box ^{1,0}$.

\subsection{K\"ahler manifolds}

Suppose $X$ is a K\"ahler manifold and $\omega$ is a K\"ahler form on $X$.  Let $E \to X$ be a holomorphic vector bundle.  An $E$-valued $(p,q)$-form may also be seen as a section of the vector bundle $\Lambda ^{p,q} _X \tensor E \to X$.  The latter vector bundle is not holomorphic unless $q=0$, but there is a smooth isomorphism 
\[
\Lambda ^{p,q} _X \tensor E \stackrel{\Phi _{\omega}}{\to} \Lambda ^{p,0} _X \tensor \Lambda ^q (T^{1,0} _X) \tensor E
\]
obtained from the identification of $(1,0)$-vectors and $(0,1)$-forms via the K\"ahler metric, given by $T^{1,0} _X \ni v \mapsto v \lrcorner \omega \in T^{*0,1}_X$.  Then one defines 
\[
\bar \fd := \Phi _{\omega} ^{-1} \dbar \Phi_{\omega} : \Gamma (X, \sC^{\infty}(\Lambda ^{p,q} _X \tensor E)) \to \Gamma (X, \sC^{\infty}(\Lambda ^{0,1} _X \tensor \Lambda ^{p,q} _X \tensor E)) 
\]
and with the induced $L^2$ inner product one defines the formal adjoint $\bar \fd^*$, and consequently the Bochner-Kodaira Laplacian $\bar \fd^* \bar \fd$ (also called the \emph{rough Laplacian} in differential geometry circles).  So far nothing about the K\"ahler nature of $X$ has been used, but the next identity relies heavily on the K\"ahler assumption.

\begin{thm}[Bochner-Kodaira Identity]\label{bk-id}
Let $X$ be a K\"ahler manifold with K\"ahler metric $g$ and let $E \to X$ be a holomorphic vector bundle with Hermitian metric $h$.  Then one has the identity 
\[
\Box = \bar \fd^* \bar \fd+ \Theta (h \tensor g^{(p)}\tensor \det g)^{\sharp} \lrcorner ,
\]
where $g^{(p)}$ and $\det (g)$ are the metrics induced on $\Lambda ^{p,0} _X$ and $K_X ^*$ respectively, and for a vector-bundle-valued $(1,1)$-form $A \in \Gamma (X, \Lambda ^{1,1}_X \tensor {\rm Hom}(F,F))$ we denote by $A^{\sharp}$ the $T^{1,0} _X \tensor {\rm Hom}(F,F)$-valued $(1,0)$-form obtained by raising with respect to the K\"ahler metric.
\end{thm}

One can also carry out the same procedure for the $(1,0)$-Laplacian 
\[
\Box ^{1,0} := \nabla ^{1,0} \nabla ^{1,0*} + \nabla ^{1,0*}\nabla ^{1,0},
\]
and define the operator $\fd = \Phi _{\omega} ^{-1} \nabla ^{1,0} \Phi _{\omega}$.  The corresponding identity is 
\[
\Box^{1,0}  = \fd^* \fd+ \Theta (h \tensor g^{(p)}\tensor \det g)^{\bar \sharp} \lrcorner ,
\]
where $\bar \sharp$ corresponds to the raising of a $(1,0)$-form to a $(0,1)$-vector field.  One can then show \cite{siu-jdg} that 
\begin{equation}\label{bochner-kodaira-formula}
\Box - \Box ^{1,0} = [\ii \Theta (h), \Lambda _{\omega}],
\end{equation}
where $\Lambda _{\omega}$ is the $L^2$-adjoint of the so-called \emph{Lefschetz operator}
\[
L_{\omega}: \alpha \mapsto \omega \wedge \alpha.
\]
Alternatively, one can arrive at \eqref{bochner-kodaira-formula} by using the (twisted) \emph{Hodge identities} 
\[
[\dbar ^*, L_{\omega}] = \ii \nabla ^{1,0}\quad \text{and} \quad[\nabla ^{1,0*}, L_{\omega}] = -\ii \dbar ,
\]
and their adjoints 
\[
[\Lambda _{\omega} , \dbar ] = - \ii \nabla ^{1,0*} \quad \text{and} \quad [\Lambda_{\omega}, \nabla ^{1,0}] = \ii \dbar ^*.
\]
\begin{s-rmk}
If $\Theta (h) = 0$ or $(p,q)= (n,0)$ then by \eqref{bochner-kodaira-formula} one has $\Box = \Box ^{1,0}$ and $G_{p,q} = G^{1,0} _{p,q}$.
\red
\end{s-rmk}

\subsection{Pointwise Lefschetz Theory}\label{lefschetz-paragraph}

In the proofs of Corollary \ref{xu-wang-cor} and Theorem \ref{berndtsson-reps-thm} primitive twisted forms play a central role.  In the present paragraph we recall some basic facts about these forms, and the general Lefschetz Theory.

\subsubsection*{\sf Classical, untwisted case}
The Lefschetz operator has a central role in the study of  K\"ahler manifolds.  We summarize a few of its salient features here, but the reader will find a more complete discussion in many texts; a few excellent sources are \cite{gh}, \cite{huybrechts} and \cite{voisin}.   

The starting point is the commutation relation 
\begin{equation}\label{l-lambda-comm-rel}
[L_{\omega}, \Lambda _{\omega}] u = (p+q-n) u, \quad \text u \in \Gamma (X, \Lambda ^{p,q}_X).
\end{equation}
Thus $[[L_{\omega}, \Lambda _{\omega}],L_{\omega}] = 2 L_{\omega}$ and $[[L_{\omega}, \Lambda _{\omega}],\Lambda _{\omega}] = - 2 \Lambda _{\omega}$.  Consequently for any point $x \in X$ the operators $X:= L_{\omega}$, $Y:= \Lambda _{\omega}$ and $H:= [L_{\omega}, \Lambda _{\omega}]$ form a representation of $\fs \fl _2 (\bbC)$ on $\Lambda ^{p,q} _{X,x}$.  The theory of finite-dimensional representations of $\fs \fl _2 (\bbC)$ is relatively simple, and a central result is that in every finite-dimensional irreducible representation the Kernel of $Y$ is 1-dimensional and the representation is spanned by the images of this kernel under $X^k$, $k=0,1,...$.  A non-zero vector in the kernel of $Y$ is called a \emph{lowest weight vector}, or a \emph{primitive} vector.  One shows that if $u \in \Gamma (X, \Lambda ^{p,q}_X)$ and $\Lambda _{\omega}u = 0$ then $L_{\omega} ^{n+1-(p+q)}u = 0$.  A study of the orbits of $L_{\omega}$ in $\Lambda ^{p,q} _X$ yields the so-called Lefschetz decomposition of forms:

\begin{prop}[Lefschetz Decomposition]\label{lefschetz-decomp}
Let $(X,\omega)$ be a compact K\"ahler manifold and let $k \in \{0,...,2n\}$.  For every $k$-form $\alpha$ on $X$ there exist unique primitive $(k-2j)$-forms $\alpha _j$, $\max (0,n-k) \le j \le k$, such that 
\[
\alpha = \sum _j L_{\omega }^j \alpha _j.
\]
Moreover, the summands are pointwise pairwise orthogonal.
\end{prop}

The metric induced on $(k,0)$-forms by the K\"ahler form has a particularly simple form: one has the identity
\[
\left < \alpha ,\beta \right > \frac{\omega ^n}{n!}= \ii ^{k^2}\alpha \wedge \bar \beta \wedge \frac{\omega^{n-k}}{(n-k)!}.
\]
This identity does not persist for $k$ forms, $k \le n$, of more general bi-degree, but remarkably it \emph{almost} holds for primitive forms: the right hand side is either equal to the left hand side or to its negative.  The result is known as the Hodge-Riemann Bilinear Relations.  

\begin{thm}[Hodge-Riemann Bilinear Relations]\label{hrbr-flat}
Let $(X,\omega)$ be a K\"ahler manifold of complex dimension $n$.  If $\alpha$ is a $(p,q)$-form with $p+q \le n$ and $\Lambda_{\omega} \alpha = 0$ then 
\[
\ii ^{(p+q)^2} \alpha \wedge \bar \alpha \wedge \frac{\omega ^{n-(p+q)}}{(n-(p+q))!} = (-1) ^q\left < \alpha, \alpha\right > \frac{\omega ^n}{n!}.
\]
\end{thm}

\begin{s-rmk}
Of greatest interest to us is the case $(p,q)= (n-1,1)$, in which the pairing
\[
\ii^{n^2} \alpha \wedge \bar \beta = - \left <\alpha , \bar \beta \right > \frac{\omega ^n}{n!}
\]
is negative definite.
\red
\end{s-rmk}

The untwisted Hodge Identity $[\dbar^*, L_{\omega}] = \ii \di$ shows that 
\[
L_{\omega}\Box  = \dbar (\dbar ^* L_{\omega} + \ii \di ) +(\dbar ^* L_{\omega} + \ii \di ) \dbar = \Box L_{\omega} + \ii (\dbar \di + \di \dbar ) = \Box L_{\omega}.
\]
Thus on a K\"ahler manifold the Lefschetz operator induces a well-defined map on Dolbeault cohomology.  By focusing on the primitive harmonic forms one can extend the pointwise Lefschetz theory just discussed to the setting of cohomology via the Fundamental Theorem of Hodge Theory:  one says that a Dolbeault class $\phi$ is \emph{primitive} if its unique harmonic representative is a primitive form.

Though it seems to be less commonly done in the algebraic geometry literature, one may also describe primitive Dolbeault cohomology classes in more Dolbeault-cohomological terms, i.e., without reference to harmonic forms.  In fact, since $\omega$ commutes with $\dbar$, a Dolbeault $(p,q)$-class $\phi$ is primitive if and only if for any $\alpha \in \phi$ the form $\omega ^{n+1- (p+q)} \wedge \alpha$ is $\dbar$-exact.  It is this equivalent way of defining primitive cohomology that allows the definition to extend to the twisted setting.  Indeed, in the untwisted setting a primitive harmonic form always defines a primitive cohomology class:  if $k \le n$ and $\alpha$ is a primitive harmonic $k$-form then $\omega ^{n+1-k} \wedge (\alpha + \dbar \beta) = \dbar (\omega ^{n-k+1} \wedge \beta)$ is $\dbar$-exact.  However, in the twisted setting the harmonic representative of a primitive Dolbeault class is not necessarily primitive.  (See Remark \ref{primitive-harmonic-rmk}.)

\subsubsection*{\sf Twisted case}
Now consider the twisted setting of forms with values in a holomorphic Hermitian vector bundle $E \to X$ with Hermitian metric $h$.  Because twisted forms are sections of $\Lambda _X ^* \tensor E \to X$ and the Lefschetz operator acts only on the first factor, one can define primitive forms in exactly the same way as in the flat, or untwisted, case.  

\begin{defn}\label{primitive-defn}
Let $(X,\omega)$ be a K\"ahler manifold of complex dimension $n$, let $E \to X$ be a holomorphic vector bundle, and fix integers $p,q \in \bbN$ with $p+q \le n$.
\begin{enumerate}

\item A Dolbeault class $C \in H^{p,q} _{\dbar} (X, E)$ is \emph{primitive} if for every $\alpha \in C$ the $E$-valued $(n+1-q,n+1-p)$-form $\alpha \wedge \omega ^{n+1- (p+q)}$ is $\dbar$-exact.  

\item An $E$-valued form $\alpha$ of degree $k\le n$ is \emph{primitive} if $\alpha \wedge \omega^{n+1-k} = 0$, or equivalently if $\Lambda _g \alpha = 0$.
\end{enumerate}
\end{defn}

\begin{rmk}\label{primitive-harmonic-rmk}
In the untwisted case, if $C$ is a primitive Dolbeault class then its harmonic representative is necessarily primitive.  Indeed, if $\alpha _o \in C$ is harmonic then $\alpha _o \wedge \omega ^{n+1-k}$ is $\dbar$-exact and harmonic, and must therefore vanish.  Thus an untwisted Dolbeault class is primitive if and only if it contains a primitive form.

In the twisted setting the harmonic representative of a primitive Dolbeault class $C$ need no longer be primitive.  The argument above breaks down because
\[
L_{\omega} \Box = \Box L_{\omega} + \Theta (h),
\]
so the Lefschetz operator no longer preserves harmonicity.  Nevertheless there does exist at least one primitive form in $C$.  The argument is not difficult, but for maximum simplicity (and since it is the only case we use) we present it in the case where $C$ is of bidegree $(n-1,1)$.  If $u_o \in C$ then one has $\omega \wedge u_o = \dbar \phi$.  Since $\phi$ has bidegree $(n,1)$, there exists a twisted $(n-1,0)$-form $\psi$ such that $\omega \wedge \psi = \phi$.  Indeed, one can take $\psi = -\Lambda _{\omega} \phi$ and use \eqref{l-lambda-comm-rel}.  Now simply take $u = u_o - \dbar \psi$.  Then $\omega \wedge u = \dbar \phi - \omega \wedge \dbar \psi  = 0$ because $L_{\omega}$ commutes with $\dbar$.

For forms of degree $k<n$ the argument is similar, but in order to solve the equation $\omega ^{n+1-k} \wedge u_o = \dbar \phi$ one must  use the commutation relation \eqref{l-lambda-comm-rel} in a more clever way.
\red
\end{rmk}

The Lefschetz Decomposition Theorem \ref{lefschetz-decomp} immediately extends to twisted forms.  

\begin{prop}[Twisted Lefschetz Decomposition]\label{lefschetz-decomp-twisted}
Let $(X,\omega)$ be a K\"ahler manifold and let $E \to X$ be a holomorphic vector bundle.  Let $k \in \{0,...,2n\}$.  For every $E$-valued $k$-form $\alpha$ on $X$ there exist unique primitive $E$-valued $(k-2j)$-forms $\alpha _j$, $\max (0,n-k) \le j \le k$, such that 
\[
\alpha = \sum _j \omega ^j \wedge \alpha _j.
\]
Moreover, the summands are pointwise pairwise orthogonal.
\end{prop}

There is also a version of the pointwise Hodge-Riemann Bilinear Relations.

\begin{prop}\label{Hodge-Riemann}
Let $E \to X, h$ be a Hermitian holomorphic vector bundle on a compact K\"ahler manifold $(X,\omega)$.  If an $E$-valued $(p,q)$-form $\alpha$ is primitive then
\[
\frac{\ii ^{(p+q)^2}}{(n-(p+q))!} \left < \alpha \wedge \bar \alpha \wedge \omega ^{n-(p+q)}, h\right > = \ii ^{p-q} |\alpha |^2 _{g,h}dV_{\omega}.
\]
\end{prop}

\begin{proof}[Partial proof]
We prove the result in the simplest special case $(p,q)= (n-1,1)$.  The simplest proof we know uses K\"ahler coordinates vanishing at a point $x \in X$:  in these coordinates 
\[
\omega (x) = \ii \di \dbar |z|^2 \quad \text{and} \quad \alpha (x)= \alpha _{i\bar j}(x) dz^1\wedge \cdots \wedge \widehat{dz ^i} \wedge ... \wedge dz ^n \wedge d\bar z ^j
\]
for some $\alpha _{i \bar j}(x) \in E_x$, $1 \le i,j \le n$.  Then $\alpha$ is primitive if and only if $\alpha _{i \bar j}(x) = 0$ for all $1 \le i \neq j  \le n$.  Thus 
\begin{eqnarray*}
 \left < \alpha \wedge \bar \alpha, h\right > &=& h(\alpha _{i \bar i}(x) , \alpha _{j \bar j}(x)) dz^1\wedge \cdots \wedge \widehat{dz ^i} \wedge ... \wedge dz ^n \wedge d\bar z ^i \wedge d\bar z^1\wedge \cdots \wedge \widehat{d\bar z ^j} \wedge ... \wedge d\bar z ^n \wedge d z ^j\\
 &=& \delta ^{i \bar j}  h(\alpha _{i \bar i}(x) , \alpha _{j \bar j}(x)) \frac{dV(z)}{\ii ^n (-1)^{n(n+1)/2}}= - \delta ^{i \bar j}h(\alpha _{i \bar i}(x) , \alpha _{j \bar j}(x))  \frac{dV(z)}{\ii ^{n^2}},
\end{eqnarray*}
i.e., $ \ii ^{n^2} \left < \alpha \wedge \bar \alpha , h\right > = - |\alpha |^2 _{g,h} dV_{\omega}$, as claimed.
\end{proof}

\begin{s-rmk}
The general case is proved using Weil's Formula for the Hodge star operator.  See \cite[Proposition 6.29]{voisin}.  
\red
\end{s-rmk}

\begin{s-rmk}
In the twisted setting the pointwise Hodge-Riemann pairing $Q(\alpha ,\beta)$  is an $E \tensor E$-valued $(n,n)$-form, so the integrated form of the Hodge-Riemann pairing $\cQ$ no longer makes sense.  To the best of our knowledge there is no obvious twisted analogue of the Hodge Index Theorem.
\red
\end{s-rmk}

\subsection{Solutions of the $\dbar$ Equation with minimal $L^2$-norm}
The Hodge Theorem yields solutions of the $\dbar$-equation with estimates.  The statement is the following.  

\begin{prop}\label{hormander-hodge}
Let $X$ be a compact complex manifold equipped with a smooth Hermitian metric $g$, let $E \to X$ be a holomorphic vector bundle equipped with a smooth Hermitian metric $h$, and let $q \ge 1$.  For every $\dbar$-closed $E$-valued $(p,q)$-form $\alpha$ such that $\wp_{p,q} \alpha = 0$ there exists a unique $E$-valued $(p,q-1)$-form $u_o$ such that 
\begin{equation}\label{norm-of-min}
\dbar u_o = \alpha \quad \text{and} \quad ||u_o||^2 = \sup _{\gamma \in \Gamma (X, \sC^{\infty} (\Lambda ^{p,q}_X \tensor E))} \frac{\left | \left ( \alpha , \gamma \right )\right |^2}{||\dbar \gamma ||^2 + ||\dbar ^* \gamma||^2}.
\end{equation}
The $L^2$-norm of any other solution $u$ of the equation $\dbar u = \alpha$ is larger.
\end{prop}

\begin{proof}
Let $H_o$ be the collection of smooth $E$-valued $(p,q)$-forms that are orthogonal to ${\rm Kernel}(\Box)$ and let $H$ be the Hilbert space closure of $H_o$ with respect to the norm 
\[
||\gamma ||^2_H := ||\dbar \gamma||^2 + ||\dbar ^* \gamma ||^2.
\]
The latter is a norm because by the Hodge Theorem there exists a constant $c> 0$ such that $(\Box \gamma , \gamma )  \ge c ||\gamma||^2$ for all $\gamma \in H$.  Let $\lambda (\gamma) := (\alpha , \gamma)$.  Then $\lambda \in H^*$ and 
\[
||\lambda ||^2_{H^*} = C:= \sup _{\gamma \in H_o} \frac{\left | \left ( \alpha , \gamma \right )\right |^2}{||\dbar \gamma ||^2 + ||\dbar ^* \gamma||^2}.
\]
By the Riesz Representation Theorem there exists $ \beta \in H$ such that $|| \beta ||_H = ||\lambda||_{H^*}$ and $\lambda (\gamma) = (\beta, \gamma) _H$.  The latter means precisely that $\Box  \beta = \alpha$ (in the weak sense).  Since $\dbar \alpha = 0$, 
\[
0 = (\dbar \Box \beta,\dbar \beta)  = (\dbar \dbar ^* \dbar \beta, \dbar \beta) = ||\dbar^* \dbar \beta||^2 \quad \text{and} \quad 0 = (\dbar ^* \dbar \beta ,\beta) = ||\dbar \beta ||^2,
\]
so $\Box \beta = \dbar \dbar ^*\beta$.  Therefore $u_o := \dbar ^*  \beta$ solves the equation $\dbar  u_o = \alpha$, and $||u_o||^2 = || \beta||^2_H = C$.  Finally, since $u_o = \dbar ^* \beta$, $u_o$ is orthogonal to the kernel of $\dbar$, and hence is the unique solution of minimal norm.  The proof is complete.
\end{proof}

\begin{cor}\label{hormander-estimates}
Let $X$, $E$, $g$, $h$ and $\alpha$ be as in Proposition \ref{hormander-hodge}. 
\begin{enumerate}
\item[{\rm i.}]  There exists a solution $u$ of the equation $\dbar u = \alpha$ such that 
\[
||u||^2 \le (\lambda ^1_{p,q})^{-1} ||\alpha||^2,
\]
where $\lambda ^1 _{p,q}$ is the smallest positive eigenvalue of $\Box : L^2 _{p,q}(g,h) \to L^2 _{p,q}(g,h)$.   
\item[{\rm ii.}]  Moreover, if the metric $g$ is K\"ahler and the endomorphism 
\[
A := \Theta (h \tensor g^{(p)} \tensor \det (g))^{\sharp} \lrcorner
\]
\emph{(c.f. Theorem \ref{bk-id})} is positive semi-definite, $\alpha \perp {\rm Kernel}(A)$ almost everywhere on the support of $\alpha$, and $( A^{-1} \alpha, \alpha ) < +\infty$ then there exists an $E$-valued $(p,q-1)$-form $u$ satisfying the Demailly-Skoda-H\"ormander Estimate 
\[
||u||^2 \le (A^{-1} \alpha, \alpha )
\]
such that $\dbar u = \alpha$.  Consequently 
\[
(G\alpha, \alpha) \le (A^{-1} \alpha , \alpha).
\]
\end{enumerate}
\end{cor}

\begin{proof}
By Proposition \ref{hormander-hodge} the Hodge Theorem and the minimal solution is $u _o = G\dbar ^* \alpha$.  Then 
\[
||u_o||^2 = (G \dbar ^* \alpha , G \dbar ^* \alpha) = (G \alpha , \dbar \dbar ^* G\alpha) = (G\alpha, \alpha).
\]
Thus i follows because $G = (\Box |_{{\rm Harm}^{\perp}})^{-1}$, and ii follows from the existence of a solution $u$ with the estimate $||u||^2 \le (A^{-1} \alpha, \alpha )$; a consequence of Skoda's and Demailly's modifications of H\"ormander's $L^2$ estimate.
\end{proof}

%\begin{rmk}
%Of course, the unique harmonic $E$-valued $(n-1,1)$-form $\mu$ in the Dolbeault class of $(\iota _{X_t}^* \dbar \tilde \xi _{\tau}) \lrcorner \iota _{X_t} ^* \alpha$ is also primitive (c.f. Definition \ref{primitive-defn} and Remark \ref{primitive-harmonic-rmk}).  However, it is unclear to the author whether $\mu = (\iota _{X_t}^*\dbar \hat \xi _{\tau}) \lrcorner f_p$ for some lift $\hat \xi _{\tau}$ of $\tau$.  It is clear that there is an $E$-valued $(n-1,0)$-form $\alpha$ such that 
%\[
%\mu = (\iota _{X_t}^* \dbar \tilde \xi _{\tau}) \lrcorner f_p - \dbar \alpha,
%\]
%and evidently $\dbar \alpha \wedge \omega = 0$.   The question is whether the harmonicity of $\mu$ somehow forces $\alpha$ to be of the form $\alpha = \theta \lrcorner f_p$ for some $(0,1)$-vector field $\theta$ on $X_t$.  If the latter is the case then the harmonic form $\mu$ must vanish along the zero locus of $f$.  It does not seem clear whether $\mu$ vanishes along the zero locus of $f$.  
%\red
%\end{rmk}

\subsection{The Bergman projection and the Neumann operator}\label{bk-paragraph}

For a general family $(E,h) \to X \stackrel{p}{\to} B$, i.e., not necessarily proper, with $h$ smooth (also not necessary, but some regularity may be required) the Hilbert spaces $\sH _t$ are closed Hilbert subspaces of $\sL_t$.  (In our proper setting of a proper holomorphic family the spaces $\sH _t = {\rm Harm}_{n,0}(g_t, h_t)$ are finite dimensional, so closedness is automatic.)  The corresponding bounded orthogonal projection 
\[
P_t : \sL_t \to \sH _t
\]
is called the \emph{Bergman projection}.   

From a section $f \in \Gamma (X,K_{X/B} \tensor E)$ that is smooth (or even just $L^2$) on the fibers of $p$ one has a section $\cP f \in \Gamma (X, K_{X/B}\tensor E)$ defined by 
\[
\iota _t \cP f := P_t \iota _t f
\]
where $\iota _t: \Gamma (X_t, (K_{X/B}\tensor E)|_{X_t}) \to \Gamma (X_t, K_{X_t}\tensor E|_{X_t} )$ is the isomorphism defined by Proposition \ref{relative-restriction-prop}.  Thus $P_{\sH} :\sL \to \sH$ is given by 
\[
P_{\sH} \ff  = \fri \left ( \cP \fa(\ff) \right ).
\]

The orthogonal projection $P_t^{\perp} :\sL _t \to \sH _t^{\perp}$ onto the orthogonal complement of $\sH_t$ in $\sL_t$ is called the \emph{Neumann operator}.  If one is given a $K_{X_t}\tensor E|_{X_t}$-valued, $\dbar$-exact $(0,1)$-form $\alpha$ then any two solutions $u_o, u_1 \in \sL _t$ of the equation $\dbar u = \alpha$ differ by an element of $\sH _t$.  Thus if $u \in \sL _t$ satisfies $\dbar u = \alpha$ then so does $P_t ^{\perp} u$, and the latter has minimal $\sL_t$-norm.

On a compact complex manifold the Hodge Theorem computes $P_t$ and $P_t ^{\perp}$.  Indeed, by the Hodge Theorem every $\dbar_{X_t}$-exact form is orthogonal to the harmonic forms, and, as we saw in the proof of Theorem \ref{FTHTvb}, if $\alpha = \Box_{X_t} u$ is $\dbar_{X_t}$-closed then $\alpha = \dbar_{X_t} \dbar_{X_t} ^* u$.  Thus $\wp (\dbar_{X_t} f) = 0$ (i.e., $\dbar _{X_t} f$ has no harmonic part) and $\dbar_{X_t} f = \dbar_{X_t} \dbar_{X_t}^* G_t \dbar_{X_t} f$.  Since $f - \dbar_{X_t}^* G_t \dbar_{X_t} f$ is holomorphic and $(\dbar_{X_t}^* G_t \dbar_{X_t} f, f_o) = (G_t \dbar_{X_t} f, \dbar_{X_t} f_o ) = 0$ for any $f_o \in \sH _t$, we see that  
\begin{equation}\label{bk-gf-formula}
P_t f = f - \dbar_{X_t}^* G_t \dbar_{X_t} f \quad \text{ and thus } \quad P_t ^{\perp} f = \dbar_{X_t} ^* G_t \dbar_{X_t} f.
\end{equation}

\begin{rmk}\label{holo-subfield-is-regular}
It is natural to ask whether $\cP$ (and hence $\cP^{\perp} = {\rm Id} - \cP$) is a smooth operator.  In view of \eqref{bk-gf-formula} it suffices to decide when the Green operator $G_t$ varies smoothly with $t$.   The latter question was answered by Kodaira and Spencer:  if the rank of the harmonic spaces is constant then $G_t$ varies smoothly with $t$ (and hence $\cP^{\perp}$ and $\cP$ are smooth operators).  See \cite[Theorem 7.6, page 344]{kodaira-def}.  And though it is not stated explicitly in the latter reference, the converse is also true:  if the Green operator $G_t$ varies smoothly with $t$ then the dimensions of the harmonic spaces are locally constant.  

Thus, in the sense of Definition \ref{smooth-subfield-defn}.c, $\sH \subset \sL$ is regular if and only if $\sH \to B$ is a vector bundle.
\red
\end{rmk}

\section{The Kodaira-Spencer Map of a Proper Holomorphic Family}\label{KS-section}

The Kodaira-Spencer map of a germ at $o\in B$ of a proper holomorphic family  $X \to B$ assigns to a tangent vector $\tau \in T^{1,0} _{B,o}$ a Dolbeault cohomology class $\kappa ^p _o (\tau) \in H^{0,1}_{\dbar} (X_o, T^{1,0} _{X_o})$ on the central fiber of that germ.  Any representative of the map acts on possibly twisted $(p,q)$-forms to produce $(p-1,q+1)$-forms.  The square norm of this action appears in the formulas of both the curvature of $\sL ^{\theta}$ and the second fundamental form of $\sH$ in $\sL ^{\theta}$.  We begin by reviewing the definition of the Kodaira-Spencer map, and then discuss how it acts on sections of the BLS field $\sH \to B$. 

\subsection{Definition of the Kodaira-Spencer Map}

Let $p :X \to B$ be a holomorphic family over a domain $B \in \bbC ^m$ containing the origin. Let us recall the definition of the Kodaira-Spencer map
\[
K ^p_o : T^{1,0} _{B,o} \to H^{0,1} _{\dbar} (X_o, T^{1,0} _{X_o}).
\]
There are some aspects of the construction that can be confusing, so let us be temporarily overly pedantic, and denote by $\dbar _E$ the $\dbar$-operator of the holomorphic vector bundle $E$.  Fix a $(1,0)$ vector $\tau \in T^{1,0} _{B,o}$, extended to a holomorphic $(1,0)$-vector field on a neighborhood of $o$ in $B$ with the same notation $\tau$.  Choose a lift $\xi _{\tau}$ of $\tau$, i.e., a $(1,0)$-vector field on $X$ such that $dp (\xi _{\tau}) = \tau$.  Since $p$ is holomorphic, 
\[
dp \left ( \dbar _{T^{1,0}_X} \xi _{\tau} \right ) = \dbar _{T^{1,0} _B} dp (\xi _{\tau}) =  \dbar _{T^{1,0} _B} {\tau} = 0,
\]
and therefore $\dbar _{T^{1,0}_X} \xi _{\tau}$ is a $(0,1)$-form with values in the vertical vector bundle $T^{1,0} _{X/B}$.  Now, $\dbar _{T^{1,0}_X} \xi _{\tau}$ is $\dbar_{T^{1,0}_X}$-exact, but in general it is not $\dbar _{T^{1,0}_{X/B}}$-exact.  However, $\dbar _{T^{1,0}_X} \xi _{\tau}$ is $\dbar _{T^{1,0}_{X/B}}$-closed.  This closedness is seen in local coordinates as follows.  We take a chart with coordinates $s = (s^1,..., s^m)$ on a neighborhood of $o$ in $B$, and in terms of this coordinate chart we write $p = (p^1,..., p^m)$.  By the inverse function theorem we can choose local coordinates $(z, t)$ in $\bbC^n\times \bbC^m$ with $t^{\mu} = p^{\mu} = p^* s^{\mu}$, $\mu =1,..., m$.  In terms of these coordinates we write $\tau =  \tau ^i (s) \frac{\di}{\di s^i}$ and $\xi _{\tau} = F^j (z,t) \frac{\di }{\di z ^j} + \tau ^i(t) \frac{\di}{\di t^i}$.  Since the functions $\tau ^i$ are holomorphic,
\[
\dbar _{T^{1,0}_{X}}\xi _{\tau} =(\dbar _{z,t} F^j) \tensor \frac{\di}{\di z^j} =  \left ( \frac{\di F^j}{d\bar z ^k} d \bar z ^k + \frac{\di F^j }{\di \bar t ^{\ell}}d\bar t ^{\ell}  \right ) \tensor  \frac{\di }{\di z ^j}.
\]
From the formula 
\[
\dbar _{T^{1,0} _{X/B}} H^j \frac{\di}{\di z^j} = (\dbar _{z,t} H^j) \tensor \frac{\di}{\di z^j}
\]
we see that 
\[
\dbar_{T^{1,0} _{X/B}} \dbar _{T^{1,0} _{X}}\xi _{\tau} = 0.
\]
To simplify notation, from here on we shall simply write $\dbar \xi _{\tau}$ for the $\dbar$-closed $T^{1,0} _{X/B}$-valued $(0,1)$-form obtained by applying $\dbar_{T^{1,0}_X}$ to the lift $\xi_{\tau}$ of the holomorphic $(1,0)$-vector field $\tau$.

We claim that the Dolbeault class of $\iota _{X_o} ^* \dbar \xi _{\tau}$  in $H^{0,1} _{\dbar} (X_o, T^{1,0} _{X_o})$ depends only on $\tau$, and not on the lift $\xi _{\tau}$.  Indeed, suppose $\tilde \xi _{\tau}$ is another lift.  Then the global section $\eta$ of $T^{1,0} _{X/B} \to X$ defined by $\eta = \xi _{\tau} - \tilde \xi _{\tau}$ satisfies 
\[
\dbar \xi _{\tau} = \dbar \tilde \xi _{\tau} + \dbar \eta,
\]
and restriction to $X_o$ shows that $[\iota _{X_o}^* \dbar \xi _{\tau}]= [\iota _{X_o}^* \dbar \tilde \xi _{\tau}]$ in $H^{0,1} _{\dbar} (X_o, T^{1,0} _{X_o})$.

In fact, one can go further, and show that $[\iota _{X_o}^* \dbar \xi _{\tau}]$ depends only on $\tau (o)$.  Indeed, while $\xi _{f\tau} \not = p^*f \cdot \xi _{\tau}$ for all $f \in \cO (B)$, the difference $\xi _{f\tau} - p^*f \cdot \xi _{\tau}$ is a section of $T^{1,0} _{X/B}$.  Similarly $\xi _{\tau _1 + \tau _2} - \xi _{\tau _1} - \xi _{\tau _2}$  is a section of $T^{1,0} _{X/B}$.  These facts imply that the class $[\iota _{X_o}^* \dbar \xi _{\tau}]$ is independent of the choice of holomorphic extension of $\tau$, and also that the dependence on $\tau$ is linear.

\begin{s-rmk}
If we choose our lift $\xi _{\tau}$ so that it comes from a horizontal distribution then the linearity follows more easily, since in that case  $\xi _{f\tau + \sigma} = p^*f \cdot \xi _{\tau} + \xi _{\sigma}$. In fact, in this case the linearity holds at the level of the forms $\dbar \xi _{\tau}$, even before restrictions to the fibers of $p$.
\red
\end{s-rmk}

\begin{s-rmk}
The choice of a \emph{holomorphic} extension of the vector $\tau$ is made for convenience, and is not required.  Indeed, one can see from the local formula that for the lift $\xi _{\tau}$ for each $t$ the section $\iota _{X_t} ^* \dbar \xi _{\tau}$ is a $\dbar$-closed $T^{1,0} _{X_t}$-valued $(0,1)$-form.  Again the class $[\iota _{X_o} ^* \dbar \xi _{\tau}]$ depends only on $\tau (o)$, but this time for slightly different reasons: the difference of two lifts $\zeta := \xi _{\tau} - \tilde \xi _{\tilde \tau}$ is vertical at $o$, and one checks that 
\[
\iota _{X_o} ^* \dbar \xi _{\tau} - \iota _{X_o} ^* \dbar \xi _{\tau} = \dbar _{T^{1,0} _{X_o}}\left ( \zeta |_{X_o}\right )
\]
by using a local computation as above.
\red
\end{s-rmk}

\begin{defn}
The linear map 
\begin{equation}\label{kodaira-spencer-map-defn}
K ^{p}_o : T^{1,0} _{B,o} \ni \tau \mapsto [\iota _{X_o}^* \dbar \xi _{\tau}] \in H^{0,1}_{\dbar} (X_o, T^{1,0} _{X_o})
\end{equation}
is called the Kodaira-Spencer map of the holomorphic family $p :X \to B$ at $o \in B$.
\end{defn}

In particular, if we choose any horizontal distribution $\theta \subset T_X$ then the maps 
\[
\kappa ^{\theta} _{t} : T^{1,0} _{B, t}  \ni \tau \mapsto \iota _{X_t} ^*\dbar \xi ^{\theta}_{\tau} \in \Gamma \left (X_t, \sC^{\infty} \left (K_{X_t}\right )\right ),
\]
(see page \pageref{horiz-lift-notation-page} for the notation $\xi ^{\theta}_{\tau}$) represent the maps 
\[
K ^p _t: T^{1,0} _{B,t} \to H^{0,1} _{\dbar} (X_t, T^{1,0} _{X_t})
\]
in the sense of Dolbeault cohomology.  In the proof of Theorem \ref{B2-thm} these representatives will make another appearance, and it will be particularly important to choose the right horizontal lift.  The next paragraph is concerned with the construction of this special horizontal lift.  

\subsection{Action on the fibers of $\sH$}

So far in this section we have been dealing with general holomorphic families $p :X \to B$.  In the present paragraph we shall assume that the family $p:X \to B$ is K\"ahler, i.e., is equipped with a \emph{relative K\"ahler form}: a closed $(1,1)$-form $\omega$ on $X$ such that for each $t \in B$ the restriction $\iota _{X_t} ^*\omega$ is a K\"ahler form on $X_t$. 

The contraction of $E$-valued holomorphic $(n,0)$-forms on the fiber $X_t$ of $p$ with any representative of the Kodaira-Spencer class $K^p _t (\tau)$ produces a $\dbar$-closed $E$-valued $(n-1,1)$-form on $X_t$ whose Dolbeault class is independent of the choice of representative of $K^p _t(\tau)$.  Moreover, one has the following remarkable proposition.

\begin{prop}\label{dolbeault-prim-prop}
For a point $t \in B$ and any holomorphic $E$-valued $(n,0)$-form $f$ on $X_t$ the Dolbeault cohomology class $K^p _t (\tau) \lrcorner f \in H^{n-1,1}_{\dbar} (X_t, E)$ is primitive with respect to the K\"ahler form $\omega _t := \iota _{X_t} ^* \omega$.
\end{prop}

Actually, one can do even better:  one can choose a horizontal lift $\theta$ such that 
\begin{equation}\label{fib-perp-lift-eq}
\xi ^{\theta}_{\tau} \lrcorner \omega = 0 \quad \text{for all $\tau \in T^{1,0}_B$}.
\end{equation}
Using this lift, for any $E$-valued $(n,0)$-form $u$ on the fiber $X_t$ one has 
\begin{eqnarray*}
0 &=& \dbar \xi ^{\theta}_{\tau} \lrcorner (\omega _t \wedge u) =  (\dbar \xi ^{\theta}_{\tau}\lrcorner \omega_t ) \wedge u + \omega _t \wedge ( \dbar \xi ^{\theta}_{\tau}\lrcorner u)=  (\dbar (\xi ^{\theta}_{\tau}\lrcorner \omega)_t ) \wedge u + \omega _t \wedge ( \dbar \xi ^{\theta}_{\tau}\lrcorner u)\\
&= &\omega _t \wedge ( \dbar \xi ^{\theta}_{\tau}\lrcorner u),
\end{eqnarray*}
which shows the following proposition.

\begin{prop}\label{good-horizontal-distribution}
The form $\iota _{X_t} ^*(\dbar \xi ^{\theta}_{\tau}\lrcorner u)$ representing $K^p_t ( \tau) \lrcorner f$ is primitive.
\end{prop}

If the relative K\"ahler form $\omega$ is actually a K\"ahler form, i.e., in addition to being closed, it is also a positive $(1,1)$-form, then the horizontal lift defined in \eqref{fib-perp-lift-eq} is just the lift that is perpendicular to the fibers.  It is easy to see that this lift is well-defined even if $\omega$ is not positive.  Indeed, the component of the lift $\xi ^{\theta}_{\tau}$ that is transverse to the fibers is already determined simply because $\xi ^{\theta}_{\tau}$ is a lift, so to determine the vertical component we need only $\omega _t$ to be K\"ahler for each $t$, which is of course the case.  More explicitly, using the implicit function coordinates $(z,t)$ defined by $p (z,t) = t$, we can write 
\[
\xi ^{\theta}_{\tau} = F^i \frac{\di}{\di z ^i} + \tau ^{\mu} \frac{\di}{\di t ^{\mu}}.
\]
Then $\xi ^{\theta}_{\tau} \lrcorner \omega = \left (F^i \omega _{i \bar j} + \tau ^{\mu} \omega _{\mu \bar j}\right ) d\bar z ^j$, so taking 
\[
F^i := - \tau ^{\mu} \omega _{\mu \bar j} \omega ^{i \bar j}
\]
yields the desired conclusion.

\section{The proof of Theorem \ref{B2-thm}}\label{B2-thm-pf-section}

From this point on in the article, we shall make use of an abuse of notation that will make formulas much more compact.  As discussed in Section \ref{lie-deriv-section}, there are two equivalent ways to realize the sections of $\sL \to B$.  One of them is as sections of $K_{X/B} \tensor E \to X$ and the other is as equivalence classes of sections of $\Lambda ^{n,0} _X \tensor E \to B$.  When applying the twisted Lie derivatives, we shall always use the second description.  But we will abuse notation and write 
\[
L^{1,0} _{\xi _{\tau} ^{\theta}} \fa (\ff)
\]
to mean the section represented by the $E$-valued $(n,0)$-form $L^{1,0} _{\xi _{\tau} ^{\theta}} u$ for any $u$ representing the section $\fa (\ff)$.  Note that by Proposition \ref{int-ptwise-holo}.a the resulting sections are independent of the choice of representatives.  

\subsection{Formula for the Chern connections of $\sL^{\theta}$}

\begin{prop}\label{Chern-connections-for-L^theta}
Let $E \to X$ be a holomorphic vector bundle with Hermitian metric $h$.  The BLS-Chern connection $\nabla ^{\sL ^{\theta}}$ for $\sL^{\theta}$ is given by the formula 
\[
\nabla ^{\sL ^{\theta}1,0}_{ \tau} \ff = \fri \left ( L^{\nabla 1,0}_{\xi ^{\theta}_{\tau}} \fa (\ff)\right ), \quad \tau \in T^{1,0} _{B,t},
\]
where $\nabla$ is the Chern connection for $(E,h)$ (c.f. Definition \ref{cplx-lie-deriv}).
\end{prop}

\noi Proposition \ref{Chern-connections-for-L^theta} will be proved momentarily.

\medskip

There are many connections compatible, in the sense of Definition \ref{connection-defn}, with the smooth structure of $\sL \to B$, including connections compatible with the almost holomorphic structures defined in Paragraph \ref{structure-of-L}, i.e., connections $\nabla$ such that $\nabla ^{0,1} = \dbar ^{\theta}$.  We restrict ourselves to such $\dbar^{\theta}$-compatible connections, since it is among these that we shall find our Chern connections.   

For every $p$-horizontal distribution $\theta \subset T_X$ and every connection $D$ for $E \to X$ such that $D^{0,1} = \dbar$ there is a connection $\nabla^{D,\theta} : \Gamma (B, \sC^{\infty} (\sL)) \to \Gamma (B, \sC^{\infty} (\Lambda ^1_B \tensor \sL))$ defined as follows.  Given a section $\ff \in \Gamma (B, \sC^{\infty} (\sL))$, for each $\tau = \tau ^{1,0} + \tau ^{0,1} \in T_{B}\tensor \bbC = T^{1,0}_B \oplus T^{0,1} _B$ the connection 
\[
\nabla ^{D,\theta}_{\tau} \ff := \fri \left ( L^{D1,0}_{\xi ^{\theta}_{\tau^{1,0}}} \fa (\ff)\right ) + \dbar ^{\theta}_{\tau^{0,1}} \ff = \fri \left ( L^{D1,0}_{\xi ^{\theta}_{\tau^{1,0}}} \fa (\ff)+ L^{0,1}_{\bar \xi^{\theta}_{\tau^{0,1}}} \fa(\ff) \right )
\]
is compatible with the smooth structure of $\sL \to B$.

\begin{prop}\label{induced-connections-smooth}
Let $E \to X$ be a holomorphic vector bundle with Hermitian metric $h$.  If $D$ is a smooth connection for $E \to X$ and $\theta$ is a smooth horizontal lift then the induced connection $\nabla ^{D,\theta}$ is a connection for $\sL \to B$ in the sense of Definition \ref{connection-defn}.  Moreover, if $Dh = 0$ then $\nabla ^{D,\theta}\fh = 0$, where $\fh$ is the $L^2$ metric for $\sL \to B$.
\end{prop}

\begin{proof}
The smoothness assertion is evident, and Lemma \ref{lie-deriv-lemma} shows that if $D$ is the Chern connection for $E$ then $\nabla^{D,\theta}$ is the BLS-Chern connection.
\end{proof}

\noi Proposition \ref{induced-connections-smooth} immediately imply Proposition \ref{Chern-connections-for-L^theta}.
\qed

\medskip

\begin{rmk}[Chern connection for $\sH$ when $\sH$ is locally trivial]\label{H-connection-paragraph}

We have not defined the Chern connection for $\sH \to B$ when the latter is not locally trivial and we do not intend to do so.  Of course, when $\sH \to B$ \emph{is} locally trivial, it has a well-defined Chern connection.  Though we do not need the Chern connection for $\sH$ in our approach, we can now easily compute it from the Chern connection of the ambient BLS bundles $\sL ^{\theta} \to B$.  Indeed,  by Proposition \ref{subbundle-chern-connection-prop}, Proposition \ref{H-is-always-subbundle} and Remark \ref{holo-subfield-is-regular} one has the following proposition.

\begin{prop}
If $\sH \to B$ is locally trivial then the Chern connection $\nabla ^{\sH}$ of $\sH$ is smooth, and is given by $\nabla ^{\sH}_{\tau} \ff = \nabla ^{\sH1,0}_{\tau} \ff = \fri \left (\cP\left (L^{1,0} _{\xi _{\tau}} \fa (\ff)\right )\right )$
for any $\ff \in \Gamma (B,\sC^{\infty}(\sH))$ and $\tau \in T^{1,0}_{B,t}$, where  $\cP$ is the fiberwise Bergman projection operator defined in Paragraph \ref{bk-paragraph} and $\xi _{\tau}$ is any lift of $\tau$ to a $(1,0)$-vector field on $X_t$.  
\end{prop}

In \cite{bo-annals} the formula for $\nabla ^{\sH}$ was derived directly using the definition of Chern connection.  We note also that the formula for $\nabla ^{\sH}$ makes sense for smooth sections of $\sH$ even when $\sH$ is not locally trivial.
\red
\end{rmk}

\subsection{The $\mathbf{(1,1)}$-component of curvature of $\sL^{\theta}$}

By Proposition \ref{H-integrable=>L-11} and Remark \ref{H-is-holo-hh} the restriction to $\sH$ fo the curvature of $\sL^{\theta}$ has $(1,1)$-form coefficients.  Since our interest is mainly in $\Theta ^{\sH}$, we shall only compute the $(1,1)$-component of $\Theta ^{\sL ^{\theta}}$.

The $(1,1)$-component of  curvature of $\sL^{\theta}$ is by definition 
\begin{equation}\label{curvature-by-defn}
\Theta ^{\sL^{\theta}}_{\sigma \bar \tau}  = \nabla ^{\sL ^{\theta}}_{\sigma} \nabla ^{\sL ^{\theta}}_{\bar \tau} - \nabla ^{\sL ^{\theta}}_{\bar \tau} \nabla ^{\sL ^{\theta}}_{\sigma} - \nabla ^{\sL ^{\theta}}_{[\sigma , \bar \tau]}, \quad \sigma ,\tau \in T^{1,0}_{B,t},\ t \in B.
\end{equation}
In the latter formula, one may take any extension of $\sigma$ and $\tau$ to $(1,0)$-vector fields, and taking holomorphic extensions eliminates the last term on the right hand side of \eqref{curvature-by-defn}.  Therefore for smooth sections $\ff_1, \ff_2$ of $\sL^{\theta}$ we have $\fa \left ( \Theta ^{\sL^{\theta}}_{\sigma \bar \tau} \ff_1 \right ) = \left ( L^{1,0} _{\xi^{\theta}_{\sigma}} L^{0,1} _{\bar \xi^{\theta}_{\tau}} - L^{0,1} _{\bar \xi^{\theta}_{\tau}} L^{1,0}_{\xi^{\theta}_{\sigma}}\right ) \fa (\ff_1)$ and 
\begin{equation}\label{(1,1)-Lcurv}
\left (\Theta ^{\sL ^{\theta _i}}_{\sigma \bar \tau} \ff _1, \ff_2\right ) = \left ( L^{1,0} _{\xi ^{\theta ^i}_{\sigma}} L^{0,1} _{\bar \xi ^{\theta ^i}_{\tau}} u_1,u_2 \right ) - \left ( L^{0,1} _{\bar \xi ^{\theta ^i}_{\tau}} L^{1,0} _{\xi ^{\theta ^i}_{\sigma}} u_1,u_2 \right ),
\end{equation}
where $u_i$ are $E$-valued $(n,0)$-forms on $X$ representing $\ff _i$.
%Note that for holomorphic vector fields  $\sigma$ and $\tau$ the complex vector field $[\xi ^{\theta}_{\sigma} , \bar \xi ^{\theta}_{\tau}]$, though not zero, is a vertical vector field.
Using \eqref{(1,1)-Lcurv} we can finally give the proof of Theorem \ref{indep}.

\begin{proof}[Proof of Theorem \ref{indep}]
Fix $t \in B$ and smooth sections $\ff _1, \ff_2 \in \Sigma_t$, i.e., $\ff _1, \ff _2 \in \sC^{\infty}(\sL)_t$ such that $\ff_1(t), \ff_2(t) \in \sH_t$ and $\iota _{X_t} ^* L^{1,0} _{\xi ^{\theta} _{\sigma}} \dbar u_i= 0$, where $u _i \in H^0(X, \sC^{\infty}(\Lambda ^{n,0}_X \tensor E))$ represents $\ff _i$. 

%If $\tau$ is a $(1,0)$-vector field near $t$ and $v$ is a vertical $(1,0)$-vector field near $X_t$ then by Lemma \ref{lie-deriv-lemma} one has $\sigma \left ( L^{0,1} _{\bar v} u_1, u_2 \right )  = \left ( L^{1,0} _{\xi ^{\theta}_{\sigma}}  L^{0,1} _{\bar v} u_1, u_2  \right ) + \left ( L^{0,1} _{\bar v} u_1, L^{0,1} _{\bar \xi ^{\theta}_{\sigma}} u_2 \right )$.  Hence at the point $t$  
%\[
%\left . \sigma \left ( L^{0,1} _{\bar v} u_1, u_2 \right ) \right |_t= 0 \text{ and (by complex conjugation) }\left . \bar \tau \left ( u _1, L^{0,1} _{\bar v} u_2 \right ) \right |_t = 0.
%\]
Defining the vertical vector field $v _{\sigma} := \xi ^{\theta _1} _{\sigma} - \xi ^{\theta _2} _{\sigma}$, we find that 
\begin{eqnarray*}
&& \left (\Theta ^{\sL ^{\theta _1}}_{\sigma \bar \tau} \ff _1, \ff_2\right ) - \left (\Theta ^{\sL ^{\theta _2}}_{\sigma \bar \tau} \ff _1, \ff_2\right ) \\
&=& \left ( L^{1,0} _{\xi ^{\theta ^1}_{\sigma}} L^{0,1} _{\bar \xi ^{\theta ^1}_{\tau}} u_1,u_2 \right ) - \left ( L^{0,1} _{\bar \xi ^{\theta ^1}_{\tau}} L^{1,0} _{\xi ^{\theta ^1}_{\sigma}} u_1,u_2 \right ) -  \left ( \left ( L^{1,0} _{\xi ^{\theta ^2}_{\sigma}} L^{0,1} _{\bar \xi ^{\theta ^2}_{\tau}} u_1,u_2 \right ) -  \left ( L^{0,1} _{\bar \xi ^{\theta ^2}_{\tau}} L^{1,0} _{\xi ^{\theta ^2}_{\sigma}} u_1,u_2 \right ) \right ) \\
&=& \left ( L^{1,0} _{v_{\sigma}} L^{0,1} _{\bar \xi ^{\theta ^1}_{\tau}} u_1,u_2 \right ) + \left ( L^{1,0} _{\xi ^{\theta ^2}_{\sigma}} L^{0,1} _{\bar \xi ^{\theta ^1}_{\tau}} u_1,u_2 \right ) - \left ( L^{0,1} _{\bar \xi ^{\theta ^1}_{\tau}} L^{1,0} _{\xi ^{\theta ^1}_{\sigma}} u_1,u_2 \right ) \\
&& -  \left ( \left ( L^{1,0} _{\xi ^{\theta ^2}_{\sigma}} L^{0,1} _{\bar \xi ^{\theta ^2}_{\tau}} u_1,u_2 \right ) -  \left ( L^{0,1} _{\bar \xi ^{\theta ^2}_{\tau}} L^{1,0} _{\xi ^{\theta ^2}_{\sigma}} u_1,u_2 \right ) \right ) \\
&=& \left ( L^{1,0} _{v_{\sigma}} L^{0,1} _{\bar \xi ^{\theta ^1}_{\tau}} u_1,u_2 \right ) - \left ( L^{0,1} _{\bar \xi ^{\theta ^1}_{\tau}} L^{1,0} _{\xi ^{\theta ^1}_{\sigma}} u_1,u_2 \right ) \\
&& -  \left ( \left ( L^{1,0} _{\xi ^{\theta ^2}_{\sigma}} L^{0,1} _{\bar \xi ^{\theta ^2}_{\tau}} u_1,u_2 \right ) - \left ( L^{1,0} _{\xi ^{\theta ^2}_{\sigma}} L^{0,1} _{\bar \xi ^{\theta ^1}_{\tau}} u_1,u_2 \right )  -  \left ( L^{0,1} _{\bar \xi ^{\theta ^2}_{\tau}} L^{1,0} _{\xi ^{\theta ^2}_{\sigma}} u_1,u_2 \right ) \right ) \\
&=& \left ( L^{1,0} _{v_{\sigma}} L^{0,1} _{\bar \xi ^{\theta ^1}_{\tau}} u_1,u_2 \right ) - \left ( L^{0,1} _{\bar \xi ^{\theta ^1}_{\tau}} L^{1,0} _{\xi ^{\theta ^1}_{\sigma}} u_1,u_2 \right ) + \left ( L^{1,0} _{\xi ^{\theta ^2}_{\sigma}} L^{0,1} _{\bar v_{\tau}} u_1,u_2 \right ) +  \left ( L^{0,1} _{\bar \xi ^{\theta ^2}_{\tau}} L^{1,0} _{\xi ^{\theta ^2}_{\sigma}} u_1,u_2 \right ).
\end{eqnarray*}
Now, one computes that 
\[
L^{1,0} _{\xi ^{\theta ^i} _{\sigma}} L^{0,1} _{\bar v _{\tau}} u_1 = L^{1,0} _{\xi ^{\theta ^i} _{\sigma}} (\bar v _{\tau} \lrcorner \dbar u_1) = \bar v _{\tau} \lrcorner (  L^{1,0} _{\xi ^{\theta ^i} _{\sigma}}\dbar u_1) + \left (\overline{ \bar \xi ^{\theta ^i}_{\tau} \lrcorner \dbar v _{\tau}} \right )  \lrcorner \dbar u_1.
\]
The first term on the right hand side vanishes because $\ff _1 \in \Sigma _t$, and the second term vanishes because $\bar \xi ^{\theta ^i}_{\tau} \lrcorner \overline{\dbar v _{\tau}} $ is a vertical $(0,1)$-vector field.  Next, by Lemma \ref{lie-deriv-lemma} and the holomorphicity of $u_i$ along $X_t$ one has
\[
 \left ( L^{1,0} _{v_{\sigma}} L^{0,1} _{\bar \xi ^{\theta ^1}_{\tau}} u_1,u_2 \right ) =  - \left (L^{0,1} _{\bar \xi ^{\theta ^1}_{\tau}} u_1, L^{0,1} _{\bar v_{\sigma}}u_2 \right ) = 0.
\]
Therefore 
\[
 \left (\Theta ^{\sL ^{\theta _1}}_{\sigma \bar \tau} \ff _1, \ff_2\right ) - \left (\Theta ^{\sL ^{\theta _2}}_{\sigma \bar \tau} \ff _1, \ff_2\right ) =  \left ( L^{0,1} _{\bar \xi ^{\theta ^2}_{\tau}} L^{1,0} _{\xi ^{\theta ^2}_{\sigma}} u_1,u_2 \right ) - \left ( L^{0,1} _{\bar \xi ^{\theta ^1}_{\tau}} L^{1,0} _{\xi ^{\theta ^1}_{\sigma}} u_1,u_2 \right ) 
\]
Now,
\begin{eqnarray*}
&& \left ( L^{0,1} _{\bar \xi ^{\theta ^2}_{\tau}} L^{1,0} _{\xi ^{\theta ^2}_{\sigma}} u_1,u_2 \right ) - \left ( L^{0,1} _{\bar \xi ^{\theta ^1}_{\tau}} L^{1,0} _{\xi ^{\theta ^1}_{\sigma}} u_1,u_2 \right ) \\
&=& - \left ( L^{0,1} _{\bar \xi ^{\theta ^2}_{\tau}} L^{1,0} _{v_{\sigma}} u_1,u_2 \right )  + \left ( L^{0,1} _{\bar \xi ^{\theta ^2}_{\tau}} L^{1,0} _{\xi ^{\theta ^1}_{\sigma}} u_1,u_2 \right ) - \left ( L^{0,1} _{\bar \xi ^{\theta ^1}_{\tau}} L^{1,0} _{\xi ^{\theta ^1}_{\sigma}} u_1,u_2 \right ) \\
&=& - \left ( L^{0,1} _{\bar \xi ^{\theta ^2}_{\tau}} L^{1,0} _{v_{\sigma}} u_1,u_2 \right )  - \left ( L^{0,1} _{\bar v_{\tau}} L^{1,0} _{\xi ^{\theta ^1}_{\sigma}} u_1,u_2 \right ) \\
&=& - \bar \tau \left (L^{1,0} _{v_{\sigma}} u_1, u_2\right ) + (L^{1,0} _{v_{\sigma}} u_1, L^{1,0} _{\xi ^{\theta ^2} _{\tau} }u_2) + \left ( L^{1,0} _{\xi ^{\theta ^1}_{\sigma}} u_1, L^{1,0} _{v_{\tau}}u_2 \right )\\
&=&  \bar \tau \left (u_1, L^{0,1} _{\bar v _{\sigma}}u_2\right ) + (L^{1,0} _{v_{\sigma}} u_1, L^{1,0} _{\xi ^{\theta ^2} _{\tau} }u_2) +  \left ( L^{1,0} _{\xi ^{\theta ^1}_{\sigma}} u_1, L^{1,0} _{v_{\tau}}u_2 \right )\\
&=& (L^{1,0} _{v_{\sigma}} u_1, L^{1,0} _{\xi ^{\theta ^2} _{\tau} }u_2) +  \left ( L^{1,0} _{\xi ^{\theta ^1}_{\sigma}} u_1, L^{1,0} _{v_{\tau}}u_2 \right ).
\end{eqnarray*}
The last equality holds because $\bar \tau \left (u_1, L^{0,1} _{\bar v _{\sigma}}u_2\right ) =  \left (L^{0,1} _{\bar \xi ^{\theta ^i}_{\tau}} u_1, L^{0,1} _{\bar v _{\sigma}}u_2\right ) +  \left (u_1, L^{1,0} _{\xi ^{\theta ^i}_{\tau}} L^{0,1} _{\bar v _{\sigma}}u_2\right )  = 0$.  

Now, since $v_{\sigma}$ is vertical 
\[
\iota _{X_t} ^* L^{1,0} _{v_{\sigma}} u_1 = \iota _{X_t} ^* (\nabla^{1,0} (v_{\sigma} \lrcorner u_1) + v_{\sigma} \lrcorner (\nabla ^{1,0} u_1) )= \iota _{X_t} ^* (\nabla^{1,0} (v_{\sigma} \lrcorner u_1)),
\]
and hence $\iota _{X_t} ^* L^{1,0} _{v_{\sigma}} u_1  \perp {\rm Kernel} (\Box ^{1,0}) = {\rm Kernel} (\Box) = {\rm kernel} (\dbar)$,  with the last two equalities holding because $\nabla^{1,0} (v_{\sigma} \lrcorner u_1)$ is an $E$-valued $(n,0)$-form.  Therefore 
\[
(L^{1,0} _{v_{\sigma}} u_1, L^{1,0} _{\xi ^{\theta ^2} _{\tau} }u_2) = (P_t ^{\perp} L^{1,0} _{v_{\sigma}} u_1, L^{1,0} _{\xi ^{\theta ^2} _{\tau} }u_2) = (P_t ^{\perp}L^{1,0} _{v_{\sigma}} u_1, P_t ^{\perp} L^{1,0} _{\xi ^{\theta ^2} _{\tau} }u_2),
\]
and similarly $ \left ( L^{1,0} _{\xi ^{\theta ^1}_{\sigma}} u_1, L^{1,0} _{v_{\tau}}u_2 \right ) =  \left (P_t ^{\perp} L^{1,0} _{\xi ^{\theta ^1}_{\sigma}} u_1, P_t ^{\perp} L^{1,0} _{v_{\tau}}u_2 \right )$.  Hence 
\begin{eqnarray*}
\left (L^{1,0} _{v_{\sigma}} u_1, L^{1,0} _{\xi ^{\theta ^2} _{\tau} }u_2\right ) +  \left ( L^{1,0} _{\xi ^{\theta ^1}_{\sigma}} u_1, L^{1,0} _{v_{\tau}}u_2 \right ) &=& \left (P_t ^{\perp}L^{1,0} _{v_{\sigma}} u_1, P_t ^{\perp} L^{1,0} _{\xi ^{\theta ^2} _{\tau} }u_2\right ) + \left (P_t ^{\perp} L^{1,0} _{\xi ^{\theta ^1}_{\sigma}} u_1, P_t ^{\perp} L^{1,0} _{v_{\tau}}u_2 \right )\\
&=& \left (P_t ^{\perp}L^{1,0} _{\xi ^{\theta _1}_{\sigma}} u_1, P_t ^{\perp} L^{1,0} _{\xi ^{\theta ^2} _{\tau} }u_2\right ) - \left (P_t ^{\perp}L^{1,0} _{\xi ^{\theta _2}_{\sigma}} u_1, P_t ^{\perp} L^{1,0} _{\xi ^{\theta ^2} _{\tau} }u_2\right ) \\
&& + \left ( \left (P_t ^{\perp} L^{1,0} _{\xi ^{\theta ^1}_{\sigma}} u_1, P_t ^{\perp} L^{1,0} _{\xi ^{\theta _1}_{\tau}}u_2 \right ) -  \left (P_t ^{\perp} L^{1,0} _{\xi ^{\theta ^1}_{\sigma}} u_1, P_t ^{\perp} L^{1,0} _{\xi ^{\theta _2}_{\tau}}u_2 \right ) \right ) \\
&=&\left (P_t ^{\perp} L^{1,0} _{\xi ^{\theta ^1}_{\sigma}} u_1, P_t ^{\perp} L^{1,0} _{\xi ^{\theta _1}_{\tau}}u_2 \right ) -  \left (P_t ^{\perp}L^{1,0} _{\xi ^{\theta _2}_{\sigma}} u_1, P_t ^{\perp} L^{1,0} _{\xi ^{\theta ^2} _{\tau} }u_2\right )\\
&=& \left ( \two ^{\theta_1} \ff _1 , \two ^{\theta _1} \ff_2 \right ) - \left ( \two ^{\theta_2} \ff _1 , \two ^{\theta _2} \ff_2 \right ) .
\end{eqnarray*}
Thus $\left (\Theta ^{\sL ^{\theta _1}}_{\sigma \bar \tau} \ff _1, \ff_2\right ) - \left (\Theta ^{\sL ^{\theta _2}}_{\sigma \bar \tau} \ff _1, \ff_2\right ) = \left ( \two ^{\theta_1} \ff _1 , \two ^{\theta _1} \ff_2 \right ) - \left ( \two ^{\theta_2} \ff _1 , \two ^{\theta _2} \ff_2 \right )$, as required.
\end{proof}

%\begin{s-rmk}
%The first proof of Theorem \ref{indep} was incorrect; it implicitly assumed that the projection operator $P_{\sH}$ is smooth, which is only the case if $\sH \to B$ is locally trivial.  We thank Pranav Upadrashta for pointing this out to us.  The proof given here is due to him. 
%\red
%\end{s-rmk}

\begin{thm}\label{L2-curvature-prop}
For any $t \in B$ and any smooth $f_1, f_2 \in \sL_t$ one has the formula 
\begin{eqnarray}\label{l2-theta-curvature-integrated}
 (\Theta ^{\sL^{\theta}}_{\sigma \bar \tau} f _1, f _2)_t &=& \int _{X_t} \left < \Theta (h)_{\xi ^{\theta} _{\sigma}\bar \xi^{\theta}_{\tau}} f_1 \wedge \bar f _2, h \right > - \int _{X_t} \left < (\dbar \xi ^{\theta}_{\sigma})\lrcorner f_1\wedge \overline{(\dbar \xi ^{\theta} _{\tau})\lrcorner f_2}, h \right >\\
 \nonumber && \qquad + \int _{X_t} \left <L^{0,1}_{[\xi ^{\theta}_{\sigma}, \bar \xi ^{\theta}_{\tau}]^{0,1}} f_1\wedge \bar f_2, h \right > + \int _{X_t} \left <f_1\wedge \overline{ L^{0,1}_{[\xi ^{\theta}_{\tau}, \bar \xi ^{\theta}_{\sigma}]^{0,1}}f_2}, h \right >.
\end{eqnarray}
for the $(1,1)$-component of $\Theta^{\sL ^{\theta}}$.  In particular, if $f _1,f_2 \in \sH_t$ then 
\begin{equation}\label{curv-on-holos}
(\Theta ^{\sL^{\theta}}_{\sigma \bar \tau} f_1, f_2)_t = \ \int _{X_t} \left < \Theta (h)_{\xi ^{\theta} _{\sigma}\bar \xi^{\theta}_{\tau}} f_1 \wedge \bar f _2, h \right > - \int _{X_t} \left < (\dbar \xi ^{\theta}_{\sigma})\lrcorner f_1\wedge \overline{(\dbar \xi ^{\theta} _{\tau})\lrcorner f_2}, h \right >.
\end{equation}
\end{thm}

In the proof of Theorem \ref{L2-curvature-prop} the following lemma is convenient.

\begin{lem}\label{l10-computation}
Let $p:X \to B$ be a holomorphic submersion, let $v \in \Gamma (X, \sC^{\infty} (\Lambda ^{n,0} _X \tensor E))$ be a smooth $E$-valued $(n,0)$-form, let $\sigma$ be a  holomorphic vector fields on a coordinate neighborhood $U$ in $B$ with coordinates $s= (s^1,..., s^m)$, and let $p = (p^1,...,p^m)$ in these coordinates.  Choose local holomorphic coordinates $(z, t)$ on an open set in $p^{-1}(U)$ such that $t^i = p^i$ and, with $d\vec z := dz^1 \wedge ... \wedge dz ^n$, write 
\[
v = v_o d\vec z + dt ^{\mu} \wedge \alpha_{\mu} 
\]
for some local section $v_o$ of $E$ and local $E$-valued $(n-1,0)$-forms $\alpha_1,...,\alpha_m$.  Then 
\[
L^{1,0} _{\xi ^{\theta}_{\sigma}} v = (\nabla ^{1,0}_{\xi ^{\theta}_{\sigma}} v_o ) d\vec z + v _o \di \left (\xi ^{\theta}_{\sigma} \lrcorner  d\vec z\right ) + \di \sigma ^{\mu} \wedge \alpha_{\mu} + dt ^{\mu} \wedge L^{1,0}_{\xi ^{\theta}_{\sigma}} \alpha _{\mu}.
\]
\end{lem}
\begin{proof}
We calculate that 
\begin{eqnarray*}
L^{1,0} _{\xi ^{\theta}_{\sigma}} v &=& \nabla ^{1,0} \left ( v_o \left ( \xi ^{\theta}_{\sigma}\lrcorner  d\vec z\right ) + \sigma ^{\mu} \alpha_{\mu}  - dt ^{\mu} \wedge (\xi ^{\theta}_{\sigma}\lrcorner \alpha _{\mu})\right ) + \xi ^{\theta}_{\sigma} \lrcorner \left ( \nabla ^{1,0} v_o \wedge d\vec z - dt ^{\mu} \wedge \nabla ^{1,0} \alpha_{\mu}\right ) \\ 
&=&( \nabla ^{1,0} v_o) \wedge \left ( \xi ^{\theta}_{\sigma} \lrcorner  d\vec z\right ) + v_o \di \left ( \xi ^{\theta}_{\sigma} \lrcorner  d\vec z\right ) + \di \sigma ^{\mu} \wedge \alpha_{\mu}  + \sigma ^{\mu} \nabla ^{1,0} \alpha _{\mu} + dt ^{\mu} \wedge \nabla ^{1,0}(\xi ^{\theta}_{\sigma}\lrcorner \alpha _{\mu})  \\
&& + (\nabla ^{1,0}_{\xi ^{\theta}_{\sigma}} v_o ) d\vec z  - \nabla ^{1,0} v_o \wedge ( \xi ^{\theta}_{\sigma}\lrcorner d\vec z)  - \sigma ^{\mu} \nabla ^{1,0} \alpha _{\mu}  + dt ^{\mu} \wedge ( \xi ^{\theta}_{\sigma}\lrcorner  \nabla ^{1,0} \alpha_{\mu}) \\ 
&=& (\nabla ^{1,0}_{\xi ^{\theta}_{\sigma}} v_o ) d\vec z + v _o \di \left ( \xi ^{\theta}_{\sigma}\lrcorner  d\vec z\right ) + \di \sigma ^{\mu} \wedge \alpha_{\mu} + dt ^{\mu} \wedge L^{1,0}_{\xi ^{\theta}_{\sigma}} \alpha _{\mu},
\end{eqnarray*}
as claimed.
\end{proof}

\begin{proof}[Proof of Theorem \ref{L2-curvature-prop}]
Fix a section $\ff \in \Gamma (B, \sC^{\infty}(\sL))$.  Fix an $E$-valued $(n,0)$-form $u$ representing $\fa (\ff)$, and locally write $u = u_o d\vec z+dt^{\mu} \wedge \alpha _{\mu}$ in the coordinates of Lemma \ref{l10-computation}.  If $\sigma$ and $\tau$ are holomorphic vector fields on an open subset of $B$ then an application of $L^{0,1} _{\bar \xi ^{\theta}_{\tau}}$ to the formula for $L^{1,0}_{\xi ^{\theta}_{\sigma}}u$ from Lemma \ref{l10-computation} gives 
\begin{eqnarray*}
L^{0,1} _{{\bar \xi ^{\theta}_{\tau}}} L^{1,0} _{\xi ^{\theta}_{\sigma}} u &=& \left (\dbar _{{\bar \xi ^{\theta}_{\tau}}}\nabla ^{1,0}_{\xi ^{\theta}_{\sigma}} u_o \right ) d\vec z + (\dbar u _o ({\bar \xi ^{\theta}_{\tau}}))\di \left ( \xi ^{\theta}_{\sigma} \lrcorner  d\vec z\right )- u _o {\bar \xi ^{\theta}_{\tau}} \lrcorner \left ( \di \left ( \dbar \xi ^{\theta}_{\sigma} \lrcorner  d\vec z\right )\right ) \\
&& + \di \sigma ^{\mu} \wedge \left ( {\bar \xi ^{\theta}_{\tau}} \lrcorner \dbar \alpha_{\mu}\right ) + dt ^{\mu} \wedge L^{0,1} _{{\bar \xi ^{\theta}_{\tau}}}L^{1,0}_{\xi ^{\theta}_{\sigma}} \alpha _{\mu},
\end{eqnarray*}
where in the third term on the right hand side of the equality we have used the identity $\dbar \di = - \di \dbar$.  Applying Lemma \ref{l10-computation} to $v = L^{0,1} _{{\bar \xi ^{\theta}_{\tau}}} u = (\dbar _{{\bar \xi ^{\theta}_{\tau}}} u_o)d\vec z + dt ^{\mu} \wedge L^{0,1}_{{\bar \xi ^{\theta}_{\tau}}} \alpha _{\mu}$ then yields  
\[
L^{1,0} _{\xi ^{\theta}_{\sigma}} L^{0,1} _{{\bar \xi ^{\theta}_{\tau}}} u  =  (\nabla ^{1,0}_{\xi ^{\theta}_{\sigma}} \dbar _{{\bar \xi ^{\theta}_{\tau}}} u_o)  d\vec z + (\dbar  u_o({\bar \xi ^{\theta}_{\tau}})) \di \left ( \xi ^{\theta}_{\sigma} \lrcorner  d\vec z\right ) + \di \sigma ^{\mu} \wedge L^{0,1}_{{\bar \xi ^{\theta}_{\tau}}} \alpha _{\mu} + dt ^{\mu} \wedge L^{1,0}_{\xi ^{\theta}_{\sigma}} L^{0,1}_{{\bar \xi ^{\theta}_{\tau}}} \alpha _{\mu}.
\]
Therefore 
\begin{eqnarray*}\label{smooth-theta-curvature}
\nonumber \fa\left ( \Theta ^{\sL^{\theta}}_{\sigma \bar \tau} \ff \right ) &=& \left ( \nabla ^{1,0} _{\xi ^{\theta}_{\sigma}} \dbar _{{\bar \xi ^{\theta}_{\tau}}} - \dbar _{{\bar \xi ^{\theta}_{\tau}}} \nabla ^{1,0}_{\xi ^{\theta}_{\sigma}} \right )u_o d\vec z + {\bar \xi ^{\theta}_{\tau}} \lrcorner ( u_o \di (\dbar \xi ^{\theta}_{\sigma} \lrcorner d\vec z))  +dt^{\mu} \wedge \left (  L^{1,0} _{\xi^{\theta}_{\sigma}} L^{0,1} _{\bar \xi^{\theta}_{\tau}} - L^{0,1} _{\bar \xi^{\theta}_{\tau}} L^{1,0}_{\xi^{\theta}_{\sigma}}\right ) \alpha _{\mu}\\
&=& \Theta (h)_{\xi ^{\theta}_{\sigma}{\bar \xi ^{\theta}_{\tau}}}u_o d\vec z + \nabla _{[\xi ^{\theta}_{\sigma}, {\bar \xi ^{\theta}_{\tau}}]}u_o d\vec z +  u_o ({\bar \xi ^{\theta}_{\tau}} \lrcorner \di (\dbar \xi ^{\theta}_{\sigma} \lrcorner d\vec z)) +dt^{\mu} \wedge \left (  L^{1,0} _{\xi^{\theta}_{\sigma}} L^{0,1} _{\bar \xi^{\theta}_{\tau}} - L^{0,1} _{\bar \xi^{\theta}_{\tau}} L^{1,0}_{\xi^{\theta}_{\sigma}} \right ) \alpha _{\mu}.
\end{eqnarray*}
Keeping in mind that $[\xi ^{\theta}_{\sigma}, \bar \xi ^{\theta}_{\tau}]$ is a vertical vector field, we compute that 
\begin{eqnarray*}
&& L_{[\xi^{\theta}_{\sigma}, {\bar \xi^{\theta}_{\tau}}]} (u_o d\vec z + dt ^{\mu} \wedge \alpha _{\mu}) \\
&=& [\xi^{\theta}_{\sigma}, {\bar \xi^{\theta}_{\tau}}] \lrcorner \nabla (u_o d\vec z + dt ^{\mu} \wedge \alpha _{\mu}) + \nabla([\xi^{\theta}_{\sigma}, {\bar \xi^{\theta}_{\tau}}] \lrcorner (u_o d \vec z+ dt ^{\mu} \wedge \alpha _{\mu})) \\
&=&[\xi^{\theta}_{\sigma}, {\bar \xi^{\theta}_{\tau}}] \lrcorner (\nabla u_o \wedge d\vec z- dt^{\mu} \wedge \nabla \alpha _{\mu}) + \nabla (u_o ([\xi^{\theta}_{\sigma}, {\bar \xi^{\theta}_{\tau}}]\lrcorner d\vec z) - dt ^{\mu} \wedge [\xi^{\theta}_{\sigma}, \bar \xi^{\theta} _{\tau}] \lrcorner \alpha _{\mu})\\
&=& \nabla _{[\xi^{\theta}_{\sigma}, {\bar \xi^{\theta}_{\tau}}]} u_o d\vec z  +u_o \di ([\xi^{\theta}_{\sigma}, {\bar \xi^{\theta}_{\tau}}] \lrcorner d\vec z)+ dt ^{\mu} \wedge L_{[\xi^{\theta}_{\sigma}, \bar \xi^{\theta} _{\tau}]} \alpha _{\mu}\\
&=& \nabla _{[\xi^{\theta}_{\sigma}, {\bar \xi^{\theta}_{\tau}}]} u_o d\vec z  - u_o \di ({\bar \xi^{\theta}_{\tau}}\lrcorner (\dbar\xi^{\theta} _{\sigma} \lrcorner d\vec z))+ dt ^{\mu} \wedge L_{[\xi^{\theta}_{\sigma}, \bar \xi^{\theta} _{\tau}]} \alpha _{\mu}.
\end{eqnarray*}
Therefore, with $\fa (\ff) = u_o d\vec z + dt ^{\mu} \wedge \alpha _{\mu}$, we have
\begin{eqnarray*}
\fa \left (\Theta ^{\sL^{\theta}}_{\sigma \bar \tau} \ff \right )&=& \Theta (h)_{\xi ^{\theta}_{\sigma}{\bar \xi ^{\theta}_{\tau}}} u_o d\vec z  + L_{[\xi ^{\theta} _{\sigma}, \bar \xi ^{\theta} _{\tau}]}\fa (\ff) + u_o \left ( {\bar \xi ^{\theta}_{\tau}} \lrcorner \di (\dbar \xi ^{\theta}_{\sigma} \lrcorner d\vec z) + \di ({\bar \xi^{\theta}_{\tau}}\lrcorner (\dbar\xi^{\theta} _{\sigma} \lrcorner d\vec z)) \right ) \\
&& \quad + dt^{\mu} \wedge \left (  L^{1,0} _{\xi^{\theta}_{\sigma}} L^{0,1} _{\bar \xi^{\theta}_{\tau}} - L^{0,1} _{\bar \xi^{\theta}_{\tau}} L^{1,0}_{\xi^{\theta}_{\sigma}} - L_{[\xi ^{\theta} _{\sigma} , \bar \xi ^{\theta}_{\tau}]} \right )\alpha _{\mu}.
\end{eqnarray*}
Now, 
\begin{eqnarray*}
L^{1,0} _{\bar \xi^{\theta} _{\tau}} \left (  u_o (\dbar \xi^{\theta} _{\sigma} \lrcorner d\vec z) \right ) &=& \nabla ^{1,0} \left ( u_o \bar \xi^{\theta}_{\tau} \lrcorner (\dbar \xi^{\theta}_{\sigma} \lrcorner d\vec z)\right ) + \bar \xi^{\theta}_{\tau} \lrcorner  \nabla ^{1,0} \left (u_o \dbar \xi^{\theta}_{\sigma} \lrcorner d\vec z\right )\\
&=& (\nabla ^{1,0} u_o) \wedge \left ( \bar \xi^{\theta}_{\tau} \lrcorner (\dbar \xi^{\theta}_{\sigma} \lrcorner d\vec z)\right ) + u_o \di \left (  \bar \xi^{\theta}_{\tau} \lrcorner (\dbar \xi^{\theta}_{\sigma} \lrcorner d\vec z)\right ) \\
&& +\bar \xi^{\theta}_{\tau} \lrcorner \left (  \nabla ^{1,0} u_o \wedge ( \dbar \xi^{\theta}_{\sigma} \lrcorner d\vec z+ u_o \di (\dbar \xi ^{\theta} _{\sigma}\lrcorner dz)\right )\\
&=& (\nabla ^{1,0}_{\bar \xi ^{\theta} _{\tau}} u_o )  \left ( \dbar \xi^{\theta}_{\sigma} \lrcorner d\vec z \right ) + u_o \di \left (  \bar \xi^{\theta}_{\tau} \lrcorner (\dbar \xi^{\theta}_{\sigma} \lrcorner d\vec z)\right ) + u_o \di (\dbar \xi ^{\theta} _{\sigma}\lrcorner dz)\\
&=& u_o \di \left (  \bar \xi^{\theta}_{\tau} \lrcorner (\dbar \xi^{\theta}_{\sigma} \lrcorner d\vec z)\right ) + u_o \di (\dbar \xi ^{\theta} _{\sigma}\lrcorner dz),
\end{eqnarray*}
\[
L^{1,0} _{\bar \xi^{\theta} _{\tau}} (dt ^{\mu} \wedge \alpha _{\mu} )= dt ^{\mu} \wedge L^{1,0} _{\bar \xi^{\theta} _{\tau}}  \alpha _{\mu}
\]
and
\[
L^{0,1} _{\xi^{\theta}_{\sigma}} \fa (\ff) = \dbar (\xi^{\theta}_{\sigma} \lrcorner \fa (\ff))+ \xi^{\theta}_{\sigma} \lrcorner \dbar \fa (\ff) = (\dbar \xi^{\theta} _{\sigma})\lrcorner \fa (\ff).
\]
Since $L^{1,0} _{\bar \xi ^{\theta} _{\tau}} \fa (\ff) = \nabla ^{1,0} (\bar \xi ^{\theta} _{\tau} \lrcorner  \fa (\ff)) + \bar \xi ^{\theta} _{\tau} \lrcorner \nabla ^{1,0} \fa (\ff) = 0$ because $\fa (\ff)$ and $\nabla ^{1,0} \fa (\ff)$ are of bi-degree $(n,0)$ and $(n+1,0)$, we have established the following formula.
\begin{equation}\label{l2-curvature-general-formula}
\iota _{X_t} ^* \Theta ^{\sL^{\theta}}_{\sigma \bar \tau} \ff = \iota _{X_t} ^* \left ( \Theta (h)_{\xi ^{\theta}_{\sigma}{\bar \xi ^{\theta}_{\tau}}}\fa (\ff)   -  [L^{0,1} _{\xi^{\theta} _{\sigma}}, L^{1,0}_{\bar \xi^{\theta} _{\tau}}] \fa (\ff) \right ) +\iota _{X_t} ^*  \left (L_{[\xi^{\theta} _{\sigma},\bar \xi^{\theta} _{\tau}]} \fa (\ff)\right )
\end{equation}
for sections $\ff \in \Gamma (B,\sC^{\infty} (\sL))$.  

Since $[\xi^{\theta}_{\sigma}, \bar \xi^{\theta}_{\tau}]$ is vertical, if $\ff_1$ and $\ff_2$ are holomorphic on fibers then $L^{0,1} _{[\xi^{\theta}_{\sigma}, \bar \xi^{\theta}_{\tau}]^{0,1}}\fa (\ff_1) = 0$, and similarly $L^{0,1}_{[\xi ^{\theta}_{\tau}, \bar \xi ^{\theta}_{\sigma}]^{0,1}}\fa (\ff_2)=0$.  Thus \eqref{curv-on-holos} follows from \eqref{l2-theta-curvature-integrated}, to which proof we now turn.

By Formula \eqref{l2-curvature-general-formula} we have 
\begin{eqnarray*}
(\Theta ^{\sL^{\theta}}_{\sigma \bar \tau} \ff _1, \ff _2)_t &=& \int _{X_t} \left <\left ( \Theta (h)_{\xi ^{\theta} _{\sigma}\bar \xi^{\theta}_{\tau}} \fa (\ff_1) \right )\wedge \overline{\fa (\ff _2)}, h \right > \\
&& \quad + \int _{X_t} \left <L^{1,0} _{\bar \xi ^{\theta}_{\tau}} L^{0,1}_{\xi ^{\theta}_{\sigma}}\fa (\ff_1)\wedge \overline{ \fa (\ff_2)}, h \right > + \int _{X_t} \left <L_{[\xi ^{\theta}_{\sigma}, \bar \xi ^{\theta}_{\tau}]} \fa (\ff_1)\wedge \overline{ \fa (\ff_2)}, h \right >.
\end{eqnarray*}
Observing that $L^{0,1}_{\xi ^{\theta}_{\sigma}}\fa (\ff_1)$ is an $E$-valued $(n-1,1)$-form,  by Proposition \ref{restricted-product-rule-2} we have 
\begin{eqnarray*}
&& \int _{X_t} \left <L^{1,0} _{\bar \xi ^{\theta}_{\tau}} L^{0,1}_{\xi ^{\theta}_{\sigma}}\fa (\ff_1)\wedge \overline{ \fa (\ff_2)}, h \right >\\
&=& \int _{X_t} L _{\bar \xi ^{\theta}_{\tau}} \left < L^{0,1}_{\xi ^{\theta}_{\sigma}}\fa (\ff_1)\wedge \overline{\fa (\ff_2)}, h \right > - \int _{X_t} \left < L^{0,1}_{\xi ^{\theta}_{\sigma}}\fa (\ff_1)\wedge \overline{L^{0,1}_{\xi ^{\theta} _{\tau}}\fa (\ff_2)}, h \right >\\
&=&\bar \tau \int _{X_t} \left < L^{0,1}_{\xi ^{\theta}_{\sigma}}\fa (\ff_1)\wedge \overline{\fa (\ff_2)}, h \right > - \int _{X_t} \left < L^{0,1}_{\xi ^{\theta}_{\sigma}}\fa (\ff_1)\wedge \overline{L^{0,1}_{\xi ^{\theta} _{\tau}}\fa (\ff_2)}, h \right >\\
&=& - \int _{X_t} \left < L^{0,1}_{\xi ^{\theta}_{\sigma}}\fa (\ff_1)\wedge \overline{L^{0,1}_{\xi ^{\theta} _{\tau}}\fa (\ff_2)}, h \right >,
\end{eqnarray*}
the last equality being a consequence of the fact that $\left < L^{0,1}_{\xi ^{\theta}_{\sigma}}\fa (\ff_1)\wedge \overline{\fa (\ff_2)}, h \right >$ is an $(n-1,n+1)$-form, and thus integrates to $0$ on $X_t$.  Finally, for reasons of bi-degree
\[
 \int _{X_t} \left <L_{[\xi ^{\theta}_{\sigma}, \bar \xi ^{\theta}_{\tau}]} \fa (\ff_1)\wedge \overline{ \fa (\ff_2)}, h \right > 
= \int _{X_t} \left <L^{0,1}_{[\xi ^{\theta}_{\sigma}, \bar \xi ^{\theta}_{\tau}]^{0,1}} \fa (\ff_1)\wedge \overline{ \fa (\ff_2)}, h \right > + \int _{X_t} \left <L^{1,0}_{[\xi ^{\theta}_{\sigma}, \bar \xi ^{\theta}_{\tau}]^{1,0}} \fa (\ff_1)\wedge \overline{ \fa (\ff_2)}, h \right >,
\]
and by Lemma \ref{lie-deriv-lemma} and the fact that, since $\sigma$ and $\tau$ are holomorphic, $[\xi^{\theta} _{\sigma} , \bar \xi ^{\theta} _{\tau}]$ is vertical 
\begin{eqnarray*}
\int _{X_t} \left <L^{1,0}_{[\xi ^{\theta}_{\sigma}, \bar \xi ^{\theta}_{\tau}]^{1,0}} \fa (\ff_1)\wedge \overline{ \fa (\ff_2)}, h \right > &=& - \int _{X_t} \left < \fa (\ff_1)\wedge \overline{ L^{0,1}_{\overline{[\xi ^{\theta}_{\sigma}, \bar \xi ^{\theta}_{\tau}]^{1,0}}}\fa (\ff_2)}, h \right >\\
&=& \int _{X_t} \left < \fa (\ff_1)\wedge \overline{ L^{0,1}_{[\xi ^{\theta}_{\tau}, \bar \xi ^{\theta}_{\sigma}]^{0,1}}\fa (\ff_2)}, h \right >.
\end{eqnarray*}
This completes the proof.
\end{proof}

\noi Theorem \ref{B2-thm} now follows from Theorem \ref{L2-curvature-prop}, Proposition \ref{Chern-connections-for-L^theta}, Definition \ref{sff-and-curv-iBLS-defn}.ii and Proposition \ref{H-integrable=>L-11}.
\qed

\section{The Second Fundamental Form of $\sH  \subset \sL^{\theta}$}\label{sff-section}

We now turn our attention to the proof of Theorem \ref{berndtsson-reps-thm}.

Let $\theta \subset T_X$ be a smooth horizontal distribution.  In the most general setting one cannot assume that $\sH \subset \sL ^{\theta}$ is regular.  Thus our second fundamental form, which by Definition \ref{iBLS-defn}.ii is 
\[
\two ^{\sL ^{\theta}/\sH} = P_{\sH} ^{\perp} \nabla ^{\sL^{\theta}1,0},
\]
does not depend smoothly on $t \in B$.  The only term that is not smooth in $t$ is the Bergman projection $P_t$, so while differentiation in the vertical directions is legitimate, we must be careful not to differentiate this projection with respect to $t$.  %The reader should note, as we proceed, that the Bergman projection is differentiated only in the vertical direction.  

\subsection{A first formula for the second fundamental form}

Our first result on the second fundamental form is useful in obtaining a more crude estimate for the curvature of $\sH \to B$ under the assumption that the metric $h$ for $E \to X$ is Nakano-positive along the fibers $X_t$.

\begin{prop}\label{sff-prop}
For every $t \in B$,  $\ff \in \Gamma (B, \sC^{\infty}(\sL))$ satisfying $\ff (t) \in \sH _t$, and $\tau \in \cO (B)_t$ one has 
\[
\iota _t \two ^{\sL ^{\theta}/\sH} (\tau) \ff = - G_t \dbar ^* _{X_t}  \left ( \iota^* _{X_t} \left ( (\xi^{\theta} _{\tau} \lrcorner \Theta (h)) \alpha \right ) +   \left (\nabla ^{1,0}\left ( \dbar \xi^{\theta} _{\tau}\lrcorner \alpha \right )\right ) + L^{1,0} _{\xi ^{\theta} _{\tau}} \dbar \alpha \right ),
\]
where $\alpha$ is any $E$-valued $(n,0)$-form representing $\fa (\ff)$.
\end{prop}

\begin{proof}
Since $\iota _t \two ^{\sL ^{\theta}/\sH}(\tau) \ff = P^{\perp}_t \iota _{X_t} ^* L^{1,0} _{\xi ^{\theta} _{\tau}}\alpha$ is orthogonal to $\sH_t = {\rm Harm}_{n,0}$,  the Hodge Theorem implies that 
\[
\iota _t \two ^{\sL ^{\theta}/\sH}(\tau) \ff = G_t \dbar ^* \dbar \left (\iota _{X_t} ^*\left ( L^{1,0} _{\xi ^{\theta} _{\tau}}\alpha\right ) \right ),
\]
and the result follows from Lemma \ref{dbar-of-sff}. 
\end{proof}

\subsection{Direct proof of Corollary \ref{xu-wang-cor}}\label{direct-cor-pf-par}

We are now in a position to give the promised proof of Corollary \ref{xu-wang-cor} that does not pass through Theorem \ref{bo-2nd-thm-cor}.

The hypotheses of Corollary \ref{xu-wang-cor} provide a metric $h$ for $E$ such that $[\ii \Theta (h), \Lambda _{\omega}]$ is positive definite for some relative K\"ahler form $\omega$.  Fix $f _1,..., f_k \in \sH _t$ and $\tau _1,..., \tau _k \in \cO(T^{1,0}_B)_t$ along which we will compute the curvature.  Also fix a smooth horizontal distribution $\theta$.

By Corollary \ref{nak-pos-fam} $\sH$ is locally trivial, and hence for every $f \in \sH _t$ one can choose smooth sections $\ff_1,...,\ff_k$ of $\sH$ such that $\ff_i(t) = f_i$.  Consequently if $u_i$ are representatives of $\fa (\ff_i)$ then  $L^{1,0} _{\xi ^{\theta}_{\tau}} \dbar u_i \equiv 0$.

Writing $\two _{\tau} \ff (t) := \two ^{\sL ^{\theta}/\sH}(\tau) \ff (t)$, we have by Lemma \ref{dbar-of-sff} and Proposition \ref{hormander-hodge} 
\begin{eqnarray*}
&& \sum _{i,j=1} ^k\left (\two _{\tau _i} \ff _i(t) , \two _{\tau _j} \ff _j(t) \right ) = \left | \left | \sum _{i=1} ^k \two _{\tau _i} \ff _i(t) \right | \right |^2\\
&& \qquad =  \sup _{\gamma} \frac{1}{(\Box \gamma, \gamma) }\left | \left  ( \sum _{i=1} ^k \iota _{X_t}^* \left ( (\xi ^{\theta} _{\tau _i} \lrcorner \Theta (h))u_i + \nabla ^{1,0}((\dbar \xi ^{\theta} _{\tau_i}) \lrcorner u_i) \right ), \gamma\right )\right|^2\\ 
&& \qquad =   \sup _{\gamma} \frac{1}{(\Box \gamma, \gamma)}\left | \left  ( \sum _{i=1} ^k \left (\iota _{X_t}^*  (\xi ^{\theta} _{\tau _i} \lrcorner \Theta (h) )u_i, \gamma \right ) \right )  + \left ( \sum _{i=1} ^k \iota _{X_t} ^*(\dbar \xi ^{\theta} _{\tau_i}) \lrcorner u_i, \nabla ^{1,0*} \gamma\right )
\right|^2,
\end{eqnarray*}
where the supremum is taken over all smooth $E$-valued $(n,1)$-forms $\gamma$.  Now, for any $E$-valued $(n,1)$-form $U$ and any $E$-valued $(n-1,1)$-form $V$ 
\begin{eqnarray*}
|(U,\gamma) + (V , \nabla ^{1,0*}\gamma)|^2 &\le & |(U,\gamma)|^2 + |(V , \nabla ^{1,0*}\gamma)|^2 + 2 |(U,\gamma)|\cdot |(V , \nabla ^{1,0*}\gamma)|\\
&\le & ([\Theta (h_t), \Lambda _{\omega}]^{-1}U,U) ([\Theta (h_t), \Lambda _{\omega}]\gamma, \gamma) + ||V||^2|| \nabla ^{1,0*}\gamma||^2\\
&& \qquad  + 2 ([\Theta (h_t), \Lambda _{\omega}]^{-1}U,U)^{1/2} ([\Theta (h_t), \Lambda _{\omega}]\gamma, \gamma)^{1/2} ||V||\cdot ||\nabla ^{1,0*}\gamma||\\
&\le& (([\Theta (h_t), \Lambda _{\omega}]^{-1}U,U) + ||V||^2)(([\Theta (h_t), \Lambda _{\omega}]\gamma, \gamma) + ||\nabla ^{1,0*}\gamma||^2)\\
&=& (([\Theta (h_t), \Lambda _{\omega}]^{-1}U,U) + ||V||^2)(([\Theta (h_t), \Lambda _{\omega}] + \Box ^{1,0})\gamma, \gamma)\\
&=& (([\Theta (h_t), \Lambda _{\omega}]^{-1}U,U) + ||V||^2)(\Box\gamma, \gamma),
\end{eqnarray*}
where the first equality follows because $\Box ^{1,0} = \nabla ^{1,0} \nabla ^{1,0*}$ for $(n,q)$-forms, and the second equality is the Bochner-Kodaira Formula \eqref{bochner-kodaira-formula}.  Thus we have the estimate 
\begin{eqnarray*}
\sum _{i,j=1} ^k\left (\two _{\tau _i} \ff _i(t) , \two _{\tau _j} \ff _j(t) \right ) &\le&  \sup _{\gamma} \frac{1}{(\Box \gamma, \gamma)} \left | \left  ( \sum _{i=1} ^k \iota _{X_t}^* \left ( (\xi ^{\theta} _{\tau _i} \lrcorner \Theta (h))u_i + \nabla ^{1,0}((\dbar \xi ^{\theta} _{\tau_i}) \lrcorner u_i) \right ), \gamma\right )\right|^2\\ 
&=&   \left  ( \sum _{i,j=1} ^k [\Theta (h_t), \Lambda _{\omega}]^{-1}\iota _{X_t}^*  (\xi ^{\theta} _{\tau _i} \lrcorner \Theta (h) )u_i, \iota _{X_t}^*(\xi ^{\theta} _{\tau _j} \lrcorner \Theta (h) )u_j \right )  \\
&& +  \sum _{i,j=1} ^k \left ( \iota _{X_t} ^*(\dbar \xi ^{\theta} _{\tau_i}) \lrcorner u_i, \iota _{X_t} ^*(\dbar \xi ^{\theta} _{\tau_j}) \lrcorner u_j\right )
\end{eqnarray*}
Therefore by Theorem \ref{B2-thm}
\begin{eqnarray*}
&& \sum _{i,j=1} ^k ( \Theta ^{\sH} _{\tau_i \bar \tau_j} f_i, f _j) \\
&\ge&  \sum _{i,j=1} ^k \left ( \Theta (h) _{\xi^{\theta} _{\tau_i} \bar \xi^{\theta} _{\tau_j}} f _i, f_j\right ) - \left ( [\Theta (h_t), \Lambda _{\omega}]^{-1}\iota _{X_t}^*  (\xi ^{\theta} _{\tau _i} \lrcorner \Theta (h) )u_i, \iota _{X_t}^*(\xi ^{\theta} _{\tau _j} \lrcorner \Theta (h) )u_j \right )  \\
&& - \sum _{i,j=1} ^k  \left ( \int _{X_t}  \ii^{n^2} \left < (\dbar \xi^{\theta} _{\tau_i}) \lrcorner u_i \wedge \overline{ (\dbar \xi^{\theta} _{\tau_j}) \lrcorner u_j}, h \right >   +  \left ( \iota _{X_t} ^*(\dbar \xi ^{\theta} _{\tau_i}) \lrcorner u_i, \iota _{X_t} ^*(\dbar \xi ^{\theta} _{\tau_j}) \lrcorner u_j\right ) \right ).
\end{eqnarray*}
The estimate holds for all smooth horizontal distributions, but if we choose the horizontal lift of Proposition \ref{good-horizontal-distribution} then by Proposition \ref{Hodge-Riemann} the last line vanishes, and we have the estimate 
\[
\sum _{i,j=1} ^k ( \Theta ^{\sH} _{\tau_i \bar \tau_j} f_i, f _j) \ge  \sum _{i,j=1} ^k \left ( \Theta (h) _{\xi^{\theta} _{\tau_i} \bar \xi^{\theta} _{\tau_j}} f _i, f_j \right ) - \left (  [\Theta (h_t), \Lambda _{\omega}]^{-1}\iota _{X_t}^*  (\xi ^{\theta} _{\tau _i} \lrcorner \Theta (h) )u_i, \iota _{X_t}^*(\xi ^{\theta} _{\tau _j} \lrcorner \Theta (h) )u_j \right ).
\]
By Proposition \ref{positivity-of-dets} the right hand side is positive (resp. non-negative) as soon as $(E,h)$ is $k$-positively (resp. $k$-non-negatively) curved.  Thus Corollary \ref{xu-wang-cor} is proved.
\qed

\subsection{Berndtsson's Method of Representatives}

In this section we present an adaptation of a construction of Berndtsson that is fundamental to the proof of Theorem \ref{berndtsson-reps-thm}.  In the case applicable to \cite[Theorem 1.2]{bo-annals} Berndtsson integrated his method into the proof, and the crux of the argument is contained in \cite[Lemmas 4.3 and 4.4]{bo-annals}.  The technique was distilled further in \cite[Lemma 2.1]{bo-sns}.  In the case of \cite[Theorem 1.2]{bo-sns} Berndtsson augments the method in \cite[Section 4]{bo-sns} in a way that is hard to state in one or two sentences.  The next result captures these two ideas of Berndtsson in a single result.

\begin{thm}[Berndtsson's Method]\label{bo-rep-method}
Let $p :X \to B$ be a holomorphic family, let $E \to X$ be a holomorphic vector bundle with Hermitian metric $h$, fix $t \in B$ and let $\omega_t$ be a K\"ahler form on $X_t$.  Fix a horizontal lift $\theta$.  For any smooth section $f \in H^0(X, \sC^{\infty} (K_{X/B} \tensor E))$ such that $\iota _t ^* f \in \sH _t$ and any $\ell \in H^0(X, \sC^{\infty}(\Lambda ^{n,0} _X \tensor E))$ such that $\iota _{X_t} ^*\ell  = 0$ and $\iota _{X_t} ^*(\xi ^{\theta} _{\tau} \lrcorner  \nabla ^{1,0} \ell )\perp \sH _t$  for all germs of holomorphic vector fields $\tau \in \cO(T^{1,0} _B)_t \cong T^{1,0} _{B,t}$ there there exists a representative $u \in H^0(X, \sC^{\infty}(\Lambda ^{n,0} _X\tensor E))$ of $f$ (in the sense of Paragraph \ref{n0-forms-description}) with the following properties.
\begin{enumerate}
\item[{\rm a.}]  The $E$-valued $(n-1,1)$-forms $\iota _{X_t} ^* (\xi ^{\theta} _{\tau} \lrcorner \dbar u)$ are primitive.
\item[{\rm b.}] $P_t ^{\perp} \iota _{X_t} ^* (\xi ^{\theta} _{\tau} \lrcorner \nabla ^{1,0} u) = \iota _{X_t} ^*(\xi ^{\theta} _{\tau} \lrcorner  \nabla ^{1,0} \ell)$.
\item[{\rm c.}] The $E$-valued $(n-1,1)$-forms $\iota _{X_t} ^* (\xi ^{\theta} _{\tau} \lrcorner \dbar u)$ and $\iota _{X_t} ^* ((\dbar \xi ^{\theta} _{\tau}) \lrcorner f)$ are in the same Dolbeault class.
\end{enumerate}
Moreover, if $\iota _{X_t} ^* L^{1,0} _{\xi ^{\theta} _{\tau}}\dbar f = 0$ then $\iota _{X_t} ^* L^{1,0} _{\xi ^{\theta} _{\tau}} \dbar u = 0$.
\end{thm}

\begin{proof}
By Proposition \ref{int-ptwise-holo} $\iota _{X_t} ^* L^{1,0} _{\xi ^{\theta} _{\tau}} \dbar u$ is independent of the choice of representative $u$ of $f$, so it suffices to find a representative $u$ with properties a -- c.

Let $u^o \in H^0(X, \sC^{\infty}(\Lambda ^{n,0} _X\tensor E))$ be some initial representative of $f$.  Using local coordinates $(t^1,...,t^m)$ near $t$, we can write 
\[
\dbar u^o = \psi + dt ^{\mu} \wedge \alpha^o _{\mu} + d\bar t^{\nu} \wedge \beta _{\bar \nu} \quad \text{and} \quad \nabla ^{1,0} u^o = dt ^{\mu} \wedge \phi ^o_{\mu},
\]
where $\alpha ^o _{\mu}$ are $E$-valued $(n-1,1)$-forms, $\beta _{\bar \nu}$ are $E$-valued $(n,0)$-forms, and $\psi$ is an $E$-valued $(n,1)$-form such that $\iota_{X_t} ^*\psi = 0$ at $t$.  (See Proposition \ref{rep-of-holo}.)  The latter condition implies that $\iota ^* _{X_t} (\xi ^{\theta} _{\tau} \lrcorner \psi)$ independent of the choice of lift $\theta$.  Since the horizontal lift $\xi ^{\theta} _{\tau}$ is linear in $\tau$, we may write 
\[
\iota_{X_t} ^*(\xi ^{\theta} _{\tau} \lrcorner \psi) =: \tau ^{\mu} \gamma _{\mu}
\]
for some $E$-valued $(n-1,1)$-forms $\gamma _1,..., \gamma _m$ on $X_t$.

Any other representative $u$ is of the form $u= u^o + dt ^{\mu} \wedge v_{\mu}$ for some $E$-valued $(n-1,0)$-forms $v_{\mu}$, and thus 
\[
\dbar u = \psi + dt ^{\mu} \wedge (\alpha^o _{\mu} - \dbar v_{\mu}) + d\bar t^{\nu} \wedge \beta _{\bar \nu} \quad \text{and} \quad \nabla ^{1,0} u = dt ^{\mu} \wedge ( \phi ^o_{\mu} - \nabla ^{1,0} v_{\mu}).
\]
Therefore we seek $v_{\mu}$ satisfying 
\[
\omega_t \wedge ( \gamma _{\mu} +  \iota ^* _{X_t} \alpha ^o _{\mu} - \iota _{X_t}^* \dbar v_{\mu} )= 0 \quad \text{ and } \quad \nabla ^{1,0} v_{\mu} = P_t ^{\perp} \iota _{X_t}^* \phi ^o _{\mu}- \iota _{X_t} ^*(\xi ^{\theta} _{\tau} \lrcorner  \nabla ^{1,0} \ell ).
\]
Note that 
\[
\dbar ( \omega \wedge u^o ) = \omega \wedge \dbar u^o = \omega \wedge \psi  + dt ^{\mu} \wedge \alpha ^o _{\mu} \wedge \omega + d\bar t ^{\nu} \wedge \beta _{\bar \nu} \wedge \omega,
\]
and hence 
\[
\iota _{X_t} ^*\left ( \xi ^{\theta} _{\tau} \lrcorner \dbar (\omega \wedge u^o)\right ) = \iota _{X_t} ^*( (\xi ^{\theta} _{\tau} \lrcorner \omega) \wedge \psi )  + \omega _t \wedge (\iota _{X_t} ^* (\xi^{\theta}_{\tau}\lrcorner \psi ) + \tau ^{\mu} \iota _{X_t} ^* \alpha ^o _{\mu}) =  \tau ^{\mu}  \omega _t \wedge (\gamma _{\mu} +\iota _{X_t} ^* \alpha ^o _{\mu}).
\]
On the other hand 
\[
\iota _{X_t} ^* \left (\xi ^{\theta} _{\tau} \lrcorner \dbar (\omega \wedge u^o) \right ) = \iota _{X_t} ^* \left ( (\dbar \xi ^{\theta} _{\tau}) \lrcorner (\omega \wedge u^o)  - \dbar ( \xi ^{\theta} _{\tau} \lrcorner (\omega \wedge u^o))\right ) =  - \iota _{X_t} ^* \left (\dbar ( \xi ^{\theta} _{\tau} \lrcorner (\omega \wedge u^o))\right ),
\]
where the second equality holds because $\dbar \xi ^{\theta} _{\tau}$ is vertical and $\omega \wedge u^o$, being of bidegree $(n+1,1)$, is of the form $\omega \wedge u^o = dt ^{\mu} \wedge c_{\mu}$, and hence vanishes along fibers.  Moreover $\xi ^{\theta} _{\tau} \lrcorner (\omega \wedge u^o) = \tau ^{\mu} c_{\mu} + dt^{\mu} \wedge (\xi ^{\theta} _{\tau} \lrcorner c_{\mu})$, so that $\dbar ( \xi ^{\theta} _{\tau} \lrcorner (\omega \wedge u^o)) = \tau ^{\mu} \dbar c_{\mu} + dt^{\mu} \wedge \dbar (\xi ^{\theta} _{\tau} \lrcorner c_{\mu})$.  Restricting to $X_t$, we see that $\omega _t \wedge ( \gamma _{\mu} + \iota _{X_t} ^* \alpha ^o _{\mu})$ is $\dbar$-exact on $X_t$.  Therefore the $E|_{X_t}$-valued $(n-1,0)$-form 
\[
V^1 _{\mu} := \Lambda _{\omega_t} \dbar ^* G_t \left (\omega _t \wedge ( \gamma _{\mu} + \iota _{X_t} ^* \alpha ^o _{\mu}) \right )
\]
satisfies 
\begin{eqnarray*}
\omega _t \wedge (\dbar V^1_{\mu} - (\gamma _{\mu}+  \iota _{X_t} ^* \alpha ^o _{\mu})) &=& \dbar (\omega _t \wedge  V^1_{\mu})  - \omega _t \wedge (\gamma _{\mu} +  \iota _{X_t} ^* \alpha ^o _{\mu}) \\
&=& \dbar \dbar ^* G_t (\omega _t \wedge (\gamma _{\mu} + \iota _{X_t} ^* \alpha ^o)) - \omega _t \wedge (\gamma _{\mu} + \iota _{X_t} ^* \alpha ^o )= 0.
\end{eqnarray*}
The second equality follows because $[L_{\omega_t}, \Lambda _{\omega_t}]= L_{\omega_t} \Lambda _{\omega_t} = {\rm Id}$ on ($E$-valued) $(n,1)$-forms, and in the last equality we have used that $\Box = \dbar \dbar ^*$ on $\dbar$-closed forms.  We note that, by the Hodge identity $\ii [\nabla ^{1,0} , \Lambda _{\omega_t}] = \dbar ^*$ , 
\[
\nabla ^{1,0} V^1 _{\mu} = \nabla ^{1,0} \Lambda _{\omega_t} \dbar ^* G_t (\omega _t \wedge \iota _{X_t} ^* \alpha ^o _{\mu})= [\nabla ^{1,0}, \Lambda _{\omega_t}] \dbar ^* G_t (\omega _t \wedge \iota _{X_t} ^* \alpha ^o _{\mu}) = 0.
\]

Next, $P_t ^{\perp} \iota _{X_t}^* \phi ^o _{\mu}$ is trivially orthogonal to the holomorphic $E$-valued $n$-forms, which are precisely the harmonic $E$-valued $(n,0)$-forms.  By assumption, for every $\mu \in \{1,..., m\}$ the twisted form $\iota _{X_t} ^*(\xi ^{\theta} _{\tfrac{\di}{\di t ^{\mu}}} \lrcorner  \nabla ^{1,0} \ell)$ is orthogonal to the holomorphic $E$-valued $n$-forms.  Since on $E$-valued $(n,0)$-forms $\Box  = \Box ' = \nabla ^{1,0} \nabla ^{1,0*}$ and hence $G'= G$, the $E$-valued $(n-1,0)$-form
\[
V^2_{\mu} := G_t ' \nabla ^{1,0*} \left ( P_t ^{\perp} \iota _{X_t}^* \phi ^o _{\mu} - \iota _{X_t} ^*(\xi ^{\theta} _{\tfrac{\di}{\di t ^{\mu}}} \lrcorner  \nabla ^{1,0} \ell)\right ) =  \nabla ^{1,0*}  G_t '\left (P_t ^{\perp} \iota _{X_t}^* \phi ^o _{\mu} - \iota _{X_t} ^*(\xi ^{\theta} _{\tfrac{\di}{\di t ^{\mu}}} \lrcorner  \nabla ^{1,0} \ell)\right )
\]
solves the equation 
\[
\nabla ^{1,0} V^2_{\mu} =  P_t ^{\perp} \iota _{X_t}^* \phi ^o _{\mu} -\iota _{X_t} ^*(\xi ^{\theta} _{\tfrac{\di}{\di t ^{\mu}}} \lrcorner  \nabla ^{1,0} \ell)
\]
and satisfies 
\begin{eqnarray*}
\omega _t \wedge \dbar V^2_{\mu} &=& \dbar L_{\omega_t} \nabla ^{1,0*} G_t \left (P_t ^{\perp} \iota _{X_t}^* \phi ^o _{\mu} - \iota _{X_t} ^*(\xi ^{\theta} _{\tfrac{\di}{\di t ^{\mu}}} \lrcorner  \nabla ^{1,0} \ell)\right )\\
&=& \dbar [L_{\omega_t} ,\nabla ^{1,0*}] G_t \left (P_t ^{\perp} \iota _{X_t}^* \phi ^o _{\mu} - \iota _{X_t} ^*(\xi ^{\theta} _{\tfrac{\di}{\di t ^{\mu}}} \lrcorner  \nabla ^{1,0} \ell)\right )\\
&=& - \ii \dbar \dbar G_t \left (P_t ^{\perp} \iota _{X_t}^* \phi ^o _{\mu} - \iota _{X_t}^*(\xi ^{\theta} _{\tfrac{\di}{\di t ^{\mu}}} \lrcorner  \nabla ^{1,0} \ell)\right )\\
&=& 0.
\end{eqnarray*}
Taking $v_{\mu}$ to be any smooth extension of $V_{\mu} := V^1_{\mu} + V^2_{\mu}$ to $X$, we find that the representative $u = u^o + dt ^{\mu} \wedge v_{\mu}$ satisfies Properties a and b.  Finally, for \emph{any} representative $u$ of $f$ one has 
\begin{equation}\label{coh-derivative}
\dbar (\xi^{\theta} _{\tau} \lrcorner u) = (\dbar \xi^{\theta} _{\tau}) \lrcorner u - \xi^{\theta} _{\tau} \lrcorner \dbar u,
\end{equation}
so $(\dbar \xi^{\theta} _{\tau}) \lrcorner u$ and $\xi^{\theta} _{\tau} \lrcorner \dbar u$ are $\dbar$-cohomologous (on fibers).  Thus Property c holds as well.  
\end{proof}

\subsection{The Proof of Theorem \ref{berndtsson-reps-thm}}

Let us fix representatives $u_i$ of $\fa (\ff_i)$, $i=1,2$, as in Theorem \ref{bo-rep-method} with $\ell = 0$.  That is to say, $u_1$ and $u_2$ have the following properties.
\begin{enumerate}
\item[a.]The $E|_{X_t}$-valued $(n-1,1)$-forms $\iota _{X_t} ^* (\xi ^{\theta} _{\tau} \lrcorner \dbar u_i)$ are primitive.
\item[b.] The $E|_{X_t}$-valued $(n,0)$-forms $\iota _{X_t} ^* (\xi ^{\theta} _{\tau} \lrcorner \nabla ^{1,0} u_i)$ are holomorphic.
\end{enumerate}
Moreover, by an application of Proposition \ref{holo-infinitesimal} we may assume that $\iota _{X_t} ^* L^{1,0} _{\xi ^{\theta} _{\tau}} \dbar u _i= 0$.  Then 
\begin{eqnarray*}
\two ^{\sL ^{\theta}/\sH} (\sigma) \ff _1 &=& P_t ^{\perp} \iota _{X_t} ^*  L^{1,0} _{\xi ^{\theta} _{\sigma}} u_1= P_t ^{\perp}  \iota _{X_t} ^* (\nabla ^{1,0}(\xi ^{\theta} _{\sigma}\lrcorner u_1) + \xi ^{\theta} _{\sigma}\lrcorner  (\nabla ^{1,0}u_1))\\
&=& P_t ^{\perp}  \nabla ^{1,0}( \iota _{X_t} ^* (\xi ^{\theta} _{\sigma}\lrcorner u_1)) = \nabla ^{1,0}( \iota _{X_t} ^* (\xi ^{\theta} _{\sigma}\lrcorner u_1)).
\end{eqnarray*}
The third equality follows from Property b, while the last equality holds because the image of $\nabla ^{1,0}$ in $\Gamma (X, \sC^{\infty} (\Lambda ^{n,0} _X \tensor E))$ is orthogonal to the kernel of $\Box^{1,0} = \Box$, and the latter kernel consists precisely of holomorphic forms.  Similarly $\two ^{\sL ^{\theta}/\sH} (\tau) \ff _2 = \nabla ^{1,0}( \iota _{X_t} ^* (\xi ^{\theta} _{\tau}\lrcorner u_2))$.  Therefore 
\begin{eqnarray*}
\left ( \two ^{\sL ^{\theta}/\sH} (\sigma) \ff _1, \two ^{\sL ^{\theta}/\sH} (\tau) \ff _2 \right )&=& \ii ^{n^2}\int _{X_t} \left < \nabla ^{1,0}( \iota _{X_t} ^* (\xi ^{\theta} _{\sigma}\lrcorner u_1)) \wedge \overline{\nabla ^{1,0}( \iota _{X_t} ^* (\xi ^{\theta} _{\tau}\lrcorner u_2))} , h \right >\\
&=& (\text{\rm -1})^{n+1} \ii ^{n^2} \int _{X_t} \left < \dbar \nabla ^{1,0} ( \iota _{X_t} ^* (\xi ^{\theta} _{\sigma}\lrcorner u_1)) \wedge \overline{\iota _{X_t} ^* (\xi ^{\theta} _{\tau}\lrcorner u_2)}, h \right >\\
&=& \ii ^{n^2-2n +1} \int _{X_t} \left < (\ii \Theta (h_t)  \iota _{X_t} ^* (\xi ^{\theta} _{\sigma}\lrcorner u_1)) \wedge \overline{\iota _{X_t} ^* (\xi ^{\theta} _{\tau}\lrcorner u_2)}, h \right >\\
&& \qquad  -  (\text{\rm -1})^{n+1} \ii ^{n^2} \int _{X_t} \left < \nabla ^{1,0} \dbar  ( \iota _{X_t} ^* (\xi ^{\theta} _{\sigma}\lrcorner u_1)) \wedge \overline{\iota _{X_t} ^* (\xi ^{\theta} _{\tau}\lrcorner u_2)}, h \right >\\
&=& \ii ^{(n-1)^2} \int _{X_t} \left < (\ii  \Theta (h_t) (\xi ^{\theta} _{\sigma}\lrcorner u_1)) \wedge \overline{(\xi ^{\theta} _{\tau}\lrcorner u_2)} , h \right >\\
&& \qquad  - \ii ^{n^2} \int _{X_t} \left <(\dbar (\xi ^{\theta} _{\sigma}\lrcorner u_1)) \wedge \overline{\dbar (\xi ^{\theta} _{\tau}\lrcorner u_2)}, h \right >.
\end{eqnarray*}
Using \eqref{l-lambda-comm-rel} we have 
\begin{eqnarray*}
\ii \Theta (h_t) \iota _{X_t} ^* (\xi ^{\theta} _{\sigma}\lrcorner u_1) &=& [L_{\omega_t} , \Lambda _{\omega_t}] \ii \Theta (h_t) \iota _{X_t} ^*(\xi ^{\theta} _{\sigma}\lrcorner u_1)=  L_{\omega_t}  \Lambda _{\omega_t} \ii \Theta (h_t) \iota _{X_t} ^*(\xi ^{\theta} _{\sigma}\lrcorner u_1))\\
&=& - L_{\omega_t} [\ii \Theta (h_t),  \Lambda _{\omega_t}] \iota _{X_t} ^* (\xi ^{\theta} _{\sigma}\lrcorner u_1),
\end{eqnarray*}
and thus 
\begin{eqnarray*}
\ii ^{(n-1)^2} \int _{X_t} \left < (\ii  \Theta (h_t) (\xi ^{\theta} _{\sigma}\lrcorner u_1)) \wedge \overline{(\xi ^{\theta} _{\tau}\lrcorner u_2)} , h \right > = - \left ( \left [\ii \Theta (h_t),  \Lambda _{\omega_t}\right ] \xi ^{\theta} _{\sigma}\lrcorner u_1, \xi ^{\theta} _{\tau}\lrcorner u_2 \right ).
\end{eqnarray*}

From the differential identity $\dbar (\xi \lrcorner u) = (\dbar \xi) \lrcorner u - \xi \lrcorner \dbar u$ and the algebraic identity
\[
(a-b , a-b) = (a,a) - (b,b) - (a-b, b) - (b, a-b)
\]
for a bilinear form $(\cdot, \cdot )$ we obtain 
\begin{eqnarray*}
&& \int _{X_t}  \left < (\dbar (\xi ^{\theta} _{\sigma}\lrcorner u_1)) \wedge \overline{\dbar (\xi ^{\theta} _{\tau}\lrcorner u_2)}, h \right > \\
&=&  \int _{X_t} \left < ((\dbar \xi ^{\theta} _{\sigma})\lrcorner u_1) \wedge \overline{((\dbar \xi ^{\theta}_{\tau})\lrcorner u_2) }, h \right > -  \int _{X_t} \left < \xi ^{\theta} _{\sigma}\lrcorner (\dbar u_1) ) \wedge \overline{\xi ^{\theta} _{\tau} \lrcorner (\dbar u_2))}, h \right >\\
&&  -  \int _{X_t}\left <  (\dbar( \xi ^{\theta} _{\sigma}\lrcorner u_1)) \wedge \overline{(\xi ^{\theta} _{\tau} \lrcorner (\dbar u_2))} , h \right > - \int _{X_t} \left < (\xi ^{\theta} _{\sigma}\lrcorner (\dbar u_1) ) \wedge \overline{(\dbar(\xi ^{\theta} _{\tau}\lrcorner u_2)}, h \right >.
\end{eqnarray*}
Now, 
\begin{eqnarray*}
&& \int _{X_t} \left < (\xi ^{\theta} _{\sigma}\lrcorner (\dbar u_1) ) \wedge \overline{(\dbar(\xi ^{\theta} _{\tau}\lrcorner u_2)}, h \right > = (-1) ^{n-1} \int _{X_t} \left < \nabla ^{1,0} (\xi ^{\theta} _{\sigma} \lrcorner (\dbar u_1)) \wedge \overline{\xi ^{\theta} _{\tau} \lrcorner u_2}, h \right > \\
&=&  (-1) ^{n} \int _{X_t} \left < (\xi ^{\theta} _{\sigma} \lrcorner (\nabla ^{1,0} \dbar u_1)) \wedge \overline{\xi ^{\theta} _{\tau} \lrcorner u_2}, h \right > - (-1) ^{n} \int _{X_t} \left < (L^{1,0} _{\xi ^{\theta} _{\sigma}} \dbar u_1) \wedge \overline{\xi ^{\theta} _{\tau} \lrcorner u_2}, h \right >\\
&=&  (-1) ^{n} \int _{X_t} \left < (\xi ^{\theta} _{\sigma} \lrcorner (\nabla ^{1,0} \dbar u_1)) \wedge \overline{\xi ^{\theta} _{\tau} \lrcorner u_2}, h \right >.
\end{eqnarray*}
We claim that $\iota _{X_t} ^* (\xi ^{\theta} _{\sigma} \lrcorner (\nabla ^{1,0} \dbar u_1)) = \iota _{X_t} ^* (\xi ^{\theta} _{\sigma} \lrcorner (\Theta(h) u_1))$.  Indeed, since $\nabla ^{1,0} u_1= dt ^{\mu} \wedge \phi _{\mu}$ with $\phi _{ \mu}$ holomorphic on $X_t$, 
\begin{eqnarray*}
\xi ^{\theta} _{\sigma} \lrcorner (\dbar \nabla ^{1,0} u_1) &=& (\dbar \xi ^{\theta} _{\sigma}) \lrcorner \nabla ^{1,0} u_1 - \dbar ( \xi ^{\theta} _{\sigma} \lrcorner \nabla ^{1,0} u_1)\\
&=& dt ^{\mu} \wedge ((\dbar \xi^{\theta} _{\sigma}) \lrcorner \phi _{\mu}) - \dbar (\sigma ^{\mu} \phi _{\mu} - dt ^{\mu} \wedge (\xi ^{\theta} _{\sigma} \lrcorner \phi _{\mu}))\\
&=& - \sigma ^{\mu} \dbar \phi _{\mu} + dt ^{\mu} \wedge \left ((\dbar \xi ^{\theta} _{\sigma}) \lrcorner \phi _{\mu} - \dbar (\xi ^{\theta} _{\sigma} \lrcorner \phi _{\mu})  \right ).
\end{eqnarray*}
Both of these terms vanish on $X_t$, which proves our claim.  Therefore 
\[
\int _{X_t} \left < (\xi ^{\theta} _{\sigma}\lrcorner (\dbar u_1) ) \wedge \overline{(\dbar(\xi ^{\theta} _{\tau}\lrcorner u_2)}, h \right > =  (-1) ^{n} \int _{X_t} \left < (\xi ^{\theta} _{\sigma} \lrcorner (\Theta (h)u_1)) \wedge \overline{\xi ^{\theta} _{\tau} \lrcorner u_2}, h \right > .
\]
By interchanging $\sigma \leftrightarrow \tau$ and $u_1 \leftrightarrow u_2$ and then taking complex conjugates, 
\[
\int _{X_t}\left <  (\dbar( \xi ^{\theta} _{\sigma}\lrcorner u_1)) \wedge \overline{(\xi ^{\theta} _{\tau} \lrcorner (\dbar u_2))} , h \right > =  (-1) ^{n} \int _{X_t} \left < (\xi ^{\theta} _{\sigma} \lrcorner u_1) \wedge \overline{\xi ^{\theta} _{\tau} \lrcorner(\Theta (h) u_2)}, h \right > 
\]
Using \eqref{l-lambda-comm-rel} again, we have 
\begin{eqnarray*}
&& (-1) ^n \int _{X_t} \ii ^{n^2} \left < \left (\xi ^{\theta} _{\sigma} \lrcorner (\Theta (h)u_1)\right ) \wedge \overline{\xi ^{\theta} _{\tau} \lrcorner u_2}, h \right > \\
&=& - \ii ^{n^2 - 2n +1} \int _{X_t}\left < \omega _t \wedge  \Lambda _{\omega _t} \iota _{X_t} ^* \left (\xi ^{\theta} _{\sigma} \lrcorner  (\ii \Theta (h)u_1)\right ) \wedge \overline{\xi ^{\theta} _{\tau} \lrcorner u_2}, h \right > \\
&=& - ( \Lambda_{\omega _t} (\xi ^{\theta} _{\sigma} \lrcorner ( \ii \Theta (h)u_1)) ,\xi ^{\theta} _{\tau} \lrcorner u_2)\\
&=& - \left ( \Lambda_{\omega _t} \iota _{X_t} ^* \left (\xi ^{\theta} _{\sigma} \lrcorner  \ii \Theta (h)\right ) u_1 ,\xi ^{\theta} _{\tau} \lrcorner u_2\right )- \left ( \Lambda_{\omega _t}  \ii \Theta (h_t) (\xi ^{\theta} _{\sigma} \lrcorner u_1) ,\xi ^{\theta} _{\tau} \lrcorner u_2\right )\\
&=& - \left ( \Lambda_{\omega _t}  \iota _{X_t} ^* \left (\xi ^{\theta} _{\sigma} \lrcorner  \ii \Theta (h)\right ) u_1 ,\xi ^{\theta} _{\tau} \lrcorner u_2\right ) +  \left ( \left [\ii \Theta (h_t),\Lambda_{\omega _t}\right ] (\xi ^{\theta} _{\sigma} \lrcorner u_1) ,\xi ^{\theta} _{\tau} \lrcorner u_2\right ),
\end{eqnarray*}
Since $[\ii \Theta (h_t),\Lambda_{\omega _t}] = \Box - \Box ^{1,0}$ is Hermitian symmetric for the $L^2$-norm, we find that 
\begin{eqnarray*}
&& \ii ^{n^2} \int _{X_t} \left < (\xi ^{\theta} _{\sigma}\lrcorner (\dbar u_1) ) \wedge \overline{(\dbar(\xi ^{\theta} _{\tau}\lrcorner u_2)}, h \right > + \ii ^{n^2}  \int _{X_t}\left <  (\dbar( \xi ^{\theta} _{\sigma}\lrcorner u_1)) \wedge \overline{(\xi ^{\theta} _{\tau} \lrcorner (\dbar u_2))} , h \right > \\
&& \qquad =  2 \left ( \left [\ii \Theta (h_t),\Lambda_{\omega _t}\right ] (\xi ^{\theta} _{\sigma} \lrcorner u_1) ,\xi ^{\theta} _{\tau} \lrcorner u_2\right )\\
&& \quad \qquad -  \left ( \Lambda_{\omega _t}  \iota _{X_t} ^* \left (\xi ^{\theta} _{\sigma} \lrcorner  \ii \Theta (h)\right ) u_1 ,\xi ^{\theta} _{\tau} \lrcorner u_2\right ) -  \left ( \xi ^{\theta} _{\sigma} \lrcorner u_1, \Lambda_{\omega _t}  \iota _{X_t} ^* \left (\xi ^{\theta} _{\tau} \lrcorner  \ii \Theta (h)\right ) u_2 \right ).
\end{eqnarray*}
Adding everything up yields 
\begin{eqnarray*}
&& \left ( \two ^{\sL ^{\theta}/\sH} (\sigma) \ff _1, \two ^{\sL ^{\theta}/\sH} (\tau) \ff _2 \right ) \\
&=& -  \int _{X_t} \ii ^{n^2} \left < \left ((\dbar \xi ^{\theta} _{\sigma})\lrcorner u_1\right ) \wedge \overline{\left ((\dbar \xi ^{\theta}_{\tau})\lrcorner u_2\right ) }, h \right > - \left (\xi ^{\theta} _{\sigma}\lrcorner (\dbar u_1) , \xi ^{\theta} _{\tau} \lrcorner (\dbar u_2)\right )\\
&& \qquad + ( [\ii \Theta (h_t),\Lambda_{\omega _t}] (\xi ^{\theta} _{\sigma} \lrcorner u_1) ,\xi ^{\theta} _{\tau} \lrcorner u_2) \\
&& \quad \qquad - ( [\Lambda_{\omega _t} , \iota _{X_t} ^* ((\xi ^{\theta} _{\sigma} \lrcorner  \ii \Theta (h))] u_1 ,\xi ^{\theta} _{\tau} \lrcorner u_2) - ( \xi ^{\theta} _{\sigma} \lrcorner u_1, [\Lambda_{\omega _t} , \iota _{X_t} ^* ((\xi ^{\theta} _{\tau} \lrcorner  \ii \Theta (h))] u_2 ).
\end{eqnarray*}
Thus Theorem \ref{berndtsson-reps-thm} is proved.
\qed

\subsection{The proof of Theorem \ref{bo-2nd-thm-cor}}

We begin with the following computation.

\begin{prop}\label{expand-the-current-prop}
Let $\ff_1, \ff_2 \in \Gamma (B,\sC^{\infty}(\sL))$ be represented by $E$-valued $(n,0)$-forms $u_1$ and $u_2$ respectively.  Then the differential $(1,1)$-form $\alpha_h(\ff_1,\ff_2)$ on $B$ defined by 
\[
\alpha _h(\ff _1, \ff_2) := p_* \left (\ii ^{n^2} \left < (\Theta (h) u_1) \wedge \bar u_2 , h \right >\right )
\]
is given by the formula 
\begin{eqnarray*}
\left < \alpha _h(\ff _1, \ff_2),  \sigma \wedge \bar \tau \right > &=& \left ( \Theta (h) _{\xi^{\theta} _{\sigma} \bar \xi^{\theta} _{\tau}} f _1, f_2\right ) - \left (\left [\ii \Theta (h_t), \Lambda_{\omega_t}\right ] (\xi ^{\theta} _{\sigma}\lrcorner  u_1), (\xi ^{\theta} _{\tau}\lrcorner u_2)\right )  \\
&& + \left (\Lambda _{\omega _t}\iota _{X_t}^*\left ( \xi ^{\theta}_{\sigma} \lrcorner \ii \Theta (h) \right ) f_1, \xi ^{\theta}_{\tau}\lrcorner u_2\right ) + \left ( \xi ^{\theta}_{\sigma}\lrcorner u_1, \Lambda _{\omega _t} \iota _{X_t}^*\left (\xi ^{\theta}_{\tau} \lrcorner \ii \Theta (h)   \right ) f_2 \right ) \end{eqnarray*}
for any choice of horizontal distribution $\theta$.
\end{prop}

\begin{proof}
Using a partition of unity and Fubini's Theorem, one sees that 
\[
\left < \alpha _h(\ff _1, \ff_2),  \sigma \wedge \bar \tau \right > = \int _{X_t} \left < \ii ^2 \bar \xi ^{\theta} _{\tau} \lrcorner \left ( \xi ^{\theta} _{\sigma} \lrcorner ((\Theta (h) u_1)\wedge \bar u_2 ) \right ), h \right >
\]
for any horizontal distribution $\theta$.  We compute that 
\begin{eqnarray*}
\xi ^{\theta} _{\sigma} \lrcorner ((\Theta (h) u_1)\wedge \bar u_2 ) &=& (( \xi ^{\theta} _{\sigma} \lrcorner \Theta (h)) u_1\wedge \bar u_2 ) + (( \Theta (h) (\xi ^{\theta} _{\sigma} \lrcorner u_1))\wedge \bar u_2 )\\
&=& ((\xi ^{\theta} _{\sigma} \lrcorner \Theta (h)) u_1\wedge \bar u_2 ) - ((\xi ^{\theta} _{\sigma} \lrcorner u_1))\wedge\overline{ \Theta (h)  u_2}),
\end{eqnarray*}
and then 
\begin{eqnarray*}
&& \bar \xi ^{\theta} _{\tau} \lrcorner (\xi ^{\theta} _{\sigma} \lrcorner (( \Theta (h) u_1)\wedge \bar u_2 )) \\
&&= (( \Theta (h)_{\xi ^{\theta} _{\sigma} \bar \xi ^{\theta} _{\tau}} u_1) \wedge \bar u_2 ) +(-1) ^{n-1} ((\xi ^{\theta} _{\sigma} \lrcorner  \Theta (h)) u_1\wedge \overline{\xi ^{\theta} _{\tau} \lrcorner  u_2} )\\
&& + (-1) ^{n}((\xi ^{\theta} _{\sigma} \lrcorner u_1))\wedge\overline{(\xi ^{\theta} _{\tau} \lrcorner  \Theta (h))  u_2}) + (-1) ^{n} ((\xi ^{\theta} _{\sigma} \lrcorner u_1))\wedge\overline{ \Theta (h) (\xi^{\theta} _{\tau} \lrcorner u_2)})\\
&&= (( \Theta (h)_{\xi ^{\theta} _{\sigma} \bar \xi ^{\theta} _{\tau}} u_1) \wedge \bar u_2 ) +(-1) ^{n-1} ((\xi ^{\theta} _{\sigma} \lrcorner  \Theta (h)) u_1\wedge \overline{\xi ^{\theta} _{\tau} \lrcorner  u_2} )\\
&& + (-1) ^{n}(\xi ^{\theta} _{\sigma} \lrcorner u_1)\wedge\overline{(\xi ^{\theta} _{\tau} \lrcorner  \Theta (h))  u_2}) + (-1) ^{n-1} ( \Theta (h) (\xi ^{\theta} _{\sigma} \lrcorner u_1))\wedge\overline{(\xi^{\theta} _{\tau} \lrcorner u_2)}.
\end{eqnarray*} 
But by \eqref{l-lambda-comm-rel}
\[
\iota _{X_t} ^* ((\xi ^{\theta} _{\sigma} \lrcorner  \Theta (h)) u_1 )= [L_{\omega_t} , \Lambda _{\omega_t}]\iota _{X_t} ^* ((\xi ^{\theta} _{\sigma} \lrcorner  \Theta (h)) u_1 ) = L_{\omega_t}  \Lambda _{\omega_t} \iota _{X_t} ^* ((\xi ^{\theta} _{\sigma} \lrcorner  \Theta (h)) u_1 ),
\]
so 
\begin{eqnarray*}
&& \ii ^{n^2}\left <(-1) ^{n-1} ((\xi ^{\theta} _{\sigma} \lrcorner  \Theta (h)) u_1\wedge \overline{\xi ^{\theta} _{\tau} \lrcorner  u_2} ), h \right >\\
&&= \left < \ii ^{(n-1)^2}  ( \omega _t \wedge \Lambda _{\omega _t} (\xi ^{\theta} _{\sigma} \lrcorner  \ii \Theta (h)) u_1\wedge \overline{\xi ^{\theta} _{\tau} \lrcorner  u_2} ) , h \right > = \left \{ \Lambda _{\omega _t} (\xi ^{\theta} _{\sigma} \lrcorner  \ii \Theta (h)) u_1,\xi ^{\theta} _{\tau} \lrcorner  u_2\right \} \frac{\omega _t ^n}{n!},
\end{eqnarray*}
where $\{ \cdot , \cdot \}$ is the pointwise inner product, on $X_t$, of $E$-valued $(n-1,0)$-forms induced by $h_t$ and $\omega _t$. 
Similarly 
\[
(-1) ^n \left < \ii ^{n^2} (\xi ^{\theta} _{\sigma} \lrcorner u_1)\wedge\overline{(\xi ^{\theta} _{\tau} \lrcorner  \Theta (h))  u_2}), h\right > = \left \{ \xi ^{\theta} _{\sigma} \lrcorner u_1, \Lambda _{\omega _t}(\xi ^{\theta} _{\tau} \lrcorner  \ii \Theta (h))  u_2 \right \} \frac{\omega_t ^n}{n!}
\]
and 
\begin{eqnarray*}
(-1) ^{n-1} \left < \ii ^{n^2}( \Theta (h) (\xi ^{\theta} _{\sigma} \lrcorner u_1))\wedge\overline{(\xi^{\theta} _{\tau} \lrcorner u_2)}, h \right >&=& \left \{ \Lambda _{\omega _t} \ii \Theta (h) (\xi ^{\theta} _{\sigma} \lrcorner u_1), \xi^{\theta} _{\tau} \lrcorner u_2\right \}\frac{\omega ^n}{n!}\\
&=& -  \left \{\left  [\ii \Theta (h), \Lambda _{\omega _t} \right ] (\xi ^{\theta} _{\sigma} \lrcorner u_1), \xi^{\theta} _{\tau} \lrcorner u_2\right \}\frac{\omega ^n}{n!}.
\end{eqnarray*}
Integration over $X_t$ completes the proof.
\end{proof}

\begin{proof}[Conclusion of the proof of Theorem \ref{bo-2nd-thm-cor}]
Fix a smooth horizontal distribution $\theta$ and a point $t \in B$.  Let $\ff _1 , \ff _2 \in \Gamma (B, \sC^{\infty} (\sL))$ be smooth sections such that $\ff _i (t) \in \sH _t$, $i=1,2$, and take representative $E$-valued $(n,0)$-forms $u_1, u_2$ of $\ff _1 , \ff _2$ as in Theorem \ref{berndtsson-reps-thm}.  By Theorem \ref{B2-thm}, Theorem \ref{berndtsson-reps-thm} and Proposition \ref{expand-the-current-prop} we obtain 
\[
( \Theta ^{\sH} _{\sigma \bar \tau} \ff_1(t), \ff _2(t)) = \left ( \xi ^{\theta} _{\sigma} \lrcorner \dbar u_1, \xi ^{\theta} _{\tau} \lrcorner \dbar u_2 \right) + \left < p_* \left ( \ii ^{n^2} \left < (\Theta (h) u_1) \wedge \bar u_2 , h \right > \right )  , \sigma \wedge \bar \tau \right >.
\]
The first term on the right hand side of this formula defines a Nakano-non-negative form on $\sH_t$, and the second defines a $k$-positive (respectively $k$-non-negative) form if and only if $h$ has $k$-positive (respectively $k$-non-negative) curvature.  The proof is complete.
\end{proof}

\end{document}